\RequirePackage{fix-cm}
\documentclass[smallcondensed,envcountsect]{svjour3}     
\smartqed  
\usepackage{amsmath}
\usepackage{graphicx}
\usepackage{subfig}
\usepackage{amsfonts}
\usepackage[colorlinks=true,citecolor={blue},linkcolor=blue]{hyperref}
\usepackage{booktabs,cases}
\usepackage{enumitem,cite}
\usepackage{boxedminipage}

\usepackage[ruled,noline,linesnumbered]{algorithm2e}
\usepackage{algorithmic}

\newcommand{\Rbb}{\mathbb{R}}
\newcommand{\X}{\mathcal{X}}
\newcommand{\Y}{\cal Y}
\newcommand{\dist}{\mathrm{dist}}
\newcommand{\nn}{\nonumber}
\newcommand{\iprod}[2]{\left\langle #1, #2 \right \rangle}
\newcommand{\tran}{\mathsf{T}}
\newcommand{\be}{\begin{equation}}
\newcommand{\ee}{\end{equation}}
\newcommand{\prox}{\textnormal{prox}}
\newcommand{\argmin}{\mathop{\mathrm{argmin}}}
\newcommand{\diag}{\mathop{\mathrm{diag}}}
\newcommand{\mdiv}{\mathrm{div}}
\newcommand{\mvec}{\mathrm{vec}}
\newcommand{\mst}{\mathrm{s.t.}}
\newcommand{\mbf}{\mathbf{m}}
\newcommand{\Rcal}{\mathcal{R}}
\newcommand{\Lcal}{\mathcal{L}}
\newcommand{\K}{K}
\newcommand{\proj}{\mathsf{Proj}}
\newcommand{\1}{\mathbf{e}}
\newcommand{\Fsf}{\mathsf{F}}

\numberwithin{equation}{section}

\begin{document}

\title{Understanding the convergence of the preconditioned PDHG method:  a view of  indefinite proximal ADMM
\thanks{Xingju Cai is supported by the NSFC grants 12131004 and 11871279. Bo Jiang is supported by the NSFC grant 11971239 and the Natural Science Foundation of the Higher Education Institutions
of Jiangsu Province (21KJA110002). Deren Han is supported by
the NSFC grants 2021YFA1003600, 12126603.}
}

\titlerunning{Understanding PrePDHG: a view of  indefinite proximal ADMM}        

\author{Yumin Ma        \and
        Xingju Cai \and 
        Bo Jiang \and 
        Deren Han 
}

\authorrunning{Y. Ma, X. Cai, B. Jiang \& D. Han} 

\institute{Yumin Ma \at
             School of Applied Mathematics, Nanjing University of Finance and Economics, Nanjing, 210023, P.R. China. \\
              \email{mayumin@nufe.edu.cn}           
           \and
           Xingju Cai  \at
             School of Mathematical Sciences, Key Laboratory for NSLSCS of Jiangsu Province,
Nanjing Normal University, Nanjing 210023, P.R. China. \\ 
\email{caixingju@njnu.edu.cn} 
\and 
Bo Jiang \at 
School of Mathematical Sciences, Key Laboratory for NSLSCS of Jiangsu Province,
Nanjing Normal University, Nanjing 210023, P.R. China.\\ 
\email{jiangbo@njnu.edu.cn} 
\and 
Deren Han (Corresponding author) \at 
LMIB, School of Mathematical Sciences, Beihang University, Beijing 100191, P.R. China.\\ 
\email{handr@buaa.edu.cn}
}

\date{Received: date / Accepted: date}

\maketitle

\begin{abstract}
The primal-dual hybrid gradient (PDHG) algorithm is popular in solving min-max problems which are being widely used in a variety of areas. To improve the applicability and efficiency of PDHG for different application scenarios, we focus on the preconditioned PDHG (PrePDHG) algorithm, which is a framework covering PDHG,  alternating direction method of multipliers (ADMM), and other methods. We give the optimal convergence condition of PrePDHG in the sense that the key parameters in the condition can not be further improved, which fills the theoretical gap in the-state-of-art convergence results of PrePDHG, and obtain the ergodic and non-ergodic sublinear convergence rates of PrePDHG. The theoretical analysis is achieved by establishing the equivalence between PrePDHG and indefinite proximal ADMM. Besides, we discuss various choices of the proximal matrices in PrePDHG and derive some interesting results. For example, the convergence condition of diagonal PrePDHG is improved to be tight, the dual stepsize of the balanced augmented Lagrangian method can be enlarged to $4/3$ from $1$, and a balanced augmented Lagrangian method with symmetric Gauss-Seidel iterations is also explored. Numerical results on the matrix game, projection onto the Birkhoff polytope, earth mover's distance, and CT reconstruction verify the effectiveness and superiority of PrePDHG.


\end{abstract}

\section{Introduction}
 In this paper, we consider the convex-concave min-max problem:
 \begin{equation}\label{PD}\tag{PD}
\min_{x\in \Rbb^n} \max_{y\in \Rbb^m}\; \Lcal(x,y):= f(x)+ \iprod{\K x}{y} -g^*(y),
\end{equation}
where~$\K \in \Rbb^{m\times n},$ $f: \Rbb^n \to (-\infty, +\infty]$, and $g: \Rbb^m \to (-\infty, +\infty]$ are proper closed convex functions, $g^*$ is the convex conjugate of $g$, i.e., $g^*(y) = \sup_{z\in \Rbb^{m}}\{\iprod{z}{y}-g(z)\}.$ Here $\iprod{\cdot}{\cdot}$ denotes the standard inner product.
The primal and dual formulations of problem \eqref{PD} are, respectively,  given as
\begin{equation}\label{P}\tag{P}
\min\limits_{x \in \Rbb^n}  f(x)+g(\K x) \\
\end{equation}
and
\begin{equation}\label{D}\tag{D}
\min\limits_{y\in\Rbb^m} \  f^*(-\K^{\tran}y)+g^*(y).
\end{equation}
Such problems have wide applications in matrix completion \cite{cai2010singular}, image denoising \cite{chambolle2011first,rudin1992nonlinear}, compressed sensing \cite{haupt2008compressed}, earth mover's distance \cite{{li2018parallel}}, computer vision \cite{pock2009algorithm}, CT reconstruction \cite{sidky2012convex}, magnetic resonance imaging \cite{valkonen2014primal},  robust face recognition \cite{yang2010review} and image restoration \cite{zhu2008efficient},  etc.

An efficient method to solve \eqref{PD} is the primal-dual hybrid gradient (PDHG) algorithm which was originally proposed by Zhu and Chan \cite{zhu2008efficient} and further developed by Chambolle and Pock \cite{chambolle2011first}. The recursion of the PDHG for \eqref{PD} reads as:

\vspace{4pt}
\hspace{-6mm}
\begin{boxedminipage}{1\textwidth}
\noindent \textbf{PDHG procedure for \eqref{PD}:} Let $\tau > 0$ and $\sigma > 0$. For given $(x^k, y^k)$, the new iterate $(x^{k+1},y^{k+1})$ is generated by:
\begin{subnumcases}{\label{alg:pdhg}}
x^{k+1}=\argmin_{x \in \Rbb^{n}}\; f(x)+ \iprod{\K x}{y^k} + \frac{1}{2\tau}\left\|x-x^k\right\|^2 \label{pre-pdhg:x},\\
y^{k+1}=\argmin_{y\in \Rbb^{m}}\; g^*(y)  - \iprod{\K (2x^{k+1} - x^k)}{y} +\frac{1}{2\sigma}\|y-y^k\|^2 \label{pre-pdhg:y}.
\end{subnumcases}
\end{boxedminipage}
\vspace{4pt}

\noindent Here, $\|\cdot\|$ means the vector $\ell_2$ norm.
In \eqref{alg:pdhg}, $\tau, \sigma >0$ are the primal and dual stepsize parameters, respectively.
Chambolle and Pock \cite{chambolle2011first} and He and Yuan \cite{he2012convergence} established the convergence of PDHG under the condition $\tau \sigma \|\K\|^{2}< 1$, in which $\|\K\|$ is the spectral norm of the matrix $\K$.  This condition is improved to  $\tau \sigma \|\K\|^{2} \leq 1$ by Condat \cite{condat2013primal} and further enhanced to
 \be \label{condition:43}
 \tau \sigma \|\K\|^{2}< \frac43
 \ee very recently by He et al. \cite{he2022generalized} for a special case of \eqref{PD}, i.e., $g^*(y)=\iprod{b}{y}$ and $b\in \Rbb^m$, other than general $g^{*}(\cdot)$. Under the condition \eqref{condition:43}, the convergence of PDHG \eqref{alg:pdhg} for \eqref{PD} with general $g^{*}(\cdot)$ is established in \cite{jiang2022solving} and \cite{li2022improved}.  For more results about the convergence of PDHG, readers can refer to \cite{cai2013improved, jiang2021approximate,jiang2021first,he2014convergence}.

As observed in \cite{pock2011diagonal} that for cases when $\|\K\|$ may not be estimated easily, or it might be
very large, the practical convergence of the PDHG \eqref{alg:pdhg} significantly slows down. To overcome this issue, we are concerned in this paper with a general algorithm, i.e., the preconditioned PDHG (PrePDHG), which is given as:

\vspace{4pt}
\hspace{-6mm}
\begin{boxedminipage}{1\textwidth}
\noindent \textbf{PrePDHG procedure for \eqref{PD}:} Let $M_1 \in \Rbb^{n \times n}$ and $M_2 \in \Rbb^{m \times m}$ be symmetric matrices. For given $(x^k, y^k)$, the new iterate $(x^{k+1},y^{k+1})$ is generated by:
\begin{subnumcases}{\label{pre-pdhg}}
x^{k+1}=\argmin_{x\in\Rbb^{n}}\; f(x)+ \iprod{\K x}{y^k} + \frac{1}{2}\left\|x-x^k\right\|_{M_1}^2 \label{pre-pdhg:x},\\
y^{k+1}=\argmin_{y\in\Rbb^{m}}\; g^*(y)  - \iprod{\K(2x^{k+1} - x^k)}{y} +\frac{1}{2}\|y-y^k\|_{M_2}^2 \label{pre-pdhg:y}.
\end{subnumcases}
\end{boxedminipage}
\vspace{4pt}

\noindent Here, $\|z\|_{M_1}^2 = \iprod{z}{M_1z}$ for a vector $z \in \Rbb^n$. Obviously, the PrePDHG \eqref{pre-pdhg} reduces to the PDHG \eqref{alg:pdhg} by taking $M_1=\tau^{-1} I_n$ and $M_2=\sigma^{-1} I_m$. More importantly, by taking other specific forms of $M_1$ and $M_2$, the framework of PrePDHG can take several other algorithms as special cases, see Section \ref{S4} for more details.

The PrePDHG is first proposed by Pock and Chambolle \cite{pock2011diagonal}\footnote{The setting in \cite{pock2011diagonal} is for the general finite-dimensional vector space other than the Euclidean space. For simplicity of presentation, we focus on the Euclidean space. However, our results in this paper can be easily extended to the general finite-dimensional vector space.}.
They established the convergence of the PrePDHG \eqref{pre-pdhg} under the condition
\begin{equation} \label{condition:PC0}
\begin{pmatrix}
M_1 & -\K^{\tran} \\
-\K & M_2
\end{pmatrix}
\succ 0,
\end{equation}
see \cite[Theorem 1]{pock2011diagonal}. For a symmetric positive matrix $M$,  denote $M^{-\frac12}$ as the square root of $M^{-1}$, namely, $M^{-1/2} M^{-1/2} = M^{-1}$,
then condition  \eqref{condition:PC0} is equivalent to (see Lemma \ref{equi14})
\begin{equation}\label{condition:PC1}
M_{1} \succ 0,\quad  M_{2}\succ 0,\quad \left\|M_2^{-\frac12}\K M_1^{-\frac12}\right\|<1.
\end{equation}
Besides, \cite{pock2011diagonal}  also proposed a family of diagonal preconditioners for $M_{1}$ and $M_{2}$, which make the subproblems easier to solve and guarantee the convergence of the algorithm. From the point view of an indefinite proximal point algorithm,  Jiang et al. \cite{jiang2021indefinite} showed that the condition  \eqref{condition:PC1} can be improved to
\[
M_{1} + \frac12 \Sigma_{f} \succ 0,~  M_{2} + \frac12 \Sigma_{g^{*}}\succ 0,~ \left\| \Big(M_2+ \frac12 \Sigma_{g^{*}}\Big) ^{-\frac12}\K \Big(M_1 +  \frac12 \Sigma_{f}\Big)^{-\frac12}\right\|<1, \nn
\]
where $\Sigma_{f}$ and $\Sigma_{g^{*}}$ are symmetric semidefinite matrices related to $f$ and $g^{*}$ (see \eqref{cm}).

Since min-max problems are equivalent to constrained or composite optimization problems under certain conditions, some literatures focus on understanding PDHG and PrePDHG from various perspectives.
For example, the equivalence between PDHG and linearized  alternating direction method of multipliers (ADMM) is discussed in \cite{esser2010general,o2020equivalence}.
Similarly,  Liu et al. \cite{liu2021acceleration} established the equivalence between PrePDHG for \eqref{PD} and positive semidefinite proximal ADMM (sPADMM) for an equivalent problem of \eqref{D}.
Based on the equivalence and the convergence analysis of the first-order primal-dual algorithm in \cite{chambolle2016ergodic}, Liu et al. \cite{liu2021acceleration} established  the ergodic convergence result (but without sequence convergence) of  PrePDHG under the  condition
\begin{equation} \label{condition:LiuYin}
\begin{pmatrix}
M_1 & -\K^{\tran} \\
-\K & M_2
\end{pmatrix}
\succeq 0,
\end{equation}
and also considered some inexact versions of PrePDHG.  Note that
 a similar condition of  \eqref{condition:LiuYin} is extended for infinite dimensional Hilbert space in \cite{briceno2021split}.
Very recently, under condition \eqref{condition:LiuYin}, Jiang and Vandenberghe \cite{jiang2022bregman} showed convergence of iterates for Bregman PDHG, of which PrePDHG is a special case.

As mentioned above, when $M_1=\tau^{-1}I_n$ and $M_2=\sigma^{-1}I_m$, the PrePDHG \eqref{pre-pdhg} reduces to the original PDHG \eqref{alg:pdhg}. However, the convergence condition \eqref{condition:LiuYin}  degrades into $\tau \sigma\|\K\|^{2}\leq 1 $ other than \eqref{condition:43}. This raises a natural question: {\it can we obtain a tighter convergence condition of PrePDHG to fill this gap?}

Motivated by \cite{liu2021acceleration}, we intend to investigate PrePDHG from the perspective of proximal ADMM.
A known result is that indefinite proximal ADMM (iPADMM), with weaker convergence conditions, outperforms positive semidefinite proximal ADMM (sPADMM) \cite{chen2021equivalence,he2020optimally,gu2015indefinite, li2016majorized, chen2019unified, zhang2020linearly, ma2021majorized,han2022survey,cai2022developments}.  In this paper, we restudy the PrePDHG \eqref{pre-pdhg} from the point view of iPADMM other than sPADMM as done in \cite{liu2021acceleration}
and give positive answers to the above question.
The main contributions of this paper are as follows:

Firstly, we establish the equivalence between PrePDHG for \eqref{PD} and iPADMM for an equivalent problem of \eqref{P}. Based on the equivalence, we improve the convergence condition \eqref{condition:PC1} of the PrePDHG to
\begin{equation*}
\begin{pmatrix}
\frac43\left(M_1 + \frac12 \Sigma_f\right) & \K^\tran\\
\K &  M_2
\end{pmatrix}
\succ 0,
\end{equation*}
which can be rewritten as (see Lemma \ref{equi14})
\be \label{equ:M1M2:condtion}
M_{1}  + \frac12 \Sigma_{f}\succ 0,\quad  M_{2}\succ 0,\quad \left\|M_2^{-\frac12}\K \Big(M_1 + \frac12 \Sigma_{f}\Big)^{-\frac12}\right\|^{2}< \frac43.
\ee
Note that \eqref{equ:M1M2:condtion} is exactly \eqref{condition:43} when PrePDHG reduces to the original PDHG and $\Sigma_f$ is taken as a zero matrix. Some counter-examples are given in Section \ref{subsection:tight} to illustrate that condition \eqref{equ:M1M2:condtion} is tight in the sense that the constants 4/3 and 1/2 can not be replaced by any  larger numbers, namely, the inequality sign ``$<$'' can not be replaced by ``$\leq$''.

Secondly, we establish the ergodic and non-ergodic sublinear convergence rate results of the PrePDHG both in the sense of the KKT residual and the function value residual.  To the best of our knowledge, the sublinear convergence rate based on the KKT residual is new for PDHG-like methods since the existing results mainly focus on the function value residual. And for the function value residual measurement, our sublinear rate result is the first non-ergodic result since the existing results are all ergodic. The numerical experiments in Section \ref{S5} show that the KKT residual is more practical than the function value residual.

Thirdly, we discuss some practical choices of $M_{1}$ and $M_{2}$ and get some interesting results. For example, condition \eqref{condition:43} is tight for PDHG \eqref{alg:pdhg}; the sharp range of parameters for diagonal PrePDHG is given, and the dual stepsize of the balanced ALM (BALM) \cite{he2021balanced} can be enlarged to 4/3 from 1, and we rename it an enhanced BALM (eBALM). Besides, we explore the eBALM with symmetric Gauss-Seidel iterations (eBALM-sGS), which can be understood as a special case of PrePDHG.

Finally, we perform four groups of numerical experiments on solving the matrix game, projection onto the Birkhoff polytope, earth mover's distance, and CT reconstruction problems. We choose proper $M_{1}$ and $M_{2}$ and
 the numerical results verify the effectiveness of the choices of $M_{1}$ and $M_{2}$ and the superiority of the PrePDHG (with tighter convergence condition).

This paper is organized as follows. Some notations and preliminaries are presented in Section \ref{S2}. In Section \ref{section:vmpdhg}, we first establish the equivalence between PrePDHG and iPADMM and then develop the global convergence of PrePDHG from the iPADMM point of view. The existing convergence condition of PrePDHG is improved to be tight, as shown by counter-examples. Then, the sublinear convergence rate of the PrePDHG is obtained. We revisit the choices of $M_1$ and $M_2$ in Section \ref{S4} and get some new results. In Section \ref{S5}, we perform numerical experiments on four practical problems to verify the effectiveness of the PrePDHG. Some concluding remarks are made in Section \ref{S6}.

\section{Notations and Preliminaries}\label{S2}
 We use $\|x\|_1,\|x\|$, and $\|x\|_\infty$ to denote the $\ell_1, \ell_2$ and $\ell_\infty$ norm of the vector $x$ respectively, and $\|A\|$ to denote the spectral norm of the matrix $A$. We use $\mvec(A)$ to denote a vector formulated by stacking the columns of $A$ one by one, from first to last.  We slightly
abuse the notation $\|x\|_M^2:= \iprod{x}{Mx}$ as long as $M$ is symmetric. When $M$ is symmetric positive semidefinite, we use $M^{\frac12}$ to represent the square root of $M$, namely, $M^{\frac12} M^{\frac12} = M$.
 For symmetric matrices $A$ and $B$, $A \succeq (\succ)\ B$ means that $A - B$ is positive semidefinite (positive definite). For a symmetric matrix $P \in \Rbb^{n \times n}$,  we can always decompose it as
\begin{equation}P = P_{+} - P_{-} \nn
\end{equation}  with $P_{+}, P_{-} \succeq 0$. We name this decomposition a {\it DC  decomposition} of $P$.  Note that the DC decomposition of a symmetric matrix is not unique.

We adopt some standard notations in convex analysis; see \cite{rockafellar2015convex} for instance. The distance from a point $x$ to a nonempty convex closed set $\mathcal{S}\subseteq \Rbb^{n}$ is denoted as $\dist(x,\mathcal{S}) = \min_{y\in \mathcal{S}}\,\|y - x\|$.  For any proper closed convex function $f: \Rbb^n \to (-\infty, +\infty]$ and $\bar x \in \mathrm{dom} f:= \{x \in \Rbb^n \mid f(x) <+ \infty \}$,  the subdifferential at $\bar x$ is defined as $\partial f(\bar x) := \{\xi\in\Rbb^n \mid f(x)\geq f(\bar x)+\iprod{\xi}{x-\bar x},\;\forall x \in \Rbb^{n}\}$, in which any $\xi$ is a subgradient at $\bar x$. Moreover, there exists a symmetric positive semidefinite matrix $\Sigma_f$ such that for all $x_1, x_2 \in \Rbb^n$ and $\xi_1 \in \partial f(x_1), \xi_2 \in \partial f(x_2)$,
\begin{equation}\label{cm}
\iprod{\xi_1 - \xi_2}{x_1 - x_2} \geq \|x_1 - x_2\|_{\Sigma_f}^2.
\end{equation}
For any proper closed convex function $f$, the convex conjugate of $f$ is defined as $f^*(y) := \sup_{x\in \Rbb^{n}}\{\iprod{x}{y}-f(x)\}$, and we have
\begin{equation}\label{equ:f:closed}
y\in \partial f(x)\Leftrightarrow x \in \partial f^*(y). 
\end{equation}

Given a symmetric matrix $M$ with $M + \Sigma_f\succ 0$,   we define the generalized proximal operator as
 \begin{equation}\label{equ:proxfM}
\prox_f^M(x) := \argmin_{z\in \Rbb^n}\, f(z) + \frac12\|z - x\|_M^2.
\end{equation}
If $M = \tau^{-1} I_{n}$ for some $\tau > 0$, we simply denote $\prox_{\tau f}(x):= \prox_f^M(x)$.
Let $\tilde f(\cdot) := f(\cdot) - \frac12 \|\cdot\|_{\Sigma_f}^2$. Observing that
\[
\begin{aligned}
f(z) + \frac12\|z - x\|_M^2 ={}& \tilde f(z) + \frac12 \left\| z - (M + \Sigma_{f})^{-1} Mx\right\|^{2}_{M + \Sigma_{f}} + \frac12 \|x\|_{M}^{2} \\ 
{}& - \frac12 \|(M + \Sigma_{f})^{-1} Mx\|_{M + \Sigma_{f}}^{2},
\end{aligned}
\]
we have an equivalent characterization of $\prox_f^M(x)$ as
\begin{equation}\label{equ:f:prox:2}
\prox_f^M(x) = \prox_{\tilde f}^{M + \Sigma_f}\left( \left(M+ \Sigma_f\right)^{-1} Mx\right).
\end{equation}
We now present a generalization of Moreau's identity, see \cite[Theorem 1 (ii)]{combettes2013moreau} or \cite[Lemma 3.3]{becker2019quasi}, which is very useful in our analysis.
\begin{lemma}\label{lemma:MI}
Let $f:\Rbb^n\to (-\infty, +\infty]$ be a proper closed convex function. Suppose $M \succ 0$, then we have
\begin{equation*}
x=\prox_{f}^M(x)+M^{-1} \prox_{f^*}^{M^{-1}}(M x), \quad \forall x \in \Rbb^n.
\end{equation*}
\end{lemma}

In the following lemma, the equivalence between \eqref{condition:PC0} and \eqref{condition:PC1} is established. In \cite{pock2011diagonal}, the authors proved that  \eqref{condition:PC1} implies \eqref{condition:PC0}. Here we present a simple  proof of the equivalence based on the well-known Schur complement.

\begin{lemma}\label{equi14}
Let $M_1\in \Rbb^{n\times n}$, $M_2\in \Rbb^{m\times m}$ be symmetric matrices. Then
\eqref{condition:PC0} is equivalent to \eqref{condition:PC1}.
\end{lemma}

\begin{proof}
By \cite[Theorem 1.12]{zhang2006schur}, we know that \eqref{condition:PC0} is equivalent to
$$
M_1 \succ 0,\quad  M_2 \succ 0, \quad M_1-\K^{\tran}M_2^{-1}\K \succ 0.
$$
Since $M_1 \succ 0$ and   $M_2 \succ 0$, we have
$$
\begin{aligned}
M_1-\K^{\tran}M_2^{-1}\K \succ 0 {}& \iff I_n-M_1^{-\frac12}\K^{\tran}M_2^{-\frac12}M_2^{-\frac12}\K M_1^{-\frac12}  \succ 0 \\
{} &
 \iff \|M_2^{-\frac12}\K M_1^{-\frac12}\|<1.
 \end{aligned}
$$
The proof is completed.
\qed \end{proof}

Throughout this paper, we assume that problem \eqref{PD} has a  saddle point $(x^\star, y^\star)$, which satisfies the optimality condition
\begin{equation}\label{equ:PD:optimality}
\Lcal(x^{\star}, y) \leq \Lcal(x^{\star},y^{\star})\leq \Lcal(x, y^{\star}), \quad \forall x \in \Rbb^{n},\quad  \forall y \in \Rbb^{m}
\end{equation}
and the KKT-type optimality condition
 \be \label{equ:opt:condition}
0 \in \partial f(x^{\star}) + \K^{\tran} y^{\star},\quad   0 \in \partial g^{*}(y^{\star}) - \K x^{\star}.
\ee
Such $x^{\star}$ and $y^{\star}$ are also optimal for \eqref{P} and \eqref{D}, respectively.
Define the KKT residual mapping $\Rcal: \Rbb^n \times \Rbb^m \rightarrow \Rbb$ as
\be \label{equ:kkt:residual}
\Rcal(x,y) =\max\left\{ \dist(0, \partial f(x) + \K^{\tran}y), \dist(0, \partial g^{*}(y) - \K x)\right\}.
\ee
Clearly, we have the following equivalent characterization of the optimality condition.
\begin{proposition}
 The KKT-type optimality condition \eqref{equ:opt:condition} holds if and only if    $\Rcal(x^\star,y^\star) = 0$.
 \end{proposition}
Based on this, we define the $\epsilon$-solution of problem \eqref{PD} as follows.
 \begin{definition}\label{def:epsilon:solution}
 Given  $\epsilon \geq 0$, a pair $(x,y)$ is called an $\epsilon$-solution of problem \eqref{PD} if $\Rcal(x,y) \leq \epsilon$.
 \end{definition}

Note that the KKT residual \eqref{equ:kkt:residual} may be difficult or expensive to calculate since it involves computing the distance of a point to a convex set.
However, in some practical circumstances, the upper bound of  $\Rcal(x,y)$ in \eqref{equ:kkt:residual} could be easily obtained; see the discussion in
Remark \ref{remark:upperbound:R} and Remark \ref{remark:R:upperbound}.

In the rest of this section, we present the existing convergence and sublinear convergence rate results of iPADMM developed in \cite{gu2015indefinite}, which are key to the convergence analysis of PrePDHG. Note that the algorithm in \cite{gu2015indefinite} is more general and takes iPADMM as a special case. Here we display the corresponding results of iPADMM.

Consider the convex minimization problem with linear constraints and a separable objective function
\begin{equation}\label{gup}
\begin{array}{cl}
\min\limits_{x\in \Rbb^{n_1},y\in \Rbb^{n_2}} & \theta_1(x) +  \theta_2(y)\\
\mathrm{s.t.} & Ax+By=0,
\end{array}
\end{equation}
where $A\in \Rbb^{m\times n_1}, B\in \Rbb^{m\times n_2}$, $\theta_1:\Rbb^{n_1}\to (-\infty, +\infty]$, and $\theta_2:\Rbb^{n_2}\to (-\infty, +\infty]$ are proper closed convex functions. The augmented Lagrangian function of \eqref{gup} is defined by:
$$
\Lcal_\beta (x,y,\lambda) = \theta_1(x) +  \theta_2(y)-\iprod{\lambda}{Ax+By}+\frac{\beta}{2}\|Ax+By\|^2,
$$
where $\lambda$ is the corresponding Lagrange multiplier of the linear constraints and $\beta > 0$ is a penalty parameter. The iPADMM for \eqref{gup} in \cite{gu2015indefinite} is given as:

\vspace{4pt}
\hspace{-6mm}
\begin{boxedminipage}{1\textwidth}
\indent \textbf{iPADMM procedure for \eqref{gup}:}
\noindent
Choose the symmetric indefinite matrices $S$ and $T$. For given $(x^k, y^k, \lambda^k)$, the new iterate $(x^{k+1}, y^{k+1}, \lambda^{k+1})$ is generated by:
\begin{subnumcases}{\label{ipspr}}
x^{k+1} = \argmin_{x\in \X} \mathcal{L}_{\beta}(x,y^k,\lambda^k)+\frac 12 \|x-x^k\|_S^2,\label{ipspr-1}\\
y^{k+1} = \argmin_{y\in \Y} \mathcal{L}_{\beta}(x^{k+1},y,\lambda^k)  + \frac12\|y - y^k\|_T^2, \label{ipspr-2}\\
\lambda^{k+1} = \lambda^k -\beta (A x^{k+1} + By^{k+1}). \label{ipspr-3}
\end{subnumcases}
\end{boxedminipage}
\vspace{4pt}

Let $\Sigma_1$ and $\Sigma_2$ be the symmetric positive semidefinite matrices related to  $\theta_1$ and $\theta_2$, respectively; see \eqref{cm} for details.
 The sequence $\{(x^k, y^k, \lambda^k)\}$ is denoted as $\{w^k\}$. Now we present the convergence results of iPADMM.
\begin{lemma}\cite[Theorem 3.2]{gu2015indefinite}\label{guth1}
Let the sequence $\{w^k\}$ be generated by iPADMM \eqref{ipspr}. If the proximal terms $S$ and $T$ are chosen such that
\be\label{gucon1}
S+\frac 12 \Sigma_1 \succeq 0, \quad S+\frac 12 \Sigma_1 +\beta A^{\tran}A \succ 0
\ee
and
\be\label{gucon2}
T+\Sigma_2 +\beta B^{\tran}B \succ 0,\quad T+\frac 12 \Sigma_2 +\kappa_1(-2T_-+\Sigma_2)+\kappa_2 \beta B^{\tran}B \succ 0,
\ee
where $\kappa_1=1, \kappa_2 \in (0, \frac34)$,  and $T_-$ comes from one DC decomposition of $T$,
then $\{w^k\}$ converges to an optimal solution of \eqref{gup}.
\end{lemma}
\begin{lemma} \cite[Theorem 4.1]{gu2015indefinite} \label{guth2}
Let the sequence $\{w^k\}$ be generated by iPADMM \eqref{ipspr}. If the proximal terms $S$ and $T$ are chosen such that \eqref{gucon1} and \eqref{gucon2} hold, and
\be\label{gucon3}
S+\frac 12 \Sigma_1 \succeq \frac c2 \Sigma_1,
\ee
with $c>0$,
 then we have
$$
\min\limits_{1\leq i\leq k}\|w^i-w^{i+1}\|_{\hat G}^2 = o(1/k),
$$
in which
\begin{equation*}
\hat G=\begin{pmatrix}S+\Sigma_1 &&\\ & T+\Sigma_2+\beta B^{\tran}B & \\ && \frac{1}{\beta}I_m \end{pmatrix}.
\end{equation*}
\end{lemma}

\begin{lemma} \cite[Theorem 4.2]{gu2015indefinite} \label{guth3}
Let the sequence $\{w^k\}$ be generated by iPADMM \eqref{ipspr}. If the proximal terms $S$ and $T$ are chosen such that  \eqref{gucon1}, \eqref{gucon2}, and \eqref{gucon3} hold, and
$
T+\frac 12 \Sigma_2 \succ 0,
$
then we have
$$
\|w^k-w^{k+1}\|_{\hat G}^2 = o(1/k),
$$
where $\hat G$ is  defined in Lemma \ref{guth2}.
\end{lemma}

\section{The Preconditioned PDHG and its Convergence} \label{section:vmpdhg}

We first present the PrePDHG with practical stopping criterion for convex-concave min-max optimization \eqref{PD} in Algorithm \ref{alg:vmpdhg}.  We shall first establish an equivalence between PrePDHG and iPADMM, which is key to analyzing the algorithm,  in Section \ref{subsection:equivalence} and deduce the global convergence of Algorithm \ref{alg:vmpdhg}  in Section \ref{subsection:convergence}. Section \ref{subsection:tight} provides counter-examples to show the tightness of condition \eqref{equ:M1M2:condtion}.
The sublinear convergence rate in both ergodic and non-ergodic sense is investigated in Section \ref{subsection:sublinear}.

The PrePDHG is given in Algorithm \ref{alg:vmpdhg}. Note that the stopping criterion $\Rcal(x^{k+1},y^{k+1}) \leq \epsilon$ can be replaced by $\Rcal(x^{k+1},y^{k}) \leq \epsilon$.
 \begin{algorithm}[!h]
\caption{PrePDHG: Preconditioned PDHG for solving (\ref{PD}).}
\label{alg:vmpdhg}
Initialization: Choose the   initial points $x^{0} \in \Rbb^{n}$, $y^{0} \in \Rbb^{m}$, and   set the tolerance $\epsilon \geq 0$.  Choose
  the matrices $M_{1} \in \Rbb^{n \times n}$ and $M_{2} \in \Rbb^{m \times m}$ satisfying  \eqref{equ:M1M2:condtion}.

\For{$k = 0, 1, \ldots, $}{
Update  $x^{k+1}$ and $y^{k+1}$ as follows:
\begin{subnumcases}{\label{vmpdhg}}
x^{k+1}=\argmin_{x\in\Rbb^{n}}\; f(x)+ \iprod{\K x}{y^k} + \frac{1}{2}\left\|x-x^k\right\|_{M_1}^2 \label{pre-pdhg:x},\\
y^{k+1}=\argmin_{y\in\Rbb^{m}}\; g^*(y)  - \iprod{\K(2x^{k+1} - x^k)}{y} +\frac{1}{2}\|y-y^k\|_{M_2}^2 \label{pre-pdhg:y},
\end{subnumcases}

\If{$\Rcal(x^{k+1},y^{k+1}) \leq \epsilon$}{
break.}
}
\end{algorithm}

\begin{remark} \label{remark:upperbound:R}
By \eqref{equ:thm:subconvergence:dist0} and \eqref{equ:thm:subconvergence:dist1},   we have $\Rcal(x^{k+1},y^{k+1}) \leq \hat \Rcal(x^{k+1}, y^{k+1})$ with
 $$
 \begin{aligned}
 \hat \Rcal(x^{k+1},y^{k+1}): = \max\big\{&\left\|\K^{\tran} (y^{k+1} - y^{k})-M_{1}(x^{k+1} - x^{k})\right\|,\\
& \left\|\K(x^{k+1} - x^{k}) - M_{2}(y^{k+1} - y^{k})\right\|\big\},
\end{aligned}$$
which can be easily computed. Therefore, if $\Rcal(x^{k+1},y^{k+1})$ is difficult to compute,  we can use the stopping criterion $\hat \Rcal(x^{k+1},y^{k+1}) \leq \epsilon$.  Similarly,  by  the first inequalities in \eqref{equ:thm:subconvergence:dist1:2}  and \eqref{equ:thm:subconvergence:dist1:3}, we can also replace $\Rcal(x^{k+1},y^{k})$ by its upper bound  as
$$
\begin{aligned}
\hat
 \Rcal(x^{k+1},y^{k})  := 
 \max\big\{&\left\|M_1(x^{k+1} - x^{k})\right\|, \\ 
  {}&\left\|\K(x^{k} - x^{k-1}) + \K(x^{k} - x^{k+1}) - M_{2}(y^{k} - y^{k-1}) \right\|\big\}.
\end{aligned}
$$
Note that for some special cases, such as $g^
*$ is a linear function, a more compact upper bound of  $\Rcal(x^{k+1},y^{k+1})$ or $\Rcal(x^{k+1},y^{k})$ can be obtained,  see Remark \ref{remark:R:upperbound} for instance.
\end{remark}

\subsection{Equivalence of PrePDHG and iPADMM} \label{subsection:equivalence}
We first show that PrePDHG \eqref{vmpdhg} can be understood as an iPADMM applied on  the equivalent formulation of problem \eqref{P}:
\begin{equation}\label{P:1}\tag{P1}
\begin{array}{cl}
\min\limits_{x \in \Rbb^{n},\;u\in \Rbb^{m}} & g(u) +  f(x)\\
\mathrm{s.t.} & M_2^{-\frac12}(\K x-u)=0,
\end{array}
\end{equation}
where $M_2 \succ 0$.
Let
$$\mathcal{L}_1(u,x,\lambda) = g(u) + f(x) + \iprod{\lambda}{M_2^{-\frac12}(\K x-u)} + \frac12 \|\K x-u\|_{M_2^{-1}}^2$$
be the augmented Lagrangian function of problem \eqref{P:1}, where $\lambda$ is the corresponding Lagrange multiplier of the linear constraints.
Given the initial points $x^0 \in \Rbb^{n}$ and $\lambda^{0} \in \Rbb^{m}$, the main iterations of the iPADMM are given as
\begin{subnumcases}{\label{ipadmm}}
u^{k+1} = \argmin_{u\in \Rbb^m} \mathcal{L}_1(u,x^k,\lambda^k),\label{equ:ipadmm:a:0}\\
x^{k+1} = \argmin_{x\in \Rbb^n} \mathcal{L}_1(u^{k+1},x,\lambda^k)  + \frac12\|x - x^k\|_{M_1 - \K^{\tran}M_2^{-1}\K}^2, \label{equ:ipadmm:b:0}\\
\lambda^{k+1} = \lambda^k + M_2^{-\frac12} (\K x^{k+1} - u^{k+1}), \label{equ:ipadmm:c:0}
\end{subnumcases}
where the proximal matrix $M_1 - \K^{\tran}M_2^{-1}\K$ could be indefinite.
Note that in \eqref{ipadmm}, there is only an additional proximal term in the second subproblem. Using the notations of \eqref{equ:proxfM} and \eqref{equ:f:prox:2},  we can equivalently formulate \eqref{ipadmm} as
\begin{subnumcases}{\label{proxadmm}}
u^{k+1} =  \prox_{g}^{M_2^{-1}}\Big(M_2^{\frac12}\lambda^k + \K x^k\Big),\label{equ:proxadmm:a}\\
x^{k+1} =   \prox_{\tilde f}^{M_1 + \Sigma_{f}}\Big( (M_{1} + \Sigma_{f})^{-1} M_{1} x^k - \\
  \quad \quad \quad\quad\quad\quad \quad\quad \quad (M_{1} + \Sigma_{f})^{-1} \K^{\tran}M_2^{-1}\big(M_2^{\frac12} \lambda^k + \K x^k - u^{k+1}\big)\!  \Big), \label{equ:proxadmm:b}\\
\lambda^{k+1} = \lambda^k + M_2^{-\frac12} (\K x^{k+1} - u^{k+1}), \label{equ:proxadmm:c}
\end{subnumcases}
where $\tilde f(\cdot) := f(\cdot) - \frac12 \|\cdot\|_{\Sigma_f}^2$ is defined in Section \ref{S2}.
Similarly, we can reformulate the iterations of PrePDHG \eqref{vmpdhg} as
\begin{subnumcases}{\label{vmpdhg:2}}
x^{k+1} = \prox_{\tilde f}^{M_1 + \Sigma_{f}}\left( (M_{1} + \Sigma_{f})^{-1} M_{1} x^k - (M_1+\Sigma_{1})^{-1} \K^{\tran}y^k \right), \label{vmpdhg:2:a}\\
y^{k+1} = \prox_{g^*}^{M_2} \left(y^k + M_2^{-1} \K(2x^{k+1} - x^k)\right). \label{vmpdhg:2:b}
\end{subnumcases}

We are now ready to  deduce the equivalence between PrePDHG \eqref{vmpdhg} and  iPADMM \eqref{ipadmm}.

\begin{lemma}\label{lemma:equiv}
PrePDHG \eqref{vmpdhg} (or \eqref{vmpdhg:2}) and iPADMM \eqref{ipadmm} (or \eqref{proxadmm}) are equivalent in the sense that the sequence generated by either algorithm can explicitly recover the sequence generated by the other.
\end{lemma}

\begin{proof}
Let  the sequence $\{(u^k, x^k, \lambda^k)\}$ be generated by iPADMM  \eqref{proxadmm} with initial points $x^{0}\in \Rbb^{n}$ and $\lambda^{0} \in \Rbb^{m}$. Consider the transform
\be\label{equ:transform:y}
y^k = M_2^{-1}\Big(M_2^{\frac12} \lambda^k + \K x^k - u^{k+1}\Big).
\ee
First,  substituting  \eqref{equ:transform:y} into  \eqref{equ:proxadmm:b} yields \eqref{vmpdhg:2:a}.
By Lemma \ref{lemma:MI}, we have from \eqref{equ:proxadmm:a}  that
$$M_2^{\frac12}\lambda^k + \K x^k   =    u^{k+1}+  M_2 \prox_{g^*}^{M_2} \Big(M_2^{-\frac12}\lambda^k + M_2^{-1}\K x^k   \Big),$$
which with the transform  \eqref{equ:transform:y} implies
$
 y^k  =  \prox_{g^*}^{M_2} \Big(M_2^{-\frac12}\lambda^k + M_2^{-1}\K x^k\Big).$
This also tells
\be \label{equ:equiv:y:b1}
y^{k+1} = \prox_{g^*}^{M_2} \Big(M_2^{-\frac12}\lambda^{k+1} + M_2^{-1}\K x^{k+1}   \Big).
\ee
Besides, with \eqref{equ:proxadmm:c} and \eqref{equ:transform:y}, we have
$M_2^{-\frac12} \lambda^{k+1}  = M_2^{-1} \K(x^{k+1} - x^k) + y^k $. Substituting this relation into \eqref{equ:equiv:y:b1} yields \eqref{vmpdhg:2:b}.
Now we can conclude that the sequence $\{(x^{k}, y^k)\}$ is exactly the sequence generated by PrePDHG \eqref{vmpdhg:2} with initial points $x^{0}$ and $y^0 = M_2^{-1}(M_2^{\frac12} \lambda^0 + \K x^0 - u^{1})$.

On the other hand, let the sequence $\{(x^k, y^k)\}$ be generated by PrePDHG \eqref{vmpdhg:2} with given initial points $x^{0} \in \Rbb^{n}$ and $y^{0} \in \Rbb^{m}$. Consider the transforms
$$
\lambda^{k+1} = M_2^{-\frac12}\K(x^{k+1} - x^k) + M_2^{\frac12}y^k, \quad u^{k+1} = M_2^{\frac12} \lambda^k + \K x^k - M_2 y^k.
$$
Using a similar argument, we can show that $\{(u^k, x^k, \lambda^k)\}$ is exactly the same sequence generated by  iPADMM \eqref{proxadmm} and the initial points of $x$ and $\lambda$ are taken as $x^0$ and $M_2^{-\frac12}\K(x^{1} - x^0) + M_2^{\frac12}y^0$, respectively. We omit the details for brevity.  The proof is completed.
\qed
\qed \end{proof}

\begin{remark}
If $M_1=\K^{\tran} M_2^{-1} \K$, then iPADMM \eqref{ipadmm} reduces to the classical ADMM \cite{glowinski1975approximation,gabay1976dual}.  In this case,  PrePDHG \eqref{vmpdhg} is equivalent to the classical ADMM.
\end{remark}

Based on the key observation that  PrePDHG and iPADMM are equivalent, we next investigate the convergence of PrePDHG \eqref{vmpdhg}, namely, Algorithm \ref{alg:vmpdhg},  via the well-established convergence results of iPADMM; see \cite{he2020optimally,gu2015indefinite, li2016majorized, chen2019unified, zhang2020linearly, ma2021majorized} for instance.  Here, we mainly use the global and sublinear convergence rate results developed in \cite{gu2015indefinite}.

It should be mentioned that Liu et al. \cite{liu2021acceleration} also showed that PrePDHG \eqref{vmpdhg} is equivalent to a proximal ADMM applied on the equivalent formulation of dual problem \eqref{D} as:
\begin{equation}\label{D:1}
\begin{array}{cl}
\min\limits_{y \in \Rbb^{m},\;v\in \Rbb^{n}} &  g^{*}(y)+f^{*}(v)\\
\mathrm{s.t.} & M_1^{-\frac12}(\K^{\tran}y+v)=0,
\end{array}
\nn
\end{equation}
where they require \eqref{condition:LiuYin} holds.  The recursion of the  proximal ADMM therein is given as
\begin{subnumcases}{\label{ipadmm:dual}}
y^{k+1} = \argmin_{y\in \Rbb^m} \mathcal{\widetilde L}_1(y,v^{k},\lambda^k) + \frac12 \|y - y^{k}\|_{M_{2} - \K M_{1}^{-1}\K^{\tran}},  \\
v^{k+1} = \argmin_{v\in \Rbb^n} \mathcal{\widetilde L}_1(y^{k+1},v,\lambda^k),  \\
\tilde \lambda^{k+1} = \tilde \lambda^k +  M_1^{-\frac12}(\K^{\tran}y^{k+1}+v^{k+1}),
\end{subnumcases}
where $$\mathcal{\widetilde L}_1(y,v,\tilde \lambda) = g^{*}(y) + f^{*}(v)+ \iprod{\lambda}{M_1^{-\frac12}(\K^{\tran}y+v)} + \frac12 \|\K^{\tran}y+v\|_{M_1^{-1}}^2,$$
in which $\tilde \lambda$ is the corresponding Lagrange multiplier of the linear constraints.  A main difference between \eqref{ipadmm} and \eqref{ipadmm:dual} lies in that the proximal term of \eqref{ipadmm} is in the second subproblem other than in the first subproblem as done by \eqref{ipadmm:dual}.  It is this key point that  makes our condition on $M_{1}$ and $M_{2}$ weaker than that in \cite{liu2021acceleration} since the iPADMM can always allow more indefiniteness of the proximal term in the second subproblem other than that in the first subproblem.

\subsection{Global Convergence} \label{subsection:convergence}
It is clear that condition \eqref{condition:LiuYin} implies $M_1-\K^{\tran} M_2^{-1} \K \succeq 0$, which further means that the proximal matrix in \eqref{equ:ipadmm:b:0} is positive semidefinite. However, the well-explored convergence results of iPADMM tell that the proximal matrix  $M_1-\K^{\tran} M_2^{-1}\K$ could be indefinite.
Therefore, we could further improve the convergence condition of   PrePDHG from the perspective of iPADMM.

\begin{lemma}\label{lemma:M1M2}
Suppose condition \eqref{equ:M1M2:condtion} holds, that is,
\begin{equation*}
M_{1}  + \frac12 \Sigma_{f}\succ 0,\quad  M_{2}\succ 0,\quad \left\|M_2^{-\frac12}\K \Big(M_1 + \frac12 \Sigma_{f}\Big)^{-\frac12}\right\|^{2}< \frac43.
\end{equation*}
Then the sequence generated by iPADMM \eqref{ipadmm} converges to an optimal solution of \eqref{P:1}.
\end{lemma}

\begin{proof}
Let the sequence $\{(u^k,x^k,\lambda^k)\}$ be generated by iPADMM \eqref{ipadmm}.
In problem \eqref{gup} and Lemma \ref{guth1}, we take $n_1 := m$, $n_2 := n$, $\theta_1 := g$, $\theta_2 := f$, $A:= M_{2}^{-1/2}$, $B:= -M_{2}^{-1/2}\K$, $\beta = 1$, $S := 0, T := M_{1} - \K^{\tran}M_2^{-1} \K$, $\Sigma_{1} := 0, \Sigma_{2} := \Sigma_{f}$, and the parameter $\kappa_{2}:= 1 - \rho$ with $\rho \in (\frac14,1)$. Then we immediately know that  $\{(u^k,x^k,\lambda^k)\}$ converges to an optimal solution of \eqref{P:1} as long as there exists a DC decomposition of $M_1 - \K^{\tran} M_2^{-1} \K$  and $\rho \in (1/4,1)$ such that
\be \label{equ:M1:M2:condition:2}
M_1 + \Sigma_f \succ0, ~ H:= M_1  + \frac32 \Sigma_f - 2 \left(M_1 - \K^{\tran} M_2^{-1} \K \right)_{-} - \rho \K^{\tran} M_2^{-1} \K \succ 0.
\ee
Now we only need to show the correctness of \eqref{equ:M1:M2:condition:2} under the condition \eqref{equ:M1M2:condtion}.

If $M_{1} - \K^{\tran} M_2^{-1} \K \succeq 0$, we can take its DC decomposition as $(M_{1} - \K^{\tran} M_2^{-1} \K)_{+} = M_{1} - \K^{\tran} M_2^{-1} \K$ and $(M_{1} - \K^{\tran} M_2^{-1} \K)_{-} = 0$. Hence, for any $\rho\in (\frac14,1)$, we have
$$H = M_{1} + \frac32 \Sigma_{f} - \rho \K^{\tran} M_2^{-1} \K = \rho(M_{1}-\K^{\tran} M_2^{-1} \K) + (1 - \rho)M_{1} + \frac32 \Sigma_{f}\succ 0,$$
where the last inequality is due to $M_{1} + \frac12 \Sigma_{f} \succ 0$ which comes from \eqref{equ:M1M2:condtion}.

Now, suppose $M_{1} - \K^{\tran} M_2^{-1} \K \not \succeq 0$. By \eqref{equ:M1M2:condtion} and the Schur complement theorem, we have
$\frac43\left(M_1+\frac12\Sigma_f\right) \succ \K^{\tran}M_2^{-1}\K,$
namely,
\begin{equation}
M_{1} -  \K^{\tran} M_2^{-1} \K + \frac13\left(M_{1} + 2\Sigma_{f}\right)\succ 0. \nn
\end{equation}
Let
\be\label{equ:lemma:M1M2:M}
M:=  \left(M_1 + 2 \Sigma_f\right)^{-\frac12} \left(M_{1} - \K^{\tran} M_2^{-1} \K\right)\left(M_1 + 2 \Sigma_f\right)^{-\frac12},
\ee
then obviously we have $ M  + \frac13 I_{n} \succ 0. $
Set  $M = U\Sigma U^{\tran}$ be the eigenvalue decomposition of $M$ with $U^{\tran}U = UU^{\tran} = I_n$ and the diagonal matrix $\Sigma = \mathrm{diag}(\sigma_1, \ldots, \sigma_n)$ with $\sigma_1 \geq \cdots \geq \sigma_{p} \geq 0 > \sigma_{p+1} \geq \cdots \geq \sigma_n$.
Then we have  $\sigma_n \in (-\frac13,0)$.

Consider a DC decomposition of $M$ as
$M_{+} = U \max(0, \Sigma) U^{\tran}$ and $M_{-} = U \max(0, -\Sigma) U^{\tran},$
where the max-operator  $\max(\cdot,\cdot)$  takes the maximum of the two matrices entry-wisely.
 It is clear that   $M_{-} \prec |\sigma_n| I_n$.
Recalling \eqref{equ:lemma:M1M2:M}, we  thus obtain a DC decomposition of $M_1 - \K^{\tran}M_2^{-1}\K$ as
\begin{equation*}
(M_1 - \K^{\tran}M_2^{-1}\K)_+ = \left(M_1 + 2 \Sigma_f\right)^{\frac12}M_{+} \left(M_1 + 2 \Sigma_f\right)^{\frac12}
\end{equation*}
and
\begin{equation}\label{temp3}
(M_1 - \K^{\tran}M_2^{-1}\K)_- = \left(M_1 + 2 \Sigma_f\right)^{\frac12} M_{-} \left(M_1 + 2 \Sigma_f\right)^{\frac12}.
\end{equation}
Choosing $\rho = \frac{1 - 2|\sigma_n|}{1 + |\sigma_n|} \in (\frac14,1)$, with \eqref{temp3} and $M_{-} \prec |\sigma_n| I_n$, we thus  have
\begin{align}
(2 + \rho)(M_1 - \K^{\tran}M_2^{-1}\K)_- {}& \prec (2 + \rho)|\sigma_n| (M_1 + 2\Sigma_f) = (1 - \rho)(M_1 + 2\Sigma_f) \nn \\
{} &   \preceq (1 - \rho) M_1 + \frac32\Sigma_f + \rho\big(M_1 -  \K^{\tran}M_2^{-1} \K \big)_+.\nn
\end{align}
Substituting $\big(M_1 -  \K^{\tran}M_2^{-1} \K \big)_+ =  M_1 -  \K^{\tran}M_2^{-1} \K  +   \big(M_1 -  \K^{\tran}M_2^{-1} \K \big)_{-}$ into the above assertion, by some easy calculations, we get \eqref{equ:M1:M2:condition:2}. The proof is completed.
\qed \end{proof}

Now we are ready to establish the convergence of PrePDHG (Algorithm \ref{alg:vmpdhg}).
\begin{theorem}\label{thm:convergence}
Let $\{(x^k, y^k)\}$ be the sequence generated by Algorithm \ref{alg:vmpdhg} with $\epsilon = 0$ and $M_1, M_2$ satisfying \eqref{equ:M1M2:condtion}.
Then $\{(x^k, y^k)\}$ converges to a  saddle point of \eqref{PD}.
\end{theorem}
\begin{proof}
 Let the sequence $\{(u^k,x^k,\lambda^k)\}$ be  generated by iPADMM \eqref{ipadmm}.   Since $M_{1}$ and $M_{2}$ satisfy  \eqref{equ:M1M2:condtion},  we know from Lemma \ref{lemma:M1M2}  that  $\{(u^k,x^k,\lambda^k)\}$ converges to an optimal solution $(u^{\star}, x^{\star}, \lambda^{\star})$ of \eqref{P:1}, namely, 
\be\label{equ:convergence:a0}
0 \in \partial f(x^{\star}) + \K^{\tran} M_2^{-\frac12} \lambda^{\star}, \quad   0 \in \partial g(u^\star) - M_2^{-\frac12} \lambda^{\star}, \quad  \K x^{\star} - u^{\star} = 0.
\ee
Recalling the transform \eqref{equ:transform:y}, we know from the proof of Lemma \ref{lemma:equiv} that $\{(x^k, y^k)\}$ is exactly the sequence generated by PrePDHG \eqref{vmpdhg:2} with $x^0$ and $y^0 = M_2^{-1}(M_2^{\frac12} \lambda^0 + \K x^0 - u^{1})$. Since $\{(u^k,x^k,\lambda^k)\}$ converges to  $(u^{\star}, x^{\star}, \lambda^{\star})$, we know from \eqref{equ:convergence:a0} that $x^k \rightarrow x^{\star}$ and $y^k \rightarrow y^{\star}:= M_2^{-\frac12} \lambda^{\star}$ and
$$0 \in \partial f(x^{\star}) + \K^{\tran} y^{\star}, \quad 0 \in \partial g(\K x^{\star}) - y^{\star},$$
which with  the fact that $g$ is proper closed convex and \eqref{equ:f:closed} shows
$$
0 \in \partial f(x^{\star}) + \K^{\tran} y^{\star}, \quad  0 \in \partial g^{*}(y^{\star}) - \K x^{\star}.
$$
This means that $(x^{\star}, y^{\star})$ is a  saddle point  of \eqref{PD}.   The proof is completed.
\qed \end{proof}

\subsection{Tightness of  Condition \eqref{equ:M1M2:condtion}} \label{subsection:tight}
We first claim that condition \eqref{equ:M1M2:condtion}  is tight in the sense that the constant ``$4/3$'' can not be replaced by any number larger than it, namely, the sign ``$<$'' can not be improved to ``$\leq$''.
\begin{lemma}\label{lemma:tight:1}
Let $\{(x^k, y^k)\}$ be the sequence generated by Algorithm \ref{alg:vmpdhg} with $\epsilon = 0$.
Suppose condition  \eqref{equ:M1M2:condtion} is replaced by
\be \label{equ:M1M2:condtion:example1}
M_{1}  + \frac12 \Sigma_{f}\succ 0,\quad  M_{2}\succ 0,\quad \left\|M_2^{-\frac12}\K\Big(M_1 + \frac12 \Sigma_{f}\Big)^{-\frac12}\right\|^2 \leq \rho_{1}.
\ee
\begin{itemize}[leftmargin=*]
\item[](a). If $\rho_{1} \in (0,4/3)$, then $\{(x^k, y^k)\}$ converges to a saddle point of \eqref{PD}.
\item[](b). If $\rho_{1} \geq 4/3$, then $\{(x^k, y^k)\}$ is not necessarily convergent.
\end{itemize}
 \end{lemma}
\begin{proof}
The assertion of (a) comes from Theorem \ref{thm:convergence} and the fact that \eqref{equ:M1M2:condtion} is true if \eqref{equ:M1M2:condtion:example1} holds for any fixed $\rho_{1} \in (0,4/3)$.
To prove (b),
consider a simple instance of problem \eqref{PD} as
\begin{equation}\label{equ:counter:example}
\min_{x\in\Rbb} \max_{y \in \Rbb} \ xy.
\end{equation}
Note that such an example is a special case of the one in \cite[Section 3.2]{li2022improved} by setting $n = m = 1$ and $A = 1$ therein.
It is easy to see \eqref{equ:counter:example} has a unique saddle point $(0,0)$.   For this problem, $\Sigma_{f} = 0, \K = 1$, and $M_{1}, M_{2}$ take the form of   $M_{1} = 1/\tau, M_{2} = 1/\sigma$ with $\tau, \sigma > 0$. In this case,  condition \eqref{equ:M1M2:condtion:example1}  becomes $\tau, \sigma > 0$ and $\tau \sigma \leq \rho_{1}$.
We next  show that  if
$$
\tau > 0,\  \sigma > 0,\ \tau \sigma = \frac43, \quad
\begin{pmatrix}
x^{0} \\ y^{0}
\end{pmatrix}
\not \in
\mathcal{S}:= \left\{
a
\begin{pmatrix}
2 \\ \sigma
\end{pmatrix}: a \in \Rbb
\right\},
 $$
 then the sequence generated by PrePDHG  diverges, which is enough to finish the proof.

Specifically, by some easy calculations, the PrePDHG recursion \eqref{vmpdhg} for problem \eqref{equ:counter:example} reads as
$$
\left\{
\begin{aligned}
x^{k+1} ={}& x^k- \tau y^k,  \\
y^{k+1} ={}&  \sigma x^k+\left(1- 2\tau \sigma\right)y^k,
\end{aligned}
\right.
\nn
$$
which can be reformulated as
\begin{equation}\label{equ:counter:example:00}
\begin{pmatrix}
x^{k+1} \\ y^{k+1}
\end{pmatrix}
= G
\begin{pmatrix}
x^{k} \\ y^{k}
\end{pmatrix} \quad
\mbox{with}\quad
G := \begin{pmatrix}
1 & -\tau\\
\sigma & 1 - 2 \tau \sigma
\end{pmatrix}.
\end{equation}
Since $\tau\sigma = 4/3$,  it is easy to verify that the two eigenvalues of $G$ is $-1$ and $1/3$ and
\begin{equation} \label{equ:counter:example:01}
G = V \begin{pmatrix} 1/3 & 0 \\ 0 & - 1\end{pmatrix}V^{-1} \quad \mbox{with} \quad  V = \begin{pmatrix} {2}/{\sigma} & {\tau}/2 \\ 1 & 1 \end{pmatrix}.
\end{equation}
We have from \eqref{equ:counter:example:00} and \eqref{equ:counter:example:01} that
\begin{equation}
\begin{pmatrix}
x^{k+1} \\ y^{k+1}
\end{pmatrix}
= G^{k+1}
\begin{pmatrix}
x^{0} \\ y^{0}
\end{pmatrix}
= V \begin{pmatrix} 3^{-k} & 0 \\ 0 & (-1)^{k}\end{pmatrix}V^{-1}
\begin{pmatrix}
x^{0} \\ y^{0}
\end{pmatrix}.
\nn
\end{equation}
It is obvious that
$$
\begin{aligned}
&\left\{\begin{pmatrix}
x^{k+1} \\ y^{k+1}
\end{pmatrix}
\right\} ~\mbox{is convergent}\\
\iff{}&
V^{-1} \begin{pmatrix}
x^{0} \\ y^{0}
\end{pmatrix}
=\begin{pmatrix}
a \\ 0
\end{pmatrix}
\ \mbox{for some $a \in \Rbb$}
\iff
\begin{pmatrix}
x^{0} \\ y^{0}
\end{pmatrix} \in \mathcal{S}.
\end{aligned}
$$
Hence, if $\begin{pmatrix}
x^{0} \\ y^{0}
\end{pmatrix} \not\in \mathcal{S}$,  then $\left\{\begin{pmatrix}
x^{k+1} \\ y^{k+1}
\end{pmatrix}
\right\}$ diverges and  certainly will not converge to $\begin{pmatrix}
0\\ 0
\end{pmatrix}.$ The proof is completed.
\qed \end{proof}

We next claim that condition \eqref{equ:M1M2:condtion}  is tight in the sense that the constant ``$1/2$'' can not be replaced by any number larger than it.
\begin{lemma}\label{lemma:tight:2}
Let $\{(x^k, y^k)\}$ be the sequence generated by Algorithm \ref{alg:vmpdhg} with $\epsilon = 0$.
Suppose condition  \eqref{equ:M1M2:condtion} is replaced by
\be \label{equ:M1M2:condtion:example2}
M_{1}  + \rho_{2} \Sigma_{f}\succ 0,\quad  M_{2}\succ 0,\quad \left\|M_2^{-\frac12}\K\left(M_1 +  \rho_{2} \Sigma_{f}\right)^{-\frac12}\right\|^{2} <  \frac43.
\ee
\begin{itemize}[leftmargin=*]
\item[](a). If $\rho_{2} \in (0,1/2]$, then $\{(x^k, y^k)\}$ converges to a saddle point of \eqref{PD}.
\item[](b). If $\rho_{2} > 1/2$, then $\{(x^k, y^k)\}$ is not necessarily convergent.
\end{itemize}
 \end{lemma}
\begin{proof}
The assertion of (a) comes from Theorem \ref{thm:convergence} and the fact that \eqref{equ:M1M2:condtion} is true if \eqref{equ:M1M2:condtion:example2} holds for any fixed $\rho_{2} \in (0,1/2]$. To prove (b),
consider a simple instance of problem  \eqref{PD} as
\be\label{equ:counter:example2}
\min_{x\in\Rbb} \max_{y \in \Rbb} \ \frac12 x^{2} + xy.
\ee
It is easy to see \eqref{equ:counter:example2} has a unique saddle point $(0,0)$.   For this problem, $\Sigma_{f} = 1, \K = 1$, and $M_{1}, M_{2}$ take the form of   $M_{1} = 1/\tau, M_{2} = 1/\sigma$ with $1/\tau + \rho_{2} > 0, \sigma > 0$. In this case,  condition \eqref{equ:M1M2:condtion:example2} becomes $0 < \sigma < \frac43 (1/\tau + \rho _{2})$.
We  only need to show that for any $\rho_{2}\in (1/2,1]$ and $\rho_{3} \in \left(1/2, \rho_{2}\right)$ if
 $$0 < \sigma = \frac43(1/\tau + \rho_{3}) < \frac43(1/\tau + \rho_{2}),$$
then the sequence generated by  PrePDHG is not necessarily convergent.

First, it is not hard to verify that the PrePDHG recursion \eqref{vmpdhg} for problem \eqref{equ:counter:example2} reads as
$$
\left\{
\begin{aligned}
x^{k+1} ={}& (x^k- \tau y^k)/(1 + \tau),  \\
y^{k+1} ={}&  \left(\sigma (1 - \tau) x^k+\left(1 + \tau- 2\tau \sigma\right)y^k\right)/(1 + \tau),
\end{aligned}
\nn
\right.
$$
which can be reformulated as
\be \label{equ:counter:example2:00}
\begin{pmatrix}
x^{k+1} \\ y^{k+1}
\end{pmatrix}
= \widetilde G
\begin{pmatrix}
x^{k} \\ y^{k}
\end{pmatrix} \quad
\mbox{with}\quad
\widetilde G := \frac{1}{1 + \tau}\begin{pmatrix}
1 & -\tau\\
\sigma(1 - \tau) & 1 + \tau - 2 \tau \sigma
\end{pmatrix}.
\ee
The characteristic polynomial of $\widetilde G$ is given as
\begin{equation}
p(\mu) = \mu^{2} -\frac{1}{1+ \tau} \left(\tau - \frac{2(1 + 4 \rho_{3} \tau)}{3}\right) \mu  - \frac{1 + 4 \rho_{3} \tau}{3(1+ \tau)}. \nn
\end{equation}
Noting $\frac{1 + \tau}{\tau} \geq \frac{1}{\tau} + \rho_{2} > 0$,  we have  $$p(-1) =  \frac{2(1 - 2\rho_{3})\tau}{1 + \tau}<0,$$
which tells that at least one eigenvalue of $\widetilde G$ is  less than $-1$, i.e., $\|\widetilde G\|>  1$. Therefore, the sequence $\{(x^{k}, y^{k})\}$ generated by \eqref{equ:counter:example2:00} is not necessarily convergent. The proof is completed.
\qed \end{proof}

\subsection{Sublinear Convergence Rate} \label{subsection:sublinear}
We now investigate the sublinear convergence rate of PrePDHG (Algorithm \ref{alg:vmpdhg}).
\begin{theorem}\label{thm:subconvergence}
Let $\{(x^k, y^k)\}$ be the sequence generated by  Algorithm \ref{alg:vmpdhg} with $\epsilon = 0$ and  $M_1,   M_2$ satisfying  \eqref{equ:M1M2:condtion}.  Then we have
\be \label{equ:lemma:rate:000}
\min_{1 \leq k \leq t}  \Rcal(x^{k}, y^{k}) = o\left(\frac{1}{\sqrt{t}}\right) \quad \mbox{and} \quad \min_{1 \leq k \leq t}  \Rcal(x^{k}, y^{k-1}) = o\left(\frac{1}{\sqrt{t}}\right).
\ee
Moreover,  if condition \eqref{equ:M1M2:condtion} is replaced by
\begin{equation}\label{equ:M1M2:condtion:2}
M_{1}  + \frac12 \Sigma_{f}\succ 0,\quad  M_{2}\succ 0,\quad \left\|M_2^{-\frac12}\K\Big(M_1 + \frac12 \Sigma_{f}\Big)^{-\frac12}\right\|< 1,
\end{equation}
then  we have
\be \label{equ:lemma:rate:111}
\Rcal(x^t, y^t) = o\left(\frac{1}{\sqrt{t}}\right)\quad   \mbox{and}\quad \Rcal(x^{t}, y^{t-1}) = o\left(\frac{1}{\sqrt{t}}\right),\quad \forall t \geq 1.
\ee
\end{theorem}
\begin{proof}
First, let us  bound  the KKT residual $\Rcal^{k+1}:= \Rcal(x^{k+1},y^{k+1})$ and  $\Rcal^{k+1/2}:= \Rcal(x^{k+1},y^{k})$ for $k \geq 0$.
From the optimality condition of \eqref{pre-pdhg:x},  we have
\begin{equation}\label{equ:optimality:x:00}
- M_{1}(x^{k+1} - x^{k}) \in \partial f(x^{k+1}) + \K^{\tran} y^{k} ,
\end{equation}
which implies
\be\label{equ:optimality:x}
\K^{\tran} (y^{k+1} - y^{k})-M_{1}(x^{k+1} - x^{k}) \in \partial f(x^{k+1}) + \K^{\tran} y^{k+1}
\ee and thus
\be \label{equ:thm:subconvergence:dist0}
\dist\left(0, \partial f(x^{k+1}) + \K^{\tran} y^{k+1}\right) \leq \left\|\K^{\tran} (y^{k+1} - y^{k})-M_{1}(x^{k+1} - x^{k})\right\|.
\ee
Similarly, using the optimality condition of \eqref{pre-pdhg:y},  we have
\be \label{equ:optimality:y}
\K(x^{k+1} - x^{k}) - M_{2}(y^{k+1} - y^{k}) \in \partial g^{*}(y^{k+1})  - \K x^{k+1}
\ee
and thus
\be \label{equ:thm:subconvergence:dist1}
\dist\left(0, \partial g^{*}(y^{k+1})  - \K x^{k+1}\right) \leq  \left\|\K(x^{k+1} - x^{k}) - M_{2}(y^{k+1} - y^{k})\right\|.
\ee

Let 
$$
\begin{aligned}
\hat \Rcal^{k+1} ={}&  \left\|\K^{\tran} (y^{k+1} - y^{k})-M_{1}(x^{k+1} - x^{k})\right\| \\
{}& +\left\|\K(x^{k+1} - x^{k}) - M_{2}(y^{k+1} - y^{k})\right\|.
\end{aligned}$$
 By the Cauchy-Schwarz inequality, we  have
\begin{equation*}
\begin{aligned}
\hat \Rcal^{k+1}
\leq \left(\|\K\|+\|M_1\|\right)\|x^{k+1} - x^{k}\|+\|\K\|  \cdot \|y^{k+1} - y^{k}\|+\|M_{2}(y^{k+1} - y^{k})\|.
\end{aligned}
\end{equation*}
Since $M_2 \succ 0$,  for any $z \in \Rbb^m$,  we have $\iprod{z}{M_2 z} \geq \lambda_{\min}(M_2) \iprod{z}{z}$ and $\|M_2 z\|^2 = \iprod{z}{M_2 M_2 z} \leq \|M_2\| \iprod{z}{M_2 z}$. Therefore, we have $\|y^{k+1} - y^{k}\| \leq  \frac{1}{\sqrt{\lambda_{\min}(M_{2})}}\|y^{k+1} - y^{k}\|_{M_2}~\textnormal{and}~
\|M_{2}(y^{k+1} - y^{k})\| \leq \sqrt{\|M_{2}\|}\|y^{k+1} - y^{k}\|_{M_2}.$
Then we immediately have
\begin{align}\label{equ:thm:subconvergence:R:hat}
 \hat \Rcal^{k+1}
 \leq c_{1}\|x^{k+1} - x^{k}\| + c_{2}\|y^{k+1} - y^{k}\|_{M_{2}},
\end{align}
where the constants
\be\label{equ:c1:c2}
c_{1} = \|\K\| + \|M_{1}\|,  \quad c_{2}= \frac{\|\K\|}{\sqrt{\lambda_{\min}(M_{2})}} + \sqrt{\|M_{2}\|}.
\ee
By the definition \eqref{equ:kkt:residual} of $\Rcal(x,y)$, it is not hard to obtain from \eqref{equ:thm:subconvergence:dist0}, \eqref{equ:thm:subconvergence:dist1}, and \eqref{equ:thm:subconvergence:R:hat} that
\be\label{equ:thm:subconvergence:R}
 \Rcal^{k+1} \leq \hat \Rcal^{k+1}
 \leq c_{1}\|x^{k+1} - x^{k}\| + c_{2}\|y^{k+1} - y^{k}\|_{M_{2}},
\ee
Using \eqref{equ:optimality:y} for $k: = k- 1$, we have $\K(x^{k} - x^{k-1}) - M_{2}(y^{k} - y^{k-1}) \in \partial g^{*}(y^{k})  - \K x^{k}$ and thus
$$ 
\K(x^{k} - x^{k-1}) + \K(x^{k} - x^{k+1}) - M_{2}(y^{k} - y^{k-1}) \in \partial g^{*}(y^{k})  - \K x^{k+1}.\nn
$$
Hence, we have
\be\label{equ:thm:subconvergence:dist1:2}
\begin{aligned}
{}&\dist\left(0, \partial g^{*}(y^{k})  - \K x^{k+1}\right) \\
\leq{}&  \left\|\K(x^{k} - x^{k-1}) + \K(x^{k} - x^{k+1}) - M_{2}(y^{k} - y^{k-1}) \right\|  \\
\leq{}& \|K\|(\|x^{k} - x^{k-1}\| + \|x^{k} - x^{k+1}\|) + \sqrt{\|M_{2}\|} \|y^{k} - y^{k-1}\|_{M_{2}},
\end{aligned}
\ee
where the second inequality uses
$\| M_{2}(y^{k} - y^{k-1})\| \leq \sqrt{\|M_{2}\|} \|y^{k} - y^{k-1}\|_{M_{2}}$.
On the other hand, \eqref{equ:optimality:x:00} implies
\be \label{equ:thm:subconvergence:dist1:3}
\dist(0, \partial f(x^{k+1}) + \K^{\tran} y^{k}) \leq \|M_{1}(x^{k+1} - x^{k})\| \leq \|M_1\| \|x^{k+1} - x^{k}\|.
\ee
Combining  \eqref{equ:thm:subconvergence:dist1:2} and \eqref{equ:thm:subconvergence:dist1:3} together, and by the definition \eqref{equ:kkt:residual} of $\Rcal(x,y)$,  we have
\begin{align}\label{equ:thm:subconvergence:R:half}
 \Rcal^{k+1/2}
 \leq  c_{1}\|x^{k+1} - x^{k}\|  + \|\K\| \cdot \|x^{k} - x^{k-1}\| + \sqrt{\|M_{2}\|} \|y^{k} - y^{k-1}\|_{M_{2}}.
\end{align}

Second, we  estimate the upper bound of $\|y^{k+1} - y^{k}\|_{M_2}$.
From   \eqref{equ:ipadmm:c:0} and \eqref{equ:transform:y}, we have
\begin{equation}
M_2^{\frac12}y^{k} = \lambda^{k} + M_2^{-\frac12}(\K x^{k} - u^{k+1})    = \lambda^{k+1} + M_2^{-\frac12}\K(x^{k} - x^{k+1}), \nn
\end{equation}
which again with \eqref{equ:transform:y} for $k:= k+1$ yields
\begin{equation}
\begin{aligned}\label{equ:thm:subconvergence:00}
M_2^{\frac12}(y^{k+1} - y^{k}) ={}& M_2^{-\frac12}(\K x^{k+2} - u^{k+2}) - M_2^{-\frac12}\K(x^{k} - x^{k+1})   \\
 {}& + M_2^{-\frac12}\K(x^{k+1} - x^{k+2}).
\end{aligned}
\end{equation}
Condition  \eqref{equ:M1M2:condtion} or \eqref{equ:M1M2:condtion:2} tells $\K^{\tran} M_2^{-1} \K   \prec \frac43 \left(M_1 + \frac12 \Sigma_f \right)$. Thus, for any $z \in \Rbb^n$, we have
\begin{align}\label{equ:lemma:rate:a00}
\Big\| M_2^{-\frac12} \K z\Big\|  = \Big\|M_2^{-1}\K z\Big\|_{M_2} = \|z\|_{\K^TM_2^{-1}\K}
\leq  \frac{2}{\sqrt{3}} \|z\|_{M_1 + \frac12\Sigma_f} \leq c_{3}\|z\|,
\end{align}
where $c_{3} =  2 \sqrt{\lambda_{\max}(M_1 + \frac12\Sigma_f)/3}$.
Hence, noticing that $\|M_2^{1/2} v\| = \|v\|_{M_2}$ for any $v\in\Rbb^m$, we have from \eqref{equ:thm:subconvergence:00}  and \eqref{equ:lemma:rate:a00} that
\begin{align}\label{equ:lemma:rate:b02}
\|y^{k+1} - y^{k}\|_{M_2}
\leq{}&  \Big\|M_2^{-\frac12}(\K x^{k+2} - u^{k+2})\Big\|  \nn \\ 
& +  c_{3}\left(\big\|x^{k+1} - x^{k}\big\| +  \big\|x^{k+1} - x^{k+2}\big\|  \right).
\end{align}
Combining  \eqref{equ:thm:subconvergence:R}  and \eqref{equ:lemma:rate:b02}, we obtain
\begin{align}\label{equ:lemma:rate:c00}
\Rcal^{k+1} \leq{}&   \left(c_1 + c_{2}c_{3}\right)  \big\|x^{k+1} - x^{k}\big\|  + c_{2}c_{3}\big\|x^{k+1} - x^{k+2}\big\|  \nn\\ 
& +   c_{2}\Big\|M_2^{-\frac12}(\K x^{k+2} - u^{k+2})\Big\|.
\end{align}
Combining  \eqref{equ:thm:subconvergence:R:half}  and \eqref{equ:lemma:rate:b02} with $k:= k - 1$, we have
\begin{align}\label{equ:lemma:rate:c00:half}
\Rcal^{k+1/2} \leq{}&  \left(c_1 + \sqrt{\|M_{2}\|}c_{3}\right)  \big\|x^{k+1} - x^{k}\big\|  +  \left(\|\K\| + \sqrt{\|M_{2}\|}c_{3}\right)\big\|x^{k} - x^{k-1}\big\| \nn \\
 {}& +  \sqrt{\|M_{2}\|}\,\big\|M_2^{-\frac12}(\K x^{k+1} - u^{k+1})\big\|.
\end{align}

Finally, similar to the proof of Theorem \ref{thm:convergence}, if condition  \eqref{equ:M1M2:condtion}  holds, it is easy to see that the conditions of Lemma \ref{guth2} for iPADMM \eqref{ipadmm} are satisfied. Thus, by applying Lemma \ref{guth2}, we have
\begin{equation}
\min_{0 \leq  k\leq t} \left\{\|x^{k} - x^{k+1}\|_{M_1 + \Sigma_f}^2 + \|M_2^{-\frac12}(\K x^{k+1} - u^{k+1})\|^2  \right\}  = o(1/t), \nn
\end{equation}
which with  $\|x^{k} - x^{k+1}\|_{M_1 + \Sigma_f} \geq \lambda_{\min}(M_1 + \Sigma_f) \|x^{k} - x^{k+1}\|$, \eqref{equ:lemma:rate:c00} and \eqref{equ:lemma:rate:c00:half} lead to  \eqref{equ:lemma:rate:000}.
If condition \eqref{equ:M1M2:condtion:2} holds, it is easy to see that the conditions of Lemma \ref{guth3} for iPADMM \eqref{ipadmm} are satisfied. Thus, by applying Lemma \ref{guth3}, we have 
$$\|x^{t} - x^{t+1}\|_{M_1 + \Sigma_f}^2 + \|M_2^{-\frac12}(\K x^{t+1} - u^{t+1})\|^2  =  o(1/t), $$
which with \eqref{equ:lemma:rate:c00} and \eqref{equ:lemma:rate:c00:half} leads to  \eqref{equ:lemma:rate:111}.
The proof is completed.
\qed \end{proof}

It is immediate to establish the iteration complexity of Algorithm \ref{alg:vmpdhg}.
\begin{corollary}
If  $\epsilon > 0$,  then Algorithm \ref{alg:vmpdhg} stops in $O(1/\epsilon^{2})$ iterations.
\end{corollary}

Revisiting the optimality condition \eqref{equ:PD:optimality},  instead of using the KKT residual, we can also measure the quality of approximate solution $(\hat x, \hat y)$ by giving an upper bound of the function value residual $\Lcal(\hat x,y) - \Lcal(x,\hat y)$ for any $x\in \Rbb^{n}$ and $y \in \Rbb^{m}$, see \cite{chambolle2016ergodic,chambolle2011first,liu2021acceleration, jiang2021approximate,jiang2021first,rasch2020inexact} and the references therein for instance.   However, the existing results for PDHG and PrePDHG under condition \eqref{condition:PC1}
 or \eqref{condition:LiuYin}
  are all ergodic,   which always have the bound:
\be \label{equ:residual:L}
\Lcal(\bar x^{t},y) - \Lcal(x, \bar y^{t}) \leq   \frac{\varphi_{1}(x,x^{0}) + \varphi_{2}(y,y^{0})}{t}, \quad \forall x \in \Rbb^{n},\ \forall y \in \Rbb^{m},
\ee
or
\be \label{equ:residual:L:2}
\Lcal(\bar x^{t},y^\star) - \Lcal(x^\star, \bar y^{t}) \leq   \frac{\varphi_{1}(x^\star,x^{0}) + \varphi_{2}(y^\star,y^{0})}{t},
\ee
where $\bar x^{t} = \frac1t \sum_{i = 1}^{t} x^{i}$, $\bar y^{t} = \frac1t \sum_{i = 1}^{t} y^{i}$, and $\varphi_{1}(\cdot, \cdot), \varphi_{2}(\cdot, \cdot)$ are some nonnegative functions and $(x^\star,y^\star)$ is a saddle point.

 Here, we aim to investigate some non-ergodic results with the help of our established bounds for the KKT residual.
\begin{lemma}\label{thm:rate2}
Let $\{(x^k, y^k)\}$ be the sequence generated by Algorithm \ref{alg:vmpdhg} with $\epsilon = 0$ and  $M_1,   M_2$ satisfying  \eqref{equ:M1M2:condtion}.   Let $(x^{\infty}, y^{\infty})$ be the limit point of  $\{(x^k, y^k)\}$. Define a constant $\bar c = \sup_{k \geq 0}\{\|x^k - x^\infty\| + \|y^k - y^\infty\|\}$ and denote
\be\label{equ:kt}
k(t) :=\argmin_{1 \leq k \leq t} \left\{c_{1}\|x^{k+1} - x^{k}\|_{M_{1} + \frac12 \Sigma_{f}} + c_{2}\|y^{k+1} - y^{k}\|_{M_{2}}\right\},
\ee
where $c_1$ and $c_2$ are defined \eqref{equ:c1:c2}.
Then for $t \geq 1$,
\be  \label{equ:thm:rate2:00}
\begin{aligned}
{}&\Lcal(x^{k(t)}, y) - \Lcal(x,y^{k(t)}) \\
\leq{}&  o(1/\sqrt{t})\left(\bar c + \|x^\infty - x\| + \|y^\infty - y\|\right), \quad \forall x \in \Rbb^{n}, \quad \forall y \in \Rbb^{m}
\end{aligned}
\ee
and
\begin{equation}\label{equ:thm:rate2:00:v2}
\Lcal(x^{k(t)}, y^{\infty}) - \Lcal(x^{\infty},y^{k(t)})  \leq   o(1/\sqrt{t}).
\end{equation}
Moreover,  if condition \eqref{equ:M1M2:condtion} is replaced by \eqref{equ:M1M2:condtion:2},  then for any  $t \geq 1$
\be  \label{equ:thm:rate2:11}
\begin{aligned}
{}&\Lcal(x^{t}, y) - \Lcal(x,y^{t})\\
 \leq{}& o(1/\sqrt{t}) \left(\bar c + \|x^{\infty} - x\| + \|y^{\infty} - y\|\right), \quad \forall x \in \Rbb^{n}, \quad \forall y \in \Rbb^{m}
\end{aligned}
\ee
and
\begin{equation} \label{equ:thm:rate2:11:v2}
\Lcal(x^{t}, y^{\infty}) -  \Lcal(x^{\infty}, y^{t})  \leq o(1/\sqrt{t}).
\end{equation}
\end{lemma}

\begin{proof}
By the convexity of $f(x) + \iprod{x}{\K^{\tran} y^{k+1}}$ and  \eqref{equ:optimality:x}, we have
\be \label{equ:lemma:rate2:f}
\begin{aligned}
&f(x^{k+1}) + \iprod{x^{k+1}}{\K^{\tran}y^{k+1}}- f(x) - \iprod{x}{\K^{\tran} y^{k+1}}  \\
\leq{}& \iprod{x^{k+1} - x}{\K^{\tran} (y^{k+1} - y^{k})-M_{1}(x^{k+1} - x^{k})}, \quad \forall x \in \Rbb^{n}.
\end{aligned}
\ee
Similarly, by  the convexity of $g^{*}(y) - \iprod{y}{\K x^{k+1}}$ and the optimality condition \eqref{equ:optimality:y}, we have
\be \label{equ:lemma:rate2:g}
\begin{aligned}
&  \left(g^{*}(y^{k+1})- \iprod{y^{k+1}}{\K x^{k+1}}\right)- \left(g^{*}(y) - \iprod{y}{\K x^{k+1}}\right)  \\
\leq{}& \iprod{y^{k+1}-y}{\K(x^{k+1} - x^{k}) - M_{2}(y^{k+1} - y^{k})}, \quad \forall y \in \Rbb^{m}.
\end{aligned}
\ee
Summing up \eqref{equ:lemma:rate2:f} and \eqref{equ:lemma:rate2:g}, for any  $x \in \Rbb^{n}$ and $y \in \Rbb^{m}$,  we have
\begin{align}\label{equ:thm:rate2:L}
&\Lcal(x^{k+1}, y) - \Lcal(x,y^{k+1})\nn \\
 \leq{}&  \iprod{x^{k+1} - x}{\K^{\tran} (y^{k+1} - y^{k})-M_{1}(x^{k+1} - x^{k})}  \nn \\
 {}& + \iprod{y^{k+1}-y}{\K(x^{k+1} - x^{k}) - M_{2}(y^{k+1} - y^{k})} \nn \\
 \leq{}& \left(c_{1}\|x^{k+1} - x^{k}\|_{M_{1} + \frac12 \Sigma_{f}} + c_{2}\|y^{k+1} - y^{k}\|_{M_{2}}\right) \left(\|x^{k+1} - x\| + \|y^{k+1} - y\|\right),
\end{align}
where the second inequality uses the Cauchy-Schwarz inequality and \eqref{equ:thm:subconvergence:R:hat}.

Suppose $M_1$ and $M_2$ satisfy \eqref{equ:M1M2:condtion}. From the proof of Theorem \ref{thm:subconvergence}, we know that
\be \label{equ:thm:rate2:L:222}
\min_{1 \leq k \leq t} \left\{c_{1}\|x^{k+1} - x^{k}\|_{M_{1} + \frac12 \Sigma_{f}} + c_{2}\|y^{k+1} - y^{k}\|_{M_{2}}\right\} = o\left( \frac{1}{\sqrt{t}} \right).
\ee
Note that for any $k$, by the Cauchy-Schwarz inequality and the definition of $\bar c$, we have
\be\label{equ:xk:yk:bound}
\begin{aligned}
{}& \|x^{k+1} - x\| + \|y^{k+1} - y\|  \\ 
\leq{} &
\|x^{k+1} - x^\infty\| +   \|y^{k+1} - y^\infty\| + \|x^{\infty} - x\| + \|y^{\infty} - y\| \\
 \leq{} & \bar c +  \|x^{\infty} - x\| + \|y^{\infty} - y\|,
\end{aligned}
\ee
which together with the definition of  $k(t)$ in \eqref{equ:kt},  \eqref{equ:thm:rate2:L}, and \eqref{equ:thm:rate2:L:222} implies \eqref{equ:thm:rate2:00}.

Suppose  $M_1$ and $M_2$ satisfy \eqref{equ:M1M2:condtion:2}.   From the proof of Theorem \ref{thm:subconvergence}, we know that
$$c_{1}\|x^{k+1} - x^{k}\|_{M_{1} + \frac12 \Sigma_{f}} + c_{2}\|y^{k+1} - y^{k}\|_{M_{2}} = o\left( \frac{1}{\sqrt{k}} \right),$$
which with \eqref{equ:thm:rate2:L} and \eqref{equ:xk:yk:bound} implies \eqref{equ:thm:rate2:11}.

By \eqref{thm:convergence}, we know that $(x^\infty, y^\infty)$ is a saddle point of \eqref{PD}.
Thus,  \eqref{equ:thm:rate2:00:v2} and \eqref{equ:thm:rate2:11:v2} follow directly from  \eqref{equ:thm:rate2:00}
and \eqref{equ:thm:rate2:11}, respectively,  by setting $x = x^\infty$ and $y = y^\infty$.
The proof is completed.
\qed \end{proof}

Some remarks on our results about the sublinear convergence rate of PrePDHG are listed below.
First, to the best of our knowledge,  the sublinear rate based on the KKT residual $\Rcal(x^{k+1}, y^{k+1})$ or $\Rcal(x^{k+1},y^{k})$ is new for PDHG like methods since the existing results mainly focus on \eqref{equ:residual:L}. Compared with \eqref{equ:residual:L}, the upper bounds of the KKT residual  $\Rcal(x^{k+1}, y^{k+1})$ or $\Rcal(x^{k+1}, y^{k})$ are always computable, see Remark \ref{remark:upperbound:R} ahead.
Our sublinear rate result for the KKT residual also tells that Algorithm \ref{alg:vmpdhg} can return an $\epsilon$-solution in $O(1/\epsilon^{2})$ iterations.
Second, for the function value residual measurement, our sublinear rate result is the first non-ergodic result since the existing results are all ergodic,  see \cite{chambolle2011first,chambolle2016ergodic,jiang2021first,jiang2021approximate} for instance. It should be clear that our non-ergodic results are $o(1/\sqrt{t})$ while the existing ergodic results are $O(1/t)$ both under the condition that  $(x,y)$ is in a compact set. It remains unknown whether the non-ergodic result can be improved to $O(1/t)$.

To end this section,  we briefly discuss a dual formulation of the PrePDHG recursion \eqref{vmpdhg} in the following remark.
\begin{remark} \label{remark:dual:pdhg}
In Section \ref{S2}, we assume that problem \eqref{PD} has a saddle point, which means that solving \eqref{PD} is equivalent to solving the following problem:
 \begin{equation}\label{DP}
\min_{y\in \Rbb^m} \max_{x\in \Rbb^n}\;   g^*(y)  - \iprod{y}{\K x} - f(x).
\end{equation}
Using PrePDHG \eqref{vmpdhg} to solve \eqref{DP} and based on the symmetry of the primal and dual variables between \eqref{PD} and \eqref{DP} (the primal variable $x$ in \eqref{PD} is the dual variable in \eqref{DP} and vice versa), we can obtain the other PrePDHG recursion, which can also be used to solve \eqref{PD}:
\begin{subnumcases}{\label{pre-pdhg:dual}}
x^{k+1}=\argmin_{x \in \Rbb^n}\;f(x)+ \iprod{\K x}{2y^k-y^{k-1}} + \frac{1}{2}\left\|x-x^k\right\|_{Q_1}^2\label{pre-pdhg1:x},\\
y^{k+1}=\argmin_{y \in \Rbb^m}\;g^*(y)  - \iprod{\K x^{k+1}}{y} +\frac{1}{2}\|y-y^k\|_{Q_2}^2\label{pre-pdhg1:y},
\end{subnumcases}
where the symmetric matrices $Q_1\in \Rbb^{n \times n}$  and  $Q_2 \in \Rbb^{m \times m}$ satisfy
$$
\begin{pmatrix}
Q_1 &  -\K^\tran\\
-\K &  \frac43\left(Q_2 + \frac12 \Sigma_{g^*}\right)
\end{pmatrix}
\succ 0.
$$
Consider an equivalent formulation of problem \eqref{D} (note that \eqref{D} is also the primal formulation of problem \eqref{DP})
\begin{equation}\label{D:1}\tag{D1}
\begin{array}{cl}
\min\limits_{z\in \Rbb^{n},\ y\in \Rbb^{m}}  &  f^*(z)+g^*(y)\\
\mathrm{s.t.} &  Q_1^{-\frac12}(z+\K^{\tran}y)=0.
\end{array}
\end{equation}
The iPADMM recursion for \eqref{D:1} is given as
\begin{subnumcases}{\label{iPADMM:dual}}
z^{k+1} = \argmin_{z\in \Rbb^n} \bar \Lcal_1(z,y^k,\lambda^k),\label{equ:proxadmm1:a:0}\\
y^{k+1} = \argmin_{y\in \Rbb^m} \bar \Lcal_1(z^{k+1},y,\lambda^k)  + \frac12\|y - y^k\|_{Q_2 - \K Q_1^{-1}\K^{\tran}}^2, \label{equ:proxadmm1:b:0}\\
\lambda^{k+1} = \lambda^k + Q_1^{-\frac12} (z^{k+1} + \K^{\tran}y^{k+1}), \label{equ:proxadmm1:c:0}
\end{subnumcases}
where
$\bar \Lcal_1(z,y,\lambda)=f^*(z)+g^*(y)+\big\langle \lambda, Q_1^{-1/2}(z+\K^{\tran}y)\big\rangle +\frac12\|z+\K^{\tran}y\|_{Q_1^{-1}}^2
$
 is the augmented Lagrangian function of \eqref{D:1}. Using the same process in Lemma \ref{lemma:equiv}, we can show the equivalence between \eqref{pre-pdhg:dual}  and the iPADMM  \eqref{iPADMM:dual}.   The convergence results of \eqref{pre-pdhg:dual} can thus be established similar to that in Sections \ref{subsection:convergence} and \ref{subsection:sublinear}. We omit the details for brevity.
\end{remark}

\section{Revisit on the Choices of $M_{1}$ and $M_{2}$}\label{S4}
In this section, we revisit PrePDHG and discuss the choices of $M_{1}$ and $M_{2}$. Specifically, with the choices in Sections \ref{section:ipdhg} and \ref{section:diagonal}, PrePDHG gives improved versions of the original PDHG and PDHG with diagonal preconditioners, respectively. In Section \ref{section:general}, we investigate the choice of $M_{1} =  \tau^{-1}I_{n}$, $M_{2} = \gamma \tau \K\K^{\tran} + P$ and its extensions. In Section \ref{subsection:by}, we consider a special case when $g^{*}(y) = \iprod{b}{y}$ and discuss an enhanced BALM (eBALM) and an eBALM with symmetric Gauss-Seidel iterations (eBALM-sGS).

\subsection{$M_1 = \tau^{-1}I_n$, $M_2 = \sigma^{-1}I_m$}\label{section:ipdhg}
If the proximal operators of $f$ and $g^{*}$ are both easy to compute, we can simply take $M_1 = \tau^{-1}I_n$, $M_2 = \sigma^{-1}I_m$ with $\tau, \sigma > 0$.  In this case,  PrePDHG \eqref{vmpdhg:2}  reduces to the original  PDHG \eqref{alg:pdhg}, which can be reformulated as
 \begin{equation}\label{ipdhg}
\left\{
\begin{aligned}
x^{k+1}={}&\prox_{\tau f}\left(x^{k} - \tau \K^{\tran}y^{k}\right),\\
y^{k+1}={}&\prox_{\sigma g^{*}}\left(y^{k} + \sigma \K(2x^{k+1} - x^k)\right),
\end{aligned}
\right.
\end{equation}
where $\prox_{\tau f}(x)$ is defined in Section \ref{S2}.
Define a constant $\lambda_{\min}^f:= \lambda_{\min}(\Sigma_f) \geq 0$.
For such choices of $M_1$ and $M_2$, we have
$$
\left\|M_2^{-\frac12}\K \Big(M_1 + \frac12 \Sigma_{f}\Big)^{-\frac12}\right\|^{2} \leq \frac{\sigma \|K\|^2}{1/\tau+ (1/2)\lambda_{\min}^f} = \frac{\tau\sigma\|K\|^2}{1 + (\tau/2) \lambda_{\min}^f}.
$$
To make condition \eqref{equ:M1M2:condtion} hold, we obtain the convergence condition of the  PDHG \eqref{alg:pdhg} or \eqref{ipdhg} as
\begin{equation}\label{equ:phdg:condition:new}
\tau, \sigma > 0, \quad \tau\sigma \|\K\|^{2} < \frac43\left(1 + \frac{\tau \lambda_{\min}^f}{2}\right),
\end{equation}
which can imply \eqref{condition:43}.  Besides, we also know from Lemma \ref{lemma:tight:1} and Lemma \ref{lemma:tight:2} that \eqref{equ:phdg:condition:new} is tight in the sense that the constant $4/3$ could not be enlarged.

\begin{remark}
 If $\lambda_{\min}^f$ is not easy to estimate or $f$ has no more property beyond convexity, we can set $\lambda_{\min}^f$ as zero. Moreover,  in the following part of this section, to make the discussion precise, we choose $\Sigma_f = 0$ and $\lambda_{\min}^f = 0$.  We refer to Section \ref{subsection:Birkhoff} for one exception, wherein there holds that $\Sigma_f = I_{n^2}$ and $\lambda_{\min}^f = 1$.
\end{remark}

\subsection{Diagonal $M_{1}$ and $M_{2}$} \label{section:diagonal}
If both $f$ and $g$ take the separable structures, namely, $f(x): = \sum_{j=1}^{n} f_{j}(x_{j})$, $g^{*}(y) = \sum_{i = 1}^{m} g_{i}^{*}(y_{i})$, and the proximal operators of $f_{j}$ and $g_{i}^{*}$ are all easy to compute, we can consider the following choices of diagonal $M_{1}$ and $M_{2}$, which were first proposed in \cite{pock2011diagonal}.
\begin{proposition}\label{prop-diag}
For any $\alpha \in [0,2]$ and $\gamma_{1}, \gamma_{2} > 0$, let
\be\label{equ:diag:M1}
M_{1} = \gamma_{1} \diag(\tau_{1}, \ldots, \tau_{n})\ \mbox{with}\ \tau_{j} = \delta +  \sum_{i=1}^{m} |\K_{ij}|^{2 - \alpha}, j = 1, \ldots, n,
\ee
\be \label{equ:diag:M2}
M_{2} = \gamma_{2} \diag(\sigma_{1}, \ldots, \sigma_{m})\ \mbox{with}\ \sigma_{i} = \delta +  \sum_{j =1}^{n} |\K_{ij}|^{\alpha}, i = 1, \ldots, m,
\ee
where $\delta \geq 0$ is chosen such that $\tau_{j}, \sigma_{i}$ are positive.    If $\gamma_{1}\gamma_{2} > \frac34$, then such $M_{1}$ and $M_{2}$ satisfy
\eqref{equ:M1M2:condtion}.
\end{proposition}
\begin{proof}
By \cite[Lemma 2]{pock2011diagonal}, we know $\|(M_{2}/\gamma_{2})^{-1/2}\K (M_{1}/\gamma_{1})^{-1/2}\| \leq 1$,  which implies that
$\|M_{2}^{-\frac12}\K M_{1}^{-\frac12}\|^2\leq \frac{1}{\gamma_{1}\gamma_{2}} < \frac43.$
The proof is completed.
\qed \end{proof}

With  choices \eqref{equ:diag:M1} and \eqref{equ:diag:M2},  PrePDHG \eqref{vmpdhg:2}  becomes
 \[
\left\{
\begin{aligned}
x_{j}^{k+1}={}&\prox_{\tau_{j} f_{j}}\left(x_{j}^{k} - \tau_{j} (\K^{\tran}y^{k})_{j}\right),\ j = 1, \ldots, n, \\
y_{i}^{k+1}={}&\prox_{\sigma_{i} g_{i}^{*}}\left(y_{i}^{k} + \sigma_{i} (\K(2x^{k+1} - x^k))_{i}\right), \ i = 1, \ldots, m.
\end{aligned}
\right.
\]
\begin{remark}
Taking $\gamma_{1} = \gamma_{2} = 1$ in \eqref{equ:diag:M1}  and \eqref{equ:diag:M2} yields the diagonal preconditioners in \cite{pock2011diagonal}. Lemma \ref{lemma:tight:1} tells that $\gamma_{1}\gamma_{2} > \frac34$ in Proposition \ref{prop-diag} is tight in the sense that ``$>$'' can not be improved to ``$\geq$''.
\end{remark}

\subsection{$M_{1} =  \tau^{-1}I_{n}$, $M_{2} = \gamma \tau \K\K^{\tran} + P$ and Extensions}\label{section:general}
Another choice is $M_{1} = \tau^{-1} I_{n}$ and $M_{2} = \tau \K\K^{\tran} + \theta I_{m}$ with $\tau, \theta > 0$, which was proposed   in \cite{liu2021acceleration}. Here, we consider a relaxed version of such choices.
\begin{proposition}\label{prop:M1M2:general}
Let $P \in \Rbb^{m\times m}$ be  a nonzero symmetric  positive semidefinite matrix such that $\K\K^{\tran} + P \succ 0$. Choose
\be \label{M1M2:general}
M_{1} =  \tau^{-1}I_{n}, \quad M_{2} = \gamma \tau \K\K^{\tran} + P, \quad \tau > 0,\quad \gamma \geq  \frac34,
\ee
then \eqref{equ:M1M2:condtion} holds.
\end{proposition}
\begin{proof}
It is easy to see that $M_{2}\succ 0$ from $\K\K^{\tran} + P\succ 0$ with $\K\K^{\tran}\succeq 0$ and $P\succeq 0$. Hence,
We have
\begin{align}
\left\|M_{2}^{-\frac12}\K M_{1}^{-\frac12}\right\|^{2}  ={} & \tau  \lambda_{\max}\left(\K^{\tran} \left(\gamma \tau \K\K^{\tran} + P\right)^{-1}  \K\right)  \nn \\
={}&  \frac{1}{\gamma}  \lambda_{\max}\left(\left(\gamma \tau \K\K^{\tran} + P\right)^{-1}  (\gamma \tau \K\K^{\tran}) \right) \nn\\
< {}&  \frac{1}{\gamma}  \lambda_{\max}\left(\left(\gamma \tau \K\K^{\tran} + P\right)^{-1} \left(\gamma \tau \K\K^{\tran}  + P\right)\right) \leq  \frac43, \nn
\end{align}
where the first inequality is due to $P \succeq 0$ but $P \neq 0$, and the second one relies on $\gamma \geq 3/4$. The proof is completed.
\qed \end{proof}

\begin{remark}
Similar to \eqref{M1M2:general}, letting $\hat P \in \Rbb^{m\times m}$ be  a  nonzero symmetric  positive semidefinite matrix such that $\K^{\tran}\K + \hat P \succ 0$, we can choose
\be \label{M1M2:general:2}
\quad M_{1} = \gamma \sigma \K^{\tran}\K + \hat P, \quad  M_{2} =  \sigma^{-1}I_{n},  \quad \sigma > 0,\quad \gamma \geq \frac34
\ee
such that condition \eqref{equ:M1M2:condtion} holds.  Note that very recently Bai \cite{bai2021new} considered  \eqref{M1M2:general}  and
\eqref{M1M2:general:2} with $\gamma = 1$ and symmetric positive definite $P$ and $\hat P$.
\end{remark}

In some problem, such as CT reconstruction in Section \ref{subsection:CT}, $g^{*}(y)$ takes a separable structure as $g^{*}(y) = g_{1}^{*}(y_{1}) + g_{2}^{*}(y_{2})$  with $y = \begin{pmatrix}y_{1} \\ y_{2}\end{pmatrix}$, $y_{1} \in \Rbb^{m_{1}}, y_{2} \in \Rbb^{m_{2}}$, in which the proximal of $g_{1}$ takes a closed form solution. In this case, motivated by  \cite{liu2021acceleration}, we can partition $\K$ as $\K = \begin{pmatrix} \K_1 \\ \K_2 \end{pmatrix}$ with $\K_1\in \Rbb^{m_1 \times n}$, $\K_2 \in \Rbb^{m_2\times n}$  and  choose
 \begin{equation} \label{M1M2:general:2block}
 M_1 = \frac{2\gamma}{ \tau}I_n,\quad  M_2 = \begin{pmatrix}
\sigma^{-1}  I_{m_1} & 0 \\ 0 & \tau  \K_2 \K_2^{\tran} + P_{2}\end{pmatrix}\quad \mbox{with}\quad \tau, \sigma > 0.
\end{equation}
We have the following result.
\begin{proposition}\label{prop:M1M2:general:2block}
Let $P_{2} \in \Rbb^{m_{2}\times m_{2}}$ be  a nonzero symmetric  positive semidefinite matrix such that $\K_{2}\K_{2}^{\tran} + P_{2} \succ 0$.  Let $\tau, \sigma > 0$. If $( \tau \sigma \|\K_{1}\|^{2} + 1)/\gamma\leq 8/3$ and  $M_{1}$ and $M_{2}$ are chosen according to \eqref{M1M2:general:2block},
then \eqref{equ:M1M2:condtion} holds.
\end{proposition}
\begin{proof}
It is easy to see that $M_{2}\succ 0$ from $\K_2\K_2^{\tran} + P_2\succ 0$ with $\K_2\K_2^{\tran}\succeq 0$ and $P_2\succeq 0$.  We thus have
\begin{align}
{}&\left\|M_{2}^{-\frac12}\K M_{1}^{-\frac12}\right\|^{2} \nn\\
 ={} &\frac{\tau}{2\gamma}  \lambda_{\max}\left(\sigma \K_{1} \K_{1}^{\tran} + \left(\tau \K_{2}\K_{2}^{\tran} + P_{2}\right)^{-1/2}  \K_{2} \left(\tau \K_{2}\K_{2}^{\tran} + P_{2}\right)^{-1/2}\right)  \nn \\
\leq{}& \frac{ \tau}{2\gamma} \sigma \|\K_{1}\|^{2} + \frac{1}{2\gamma}  \lambda_{\max}\left(\left(\tau \K_{2}\K_{2}^{\tran} + P_{2}\right)^{-1}  (\tau \K_{2}\K_{2}^{\tran}) \right) \nn\\
< {}& \frac{\tau \sigma \|\K_{1}\|^{2} + 1}{2\gamma} \leq \frac43. \nn
\end{align}
 The proof is completed.
\qed \end{proof}
\begin{remark}\label{remark:M1M2:general:2block}
A particular choice in \eqref{M1M2:general:2block}  is $\tau>0, \sigma>0$, and $\tau \sigma \|\K_{1}\|^{2} = 1$. In this case, Proposition \ref{prop:M1M2:general:2block} yields $ \gamma \geq \frac34.$
\end{remark}

\subsection{A Special Case  $g^{*} = \iprod{b}{y}$ and Beyond} \label{subsection:by}
In this subsection, we mainly consider the case when $g^{*}$ is a linear function, for which
with choice \eqref{M1M2:general}, the $y$-subproblem in PrePDHG \eqref{vmpdhg:2}  can be efficiently solved. Some more general cases of $g^*$ are also discussed at the end of this subsection.

Given a vector $b \in \Rbb^{m}$, we consider
 $$ 
 g(y) = \mathbb{I}_{\{b\}}\quad \mbox{and}\quad  g^*(y)=\sup_{z \in \Rbb^{m}}\{\iprod{z}{y}-g(z)\}= \iprod{b}{y},
 $$
 where $\mathbb{I}_{\{b\}}$ is the indicator function of the singleton  $\{b\}$. Hence, problem \eqref{PD} becomes
 \begin{equation}\label{PD:2} 
\min_{x\in \Rbb^n} \max_{y\in \Rbb^m}\; \Lcal(x,y):= f(x)+ \iprod{y}{\K x} - \iprod{b}{y}.
\end{equation}
The recursions of PrePDHG \eqref{vmpdhg} for \eqref{PD:2} are given as
\be \label{alg:ebalm}
\left\{
\begin{aligned}
x^{k+1}={}&\prox_{\tau f}\left(x^{k} - \tau \K^{\tran}y^{k}\right), \\
y^{k+1}={}&y^k+ (\gamma \tau \K\K^{\tran} + P)^{-1}\left(\K(2x^{k+1}-x^k)-b\right).
\end{aligned}
\right.
\ee

\begin{remark} \label{remark:R:upperbound}
For $g^*(y) = \langle b, y\rangle$,  compared with the results in Remark \ref{remark:upperbound:R}, we can obtain more compact upper bounds of $\Rcal(x^{k+1}, y^{k+1})$ and $\Rcal(x^{k+1}, y^{k})$.
By  \eqref{equ:thm:subconvergence:dist0} and \eqref{equ:kkt:residual}, we have
$$
\Rcal(x^{k+1}, y^{k+1}) \leq \max\{\|\K^{\tran} (y^{k+1} - y^{k})- \tau^{-1} (x^{k+1} - x^{k})\|, \|Kx^{k+1} - b\|\}.
$$
Besides, by \eqref{equ:optimality:x:00}  and \eqref{equ:kkt:residual},   we have
\be \label{equ:Rxkyk-1:Kxb}
\Rcal(x^{k+1}, y^k)  \leq \max\{\|\tau^{-1} (x^{k+1} - x^{k})\|, \|Kx^{k+1} - b\|\}.
\ee
\end{remark}
We next consider two choices of $P$, where the $y$-subproblem in \eqref{alg:ebalm} is easy to solve.

The first one is to choose $\gamma = 1$ and $P = \theta I_{m}$ for some  $\theta>0$. Then \eqref{alg:ebalm} reduces to the balanced ALM (BALM) \cite{he2021balanced} for solving  the following convex optimization problem
\begin{equation}\label{P1}
\min_{x\in\Rbb^{n}} \quad  f(x)\quad  \mst  \quad \K x=b,
\end{equation}
which corresponds to the primal formulation of \eqref{PD:2} (see Sections \ref{subsection:Birkhoff} and  \ref{section:EMD} for  two instances of \eqref{P1}).

\vspace{4pt}
\hspace{-6mm}
\begin{boxedminipage}{1\textwidth}
\noindent \textbf{BALM procedure:} Let $\tau > 0$ and $\theta > 0$. For given $(x^k, y^k)$, the new iterate $(x^{k+1}, y^{k+1})$ is generated by:
\begin{equation}\label{balm}
\left\{
\begin{aligned}
x^{k+1}={}&\prox_{\tau f}\left(x^{k} - \tau \K^{\tran}y^{k}\right),\\
 y^{k+1}={}&y^k+(\tau \K\K^{\tran}+\theta I_{m})^{-1}\left(\K(2x^{k+1}-x^k)-b\right).
\end{aligned}
\right.
\end{equation}
\end{boxedminipage}
\vspace{4pt}

\noindent In \cite{he2021balanced}, He and Yuan proved the convergence of BALM \eqref{balm}  in an elegant way by using the framework of variational inequalities.  Note that the parameters  $\tau$ and $\theta$ can be arbitrary  positive constants. By applying the results in Section \ref{section:vmpdhg}, we  obtain an enhanced BALM (eBALM), with global convergence and sublinear convergence rate, as follows:

\vspace{4pt}
\hspace{-6mm}
\begin{boxedminipage}{1\textwidth}
\noindent \textbf{eBALM procedure:} Let $\tau > 0$, $\theta > 0$ and $\gamma \geq 3/4$. For given $(x^k, y^k)$, the new iterate $(x^{k+1}, y^{k+1})$ is generated by:
\begin{equation}\label{ebalm}
\left\{
\begin{aligned}
x^{k+1}={}&\prox_{\tau f}\left(x^{k} - \tau \K^{\tran}y^{k}\right),\\
y^{k+1}={}&y^k+\gamma^{-1}(\tau \K\K^{\tran}+\theta I_{m})^{-1}\left(\K(2x^{k+1}-x^k)-b\right).
\end{aligned}
\right.
\end{equation}
\end{boxedminipage}
\vspace{4pt}

\begin{remark}
Note that $\gamma^{-1}$ is taken as 1 in \eqref{balm} and can be any number in $(0, 4/3]$ in \eqref{ebalm}. Therefore, compared with BALM, the stepsize of $y$-subproblem in eBALM can be enlarged to $4/3$ from 1  Moreover, $4/3$ is a tight upper bound of $\gamma^{-1}$ according to Lemma \ref{lemma:tight:1}.
\end{remark}

Next, we discuss the case when the inverse of the matrix in \eqref{ebalm} does not take a closed form or solving the corresponding linear system is difficult; see the earth mover's distance problem in Section \ref{section:EMD} for instance. In this case, we can use the block Gauss-Seidel method or the conjugate gradient method to inexactly solve the corresponding linear system. However,  the convergence issues are beyond the scope of this paper, and we refer the readers to \cite{jiang2021approximate, liu2021acceleration, jiang2021first} and the reference therein for some discussion on the inexact PDHG. As an alternative, we can adopt one block symmetric Gauss-Seidel  (sGS) iteration to solve the linear system inexactly. By the sGS decomposition theorem developed by  Li et al. \cite{li2019block}, this approach corresponds to taking $P$ as a specific positive definite matrix in \eqref{alg:ebalm}.   More specifically,  let $Q = \gamma \tau \K\K^{\tran}+\theta I_{m}$. Suppose that $Q$ takes the block structure
$$
Q = \begin{pmatrix} Q_{1,1} & \cdots & Q_{1,s} \\
\vdots & \vdots & \vdots \\
Q_{1,s}^{\tran} & \cdots & Q_{s,s}
\end{pmatrix},
$$
where $Q_{i,j} \in \Rbb^{m_{i}\times n_{j}}$ for $1\leq i,j\leq s$ and  $Q_{i,i}$ is symmetric positive definite and its inverse is easy to compute. Note that if $(\gamma \tau \K\K^{\tran})_{i,i}$ is positive definite, then $\theta$ can be chosen to be zero. Let
$$
U = \begin{pmatrix} 0 & Q_{1,2} &  \cdots & Q_{1,s} \\
  & \ddots &  & \vdots \\
 &  &  \ddots & Q_{s-1,s}\\
& & & 0
\end{pmatrix},
\quad D =  \begin{pmatrix} Q_{1,1} & &    &  \\
  & Q_{2,2} &  &  \\
 &  &  \ddots &  \\
& & & Q_{s,s}
\end{pmatrix}.
$$
Suppose $U \neq 0$, otherwise, the $y$-subproblem in \eqref{ebalm} takes closed form solution since the inverse of $Q_{i,i}$ is easy to compute. Taking $\tilde P = U D^{-1}U^{\tran}$, by \cite[Theorem 1]{li2019block}, we have $Q + \tilde P = (D+U)D^{-1}(D+U^{\tran}) \succ 0$ and that \eqref{alg:ebalm} with $$P = \theta I_{m} +  U D^{-1}U^{\tran}$$ is equivalent to the following procedure.

\vspace{4pt}
\hspace{-6mm}
\begin{boxedminipage}{1\textwidth}
\noindent \textbf{eBALM-sGS procedure:} Let $\tau > 0$, $\theta > 0$,  $\gamma \geq 3/4$ and $Q = \gamma \tau \K\K^{\tran}+\theta I_{m}$. For given $(x^k, y^k)$, the new iterate $(x^{k+1}, y^{k+1})$ is generated by:
\be \label{alg:ebalm:sGS}
\left\{
\begin{aligned}
x^{k+1}={}&\prox_{\tau f}\left(x^{k} - \tau \K^{\tran}y^{k}\right), \\
\bar b^{k+1} ={}& \K(2x^{k+1}-x^k)-b, \\
\bar y_{i}^{k+1}={}&y_{i}^k+Q_{i,i}^{-1}\bigg(\bar b_{i}^{k+1} - \sum_{j = 1}^{i-1}Q_{j,i}^{\tran} y_{j}^{k} - \sum_{j = i+1}^{s} Q_{i,j} \bar y_{j}^{k+1}\bigg), ~ i = s, \ldots, 2, \\
y_{i}^{k+1}={}&y_{i}^k + Q_{i,i}^{-1}\bigg(\bar b_{i}^{k+1} - \sum_{j = 1}^{i-1}Q_{j,i}^{\tran} y_{j}^{k+1} - \sum_{j = i+1}^{s} Q_{i,j} \bar y_{j}^{k+1}\bigg), ~ i = 1, \ldots, s, \\
\end{aligned}
\right.
\ee
where $\bar y_{i}^{k+1}, y_{i}^{k+1} \in \Rbb^{m_{i}}$ for $1 \leq i \leq s$ and $y^{k+1} = \begin{pmatrix} (y^{k+1}_{1})^\tran,  \cdots, (y^{k+1}_{s})^\tran \end{pmatrix}^\tran$.
\end{boxedminipage}
\vspace{4pt}

\noindent We name \eqref{alg:ebalm:sGS} as enhanced BALM with symmetric Gauss-Seidel iterations (eBALM-sGS) for solving problem \eqref{P1}. By Proposition \ref{prop:M1M2:general},  Theorem \ref{thm:convergence}
and Theorem \ref{thm:subconvergence}, we have the following results.

\begin{lemma}\label{lemma:sublinear:LCP}
Suppose $U\neq 0$. Let $\tau>0, \theta > 0$, and $\gamma \geq 3/4$.  Then the sequence $\{(x^{k},y^{k})\}$ generated by  eBALM-sGS \eqref{alg:ebalm:sGS} converges to an optimal solution of problem \eqref{P1}. Moreover, for $t \geq 1$, we have
\[
\begin{aligned}
&\min_{1 \leq k \leq t} \dist(0, \partial f(x^{k}) + \K^{\tran}y^{k}) = o\left(\frac{1}{\sqrt{t}}\right),  \\
 &\min_{1 \leq k \leq t} \dist(0, \partial f(x^{k}) + \K^{\tran}y^{k-1}) = o\left(\frac{1}{\sqrt{t}}\right),
\end{aligned}
\]
and
\[
 \min_{1 \leq k \leq t} \|\K x^{k} - b\| = o\left(\frac{1}{\sqrt{t}}\right). 
\]
If  $\gamma \geq 1$, $\theta > 0$,   the sublinear rate results are refined as
\[
  \dist(0, \partial f(x^{t}) + \K^{\tran}y^{t}) = o\left(\frac{1}{\sqrt{t}}\right), \quad   \dist(0, \partial f(x^{t}) + \K^{\tran}y^{t-1}) = o\left(\frac{1}{\sqrt{t}}\right)
\]
and
\[
 \|\K x^{t} - b\| = o\left(\frac{1}{\sqrt{t}}\right). 
\]
\end{lemma}
\begin{remark}
If the $i$-th block $(\gamma \tau \K\K^{\tran})_{i,i}$  is positive definite for any $1 \leq i \leq s$, then  $\theta > 0$ in the above lemma becomes $\theta \geq 0$.
\end{remark}

To end this subsection,  we consider a more general scenario that $g^{*}$ takes the block separable structure, i.e., $y = \begin{pmatrix} y_{1}^\tran, \ldots, y_{s}^\tran\end{pmatrix}^\tran$ and $g^{*}(y) = \sum_{j = 1}^{s} g_{j}(y_{j})$.  In this case, the $y$-subproblem in  PrePDHG \eqref{vmpdhg:2}  can be efficiently solved by cyclic proximal block coordinate descent method, see \cite{liu2021acceleration} for more details.

\section{Numerical Experiments}\label{S5}
In this section, we present plenty of numerical results on the matrix game, projection onto the Birkhoff polytope, earth mover's distance, and CT reconstruction problems to verify the superiority of the larger range of the corresponding parameters in our PrePDHG.
The codes are written in MATLAB (Release 2017b) and run
in macOS 10.15.4 on a MacBook Pro with a 2.9GHz Intel Core i7 processor with 16GB memory.

\subsection{Matrix Game}\label{subsection:matrixgame}
Let $\Delta_n = \{x\in \Rbb^n\mid \sum_{i = 1}^n x_i = 1, x \geq 0\}$ be the standard unit simplex in $\Rbb^n$. Given a matrix $K \in \Rbb^{m \times n}$, we consider the min-max matrix game
\be \label{prob:matrixgame}
\min_{x \in \Delta_n} \max_{y \in \Delta_m} \iprod{Kx}{y}.
\ee
This problem is a form of problem \eqref{PD} with  $f$ and $g^*$ chosen as the indicator functions of $\Delta_n$ and $\Delta_m$.
The main iterations of PDHG \eqref{ipdhg} are thus given as
 \begin{equation}\label{ipdhg:matrixgame}
\left\{
\begin{aligned}
x^{k+1}={}&\proj_{\Delta_n}\left(x^{k} - \tau \K^{\tran}y^{k}\right),\\
y^{k+1}={}&\proj_{\Delta_m}\left(y^{k} + \sigma \K(2x^{k+1} - x^k)\right),
\end{aligned}
\right.
\end{equation}
where $\proj_{\Delta_n}(\cdot)$ is the projection operator onto the simplex. For this problem, $\lambda_{\min}^f = 0$. By \eqref{equ:phdg:condition:new},  the stepsizes $\sigma> 0$ and $\tau> 0$ satisfy $\tau \sigma \|K\|^2 < 4/3$.  In our numerical results, we consider $\tau = \tilde \tau/\|K\|$ and $\sigma =  {1/}{(\gamma \tilde \tau \|K\|)}$ with  $\gamma \in \{1, 0.9, 0.85, 0.751\}$ (the requirement on $\gamma$ is $\gamma > 3/4$). Note that $\gamma = 1$ corresponds to the original PDHG method.

 By Remark \ref{remark:upperbound:R}, we stop the algorithm when the iterations exceed $10^6$ or
$\max\{\|\K^{\tran} (y^{k+1} - y^{k})-\tau^{-1}(x^{k+1} - x^{k})\|, \|\K(x^{k+1} - x^{k}) - \sigma^{-1}(y^{k+1} - y^{k})\|\}\leq 10^{-5}$.
The starting points  are always chosen as $x^0 = \frac1n \begin{pmatrix}1, \ldots, 1 \end{pmatrix}^\tran \in \Rbb^n$ and
$y^0 = \frac1m \begin{pmatrix}1, \ldots, 1 \end{pmatrix}^\tran \in \Rbb^m$.
We follow the way in \cite{malitsky2018first,chang2022golden} to generate the matrix $K$. The corresponding Matlab commands are given as:  i) \texttt{m = 100; n = 100; A = rand(m,n)}; ii) \texttt{m = 100; n = 100; A = randn(m,n)}; iii) \texttt{m = 500; n = 100; A = 10.*randn(m,n)}; iv) \texttt{m = 1000; n = 2000; A  = sprand(m,n,0.1)}.
For each case, we randomly generate the matrix $K$ 20 times and report the average performance of each algorithm.

\begin{figure}[!t]
\centering
\subfloat[Test 1]{\includegraphics[width=2.8cm,height=2.33cm]{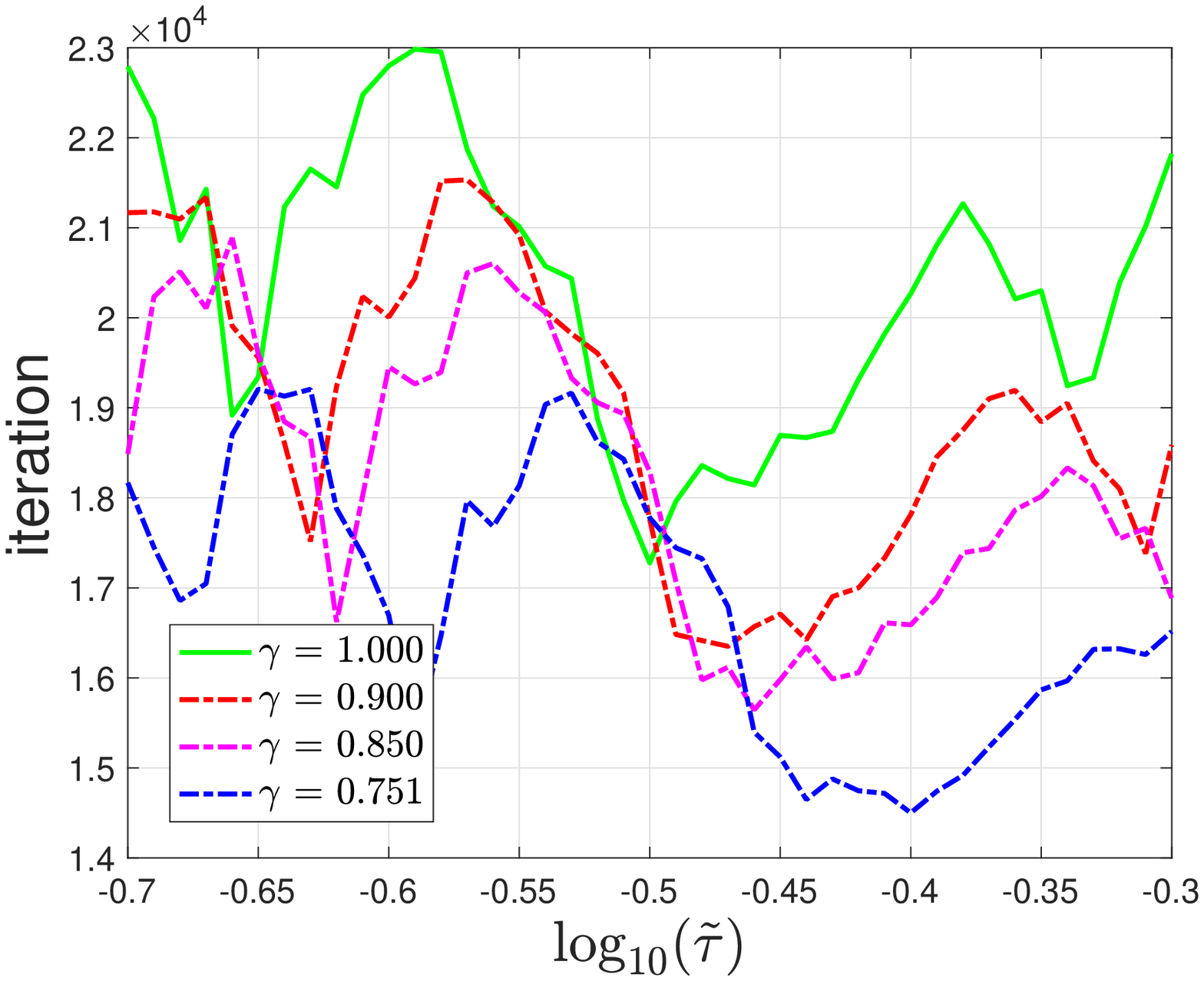}}~
\subfloat[Test 2]{\includegraphics[width=2.8cm,height=2.33cm]{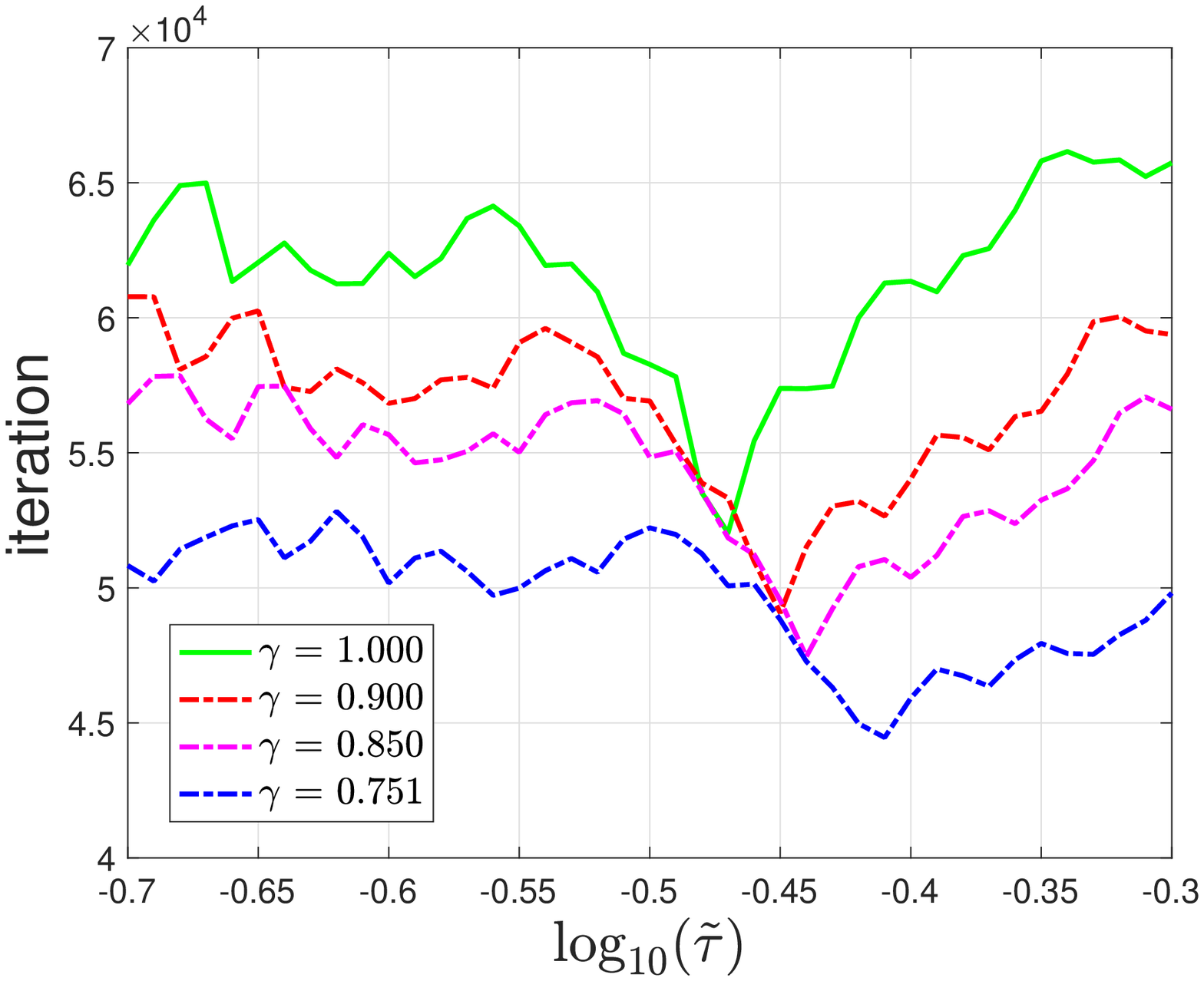}}~
\subfloat[Test 3]{\includegraphics[width=2.8cm,height=2.33cm]{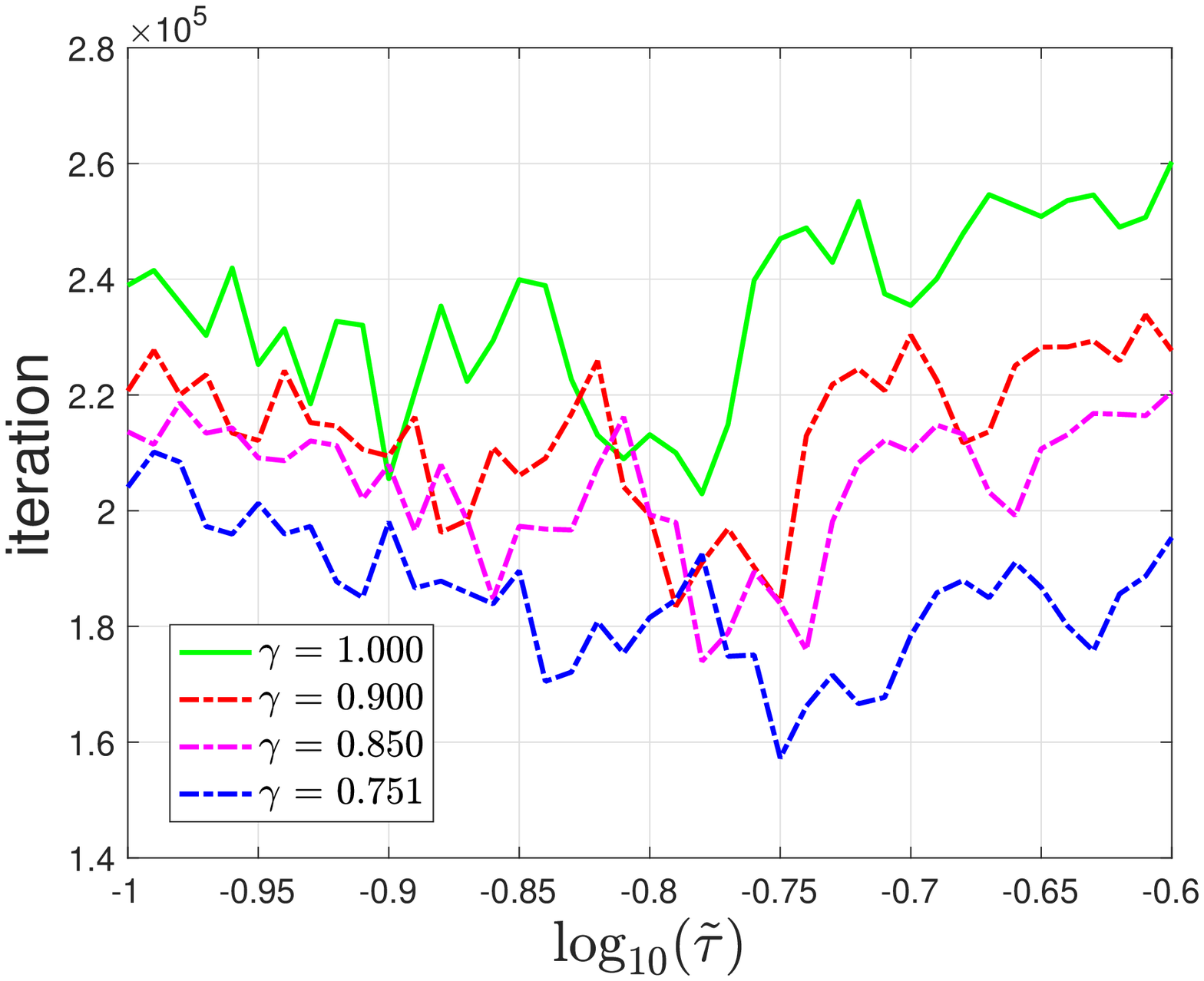}}~
\subfloat[Test 4]{\includegraphics[width=2.8cm,height=2.33cm]{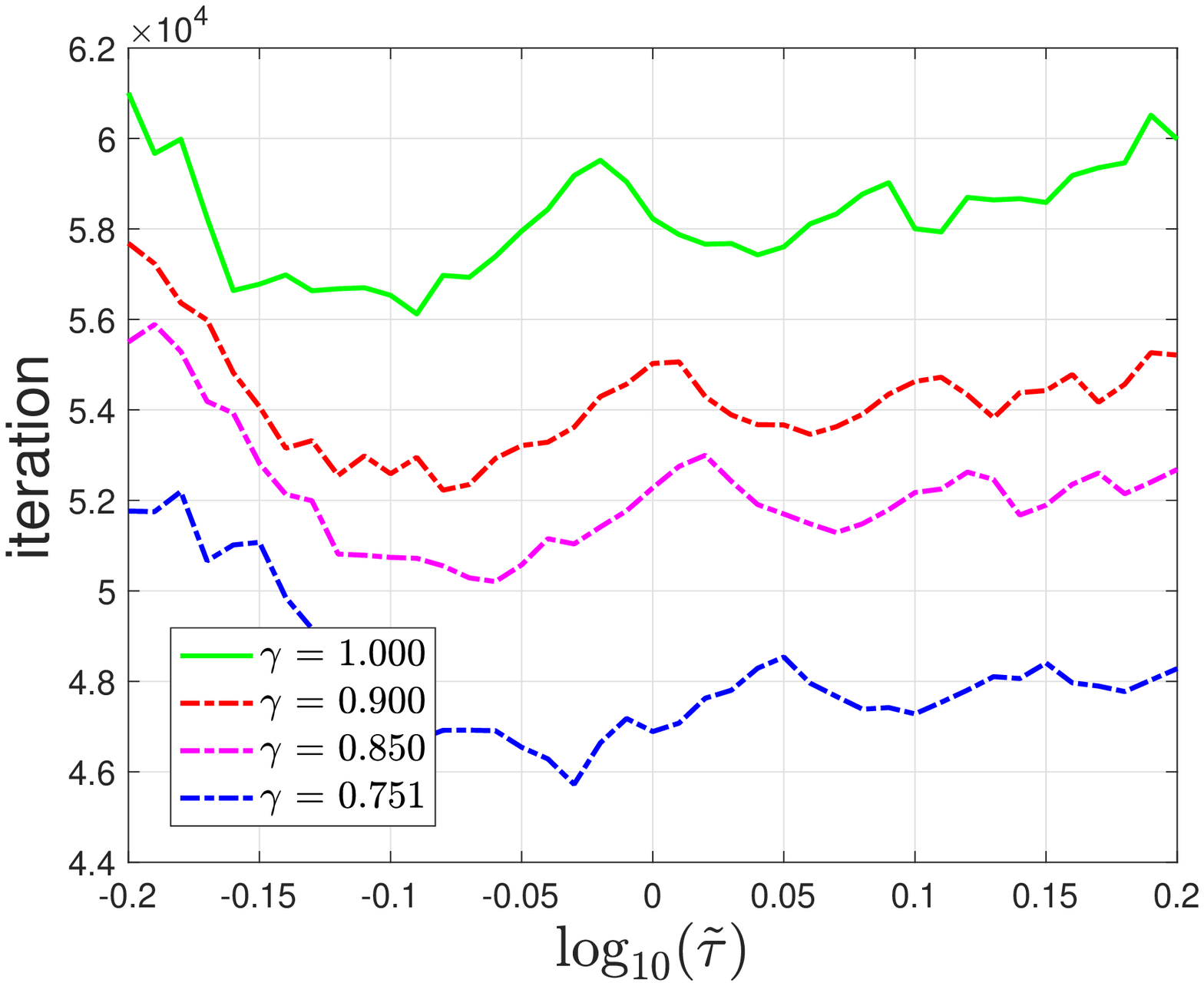}}\\
\subfloat[Test 1]{\includegraphics[width=2.8cm,height=2.33cm]{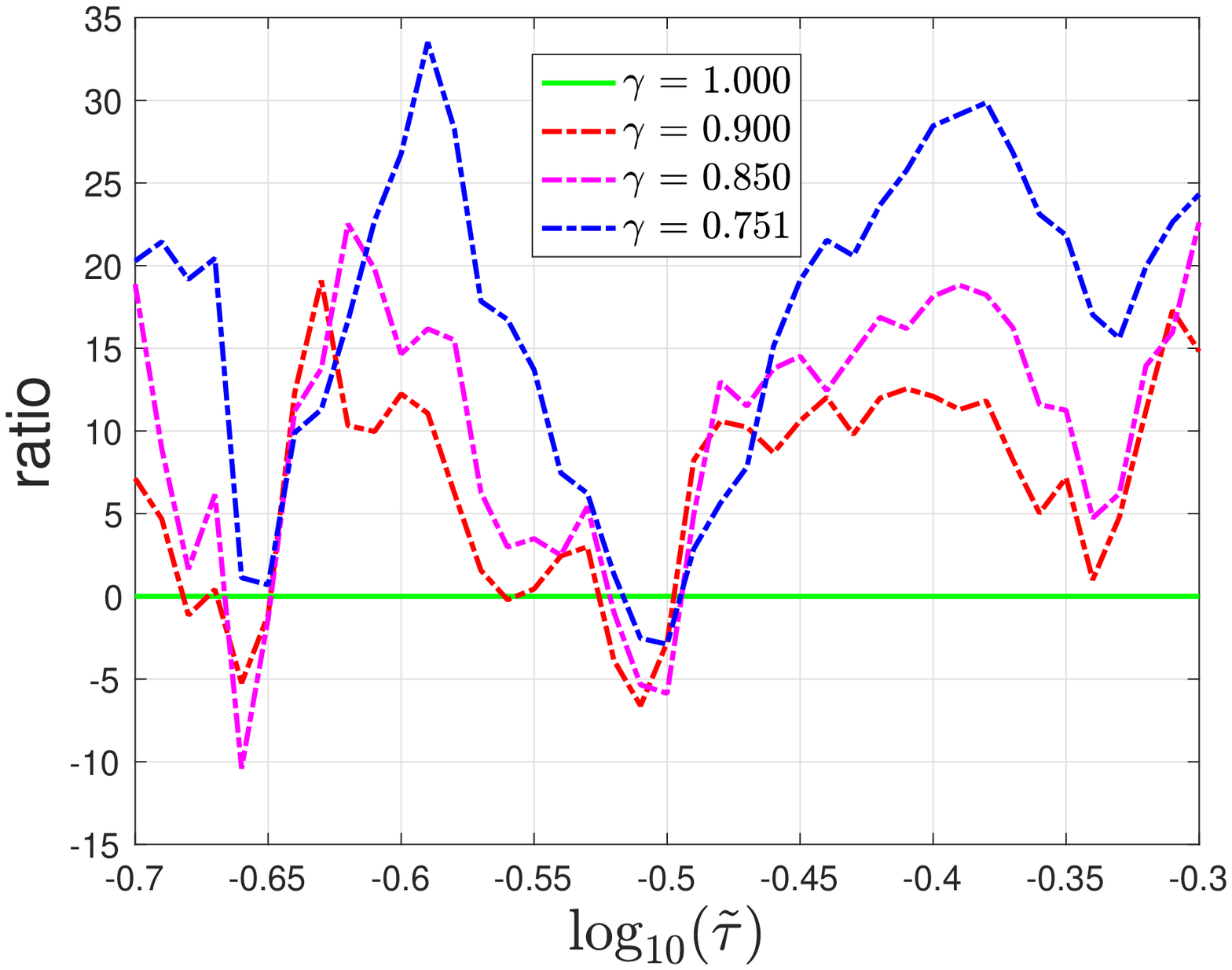}}~
\subfloat[Test 2]{\includegraphics[width=2.8cm,height=2.33cm]{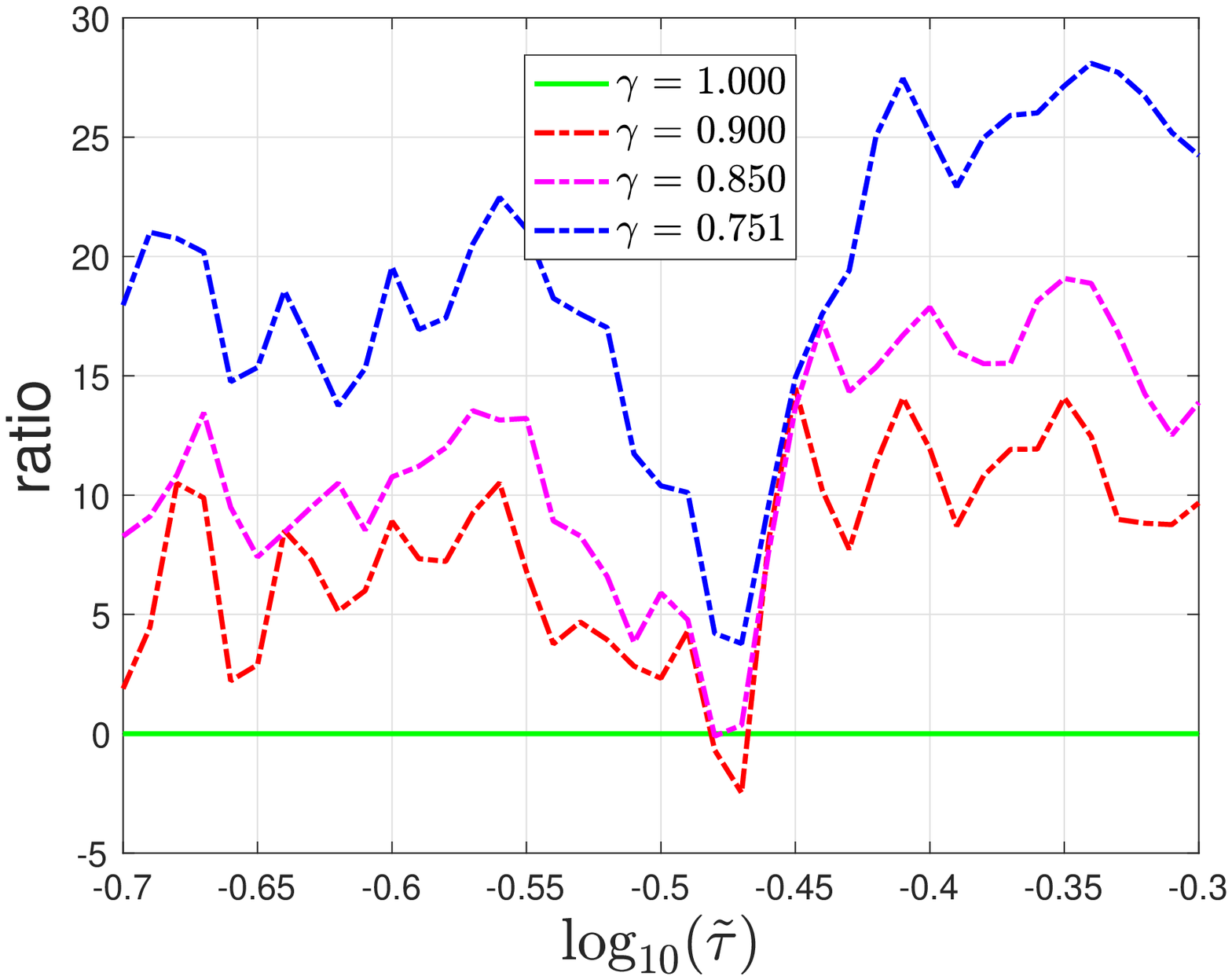}}~
\subfloat[Test 3]{\includegraphics[width=2.8cm,height=2.33cm]{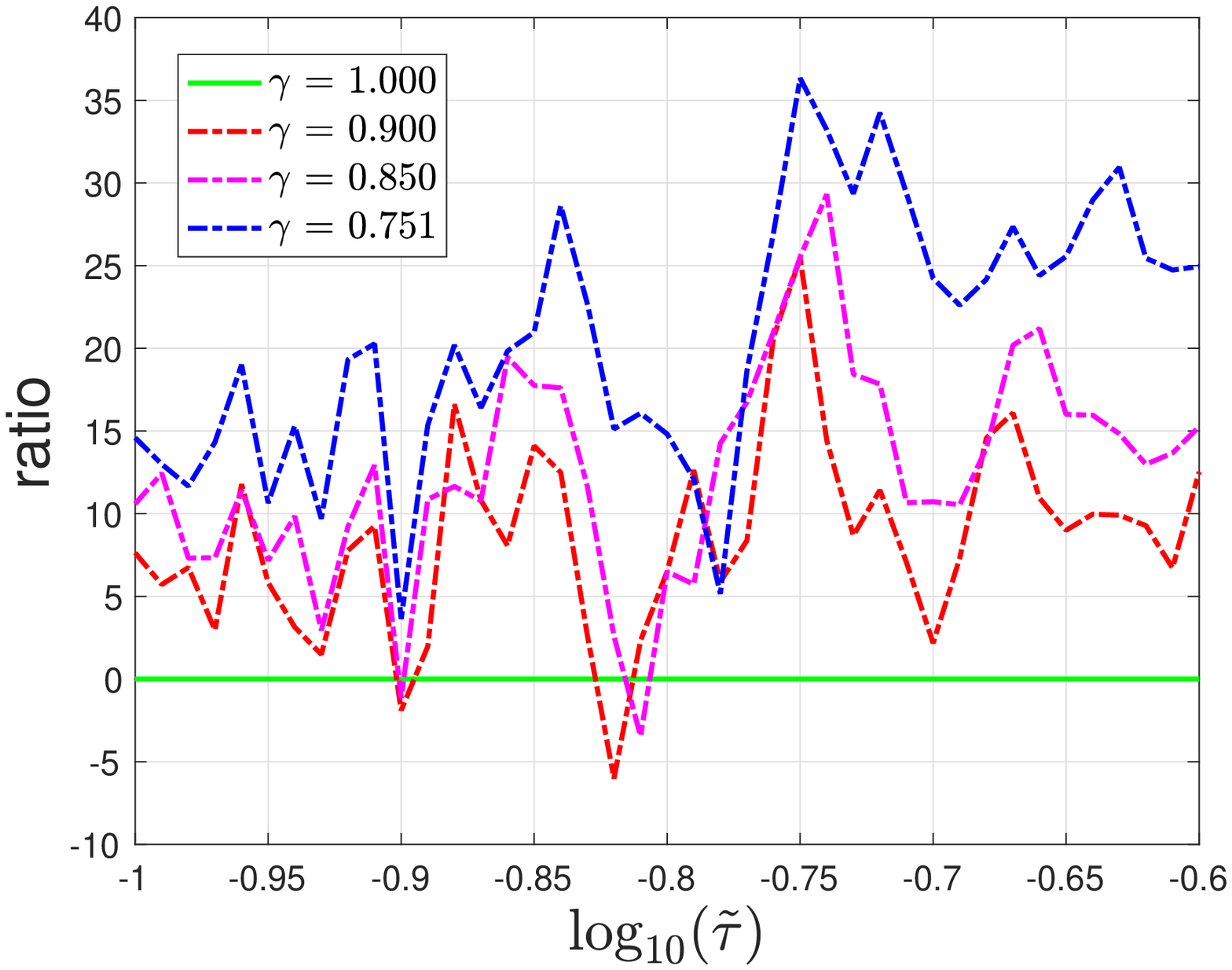}}~
\subfloat[Test 4]{\includegraphics[width=2.8cm,height=2.33cm]{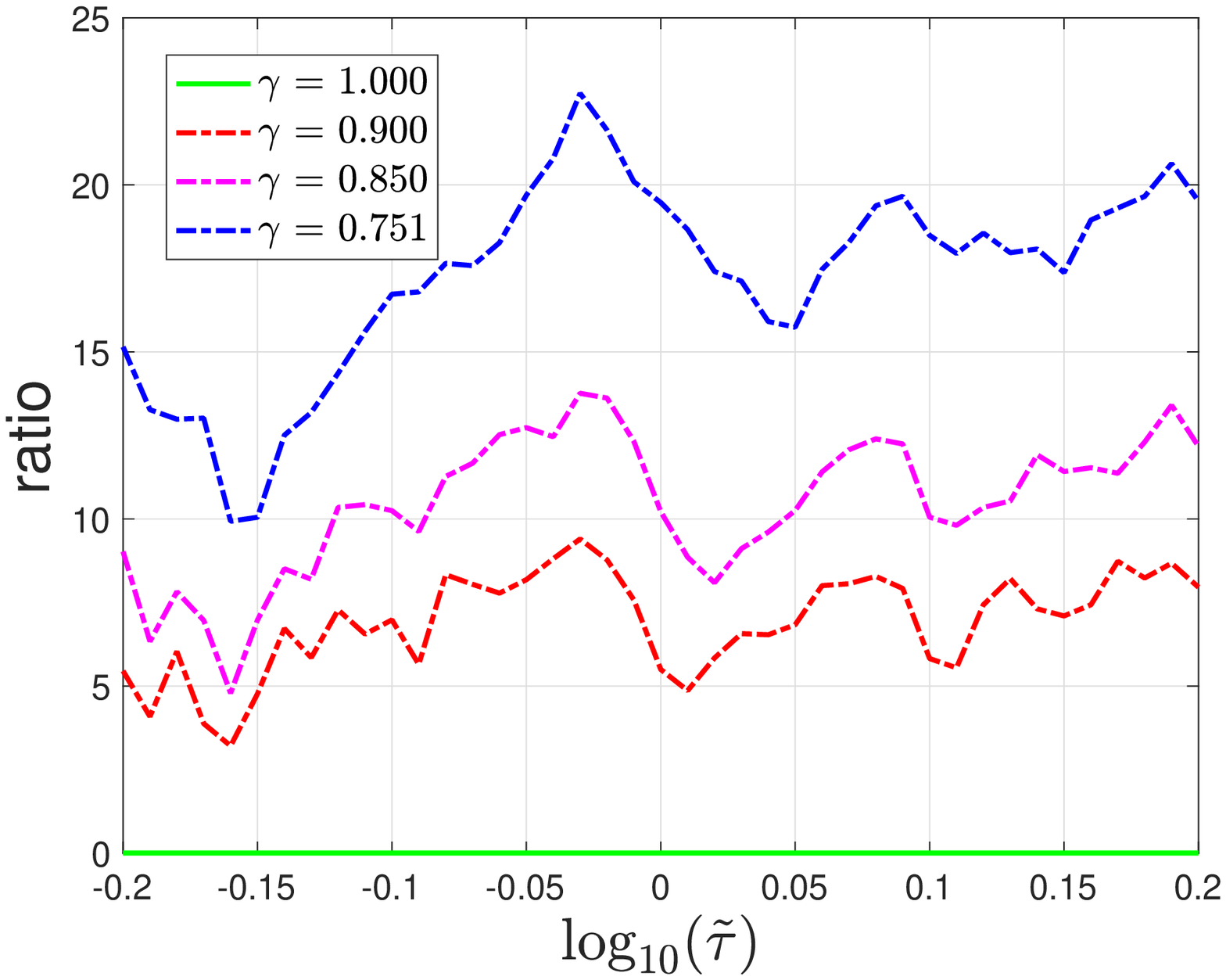}}\\
\caption{Comparison  of PDHG \eqref{ipdhg:matrixgame} with different values of $\gamma$ for matrix game problem \eqref{prob:matrixgame}.
}
\label{matrixgame:figures:ratio:iter}
\end{figure}

We test a series of  $\tilde \tau \in 10^a$ with $a = [a_1:0.01:a_1 + 0.4]$ with $a_1  = -0.7$ for Test 1 and Test 2, and $a_1 = -1.0$ for Test 3 and $a_1 = -0.2$ for Test 4. The comparison results among different
$\gamma$ are reported in Figure \ref{matrixgame:figures:ratio:iter}, wherein the saved ratio in terms of iteration number is defined as
\begin{equation}\label{equ:ratio}
\mathrm{ratio} = \left(\frac{\textnormal{\underline{iter}} - \textnormal{iter}}{\textnormal{\underline{iter}}} \times 100\right) \%,
\end{equation}
where the baseline iteration number ``$\textnormal{\underline{iter}}$'' is taken as the iteration number  of  PDHG with  $\gamma = 1$ and ``iter'' means the iteration number of PDHG with a chosen  $\gamma$.
From these figures, we can see that PDHG with smaller $\gamma$ always has better performance than the classical PDHG with $\gamma = 1$, and for a large range of $\tilde \tau$, the saved ratio is more than 20\% for Tests 1-3 and is more than 15\% for Test 4.  We also observe that the performance of PDHG with different $\gamma$ might depend on the choice of  $\tilde \tau$. Therefore, to make a fair comparison, for PDHG with fixed $\gamma$, we take the best $\tilde \tau$ (in terms of the lowest iteration number),  denoted by $\tilde \tau_{\mathrm{best}}$, from the set $10^{a}$.   The comparison results are shown in Table \ref{table:matrixgame}.  This table shows that PDHG with smaller $\gamma$ is still better than PDHG with $\gamma = 1$. For $\gamma = 0.751$, the saved ratio is always more than 22\%.
Note that such improvement only needs to change a parameter in the original PDHG without additional cost.

\begin{table}[!t]
\small
\setlength{\tabcolsep}{1.pt}
\centering
\caption{Performance of PDHG \eqref{ipdhg:matrixgame} with best $\tilde \tau$  for  problem \eqref{prob:matrixgame}.  In the table, ``a'', ``b'', ``c'',  and ``d'' stands for PDHG \eqref{ipdhg:matrixgame} with $\gamma = 1.0$, $\gamma = 0.90$, $\gamma = 0.85$, and $\gamma = 0.751$, respectively.}~\\
\label{table:matrixgame}
\begin{tabular}{@{}crrrrrrrrrrrrrrr@{}}
\toprule
&\multicolumn{4}{c}{$10 \times \log_{10}(\tilde \tau_{\mathrm{best}})$}  & \multicolumn{4}{c}{time}  & \multicolumn{4}{c}{iter} & \multicolumn{3}{c}{ratio \%}  \\
\cmidrule(lr){2-5} \cmidrule(lr){6-9} \cmidrule(lr){10-13}  \cmidrule(l){14-16}
Test & \multicolumn{1}{c}{a}  & \multicolumn{1}{c}{b}  & \multicolumn{1}{c}{c}   & \multicolumn{1}{c}{d}   &\multicolumn{1}{c}{a}   & \multicolumn{1}{c}{b}  & \multicolumn{1}{c}{c}    & \multicolumn{1}{c}{d} & \multicolumn{1}{c}{a}   & \multicolumn{1}{c}{b}   & \multicolumn{1}{c}{c}    & \multicolumn{1}{c}{d}   & \multicolumn{1}{c}{b}  & \multicolumn{1}{c}{c}   & \multicolumn{1}{c}{d} \\
\midrule
1& -5.0& -4.7& -4.6& -4.0& 1.6e-1& 1.5e-1& 1.5e-1& 1.3e-1& 17278& 16350& 15644& 14500& 10.2& 13.8& 28.5\\
2& -4.7& -4.5& -4.4& -4.1& 4.7e-1& 4.4e-1& 4.3e-1& 4.0e-1& 52040& 49075& 47458& 44458& 14.5& 17.3& 27.5\\
3& -7.8& -7.9& -7.8& -7.5& 2.9e1& 2.6e1& 2.5e1& 2.3e1& 202919& 183475& 173976& 157250& 12.6& 14.3& 36.3\\
4& -0.9& -0.8& -0.6& -0.3& 1.7e1& 1.6e1& 1.6e1& 1.4e1& 56122& 52226& 50205& 45723&  8.3& 12.5& 22.7\\
\bottomrule
\end{tabular}
\end{table}

\subsection{Projection onto the Birkhoff Polytope}\label{subsection:Birkhoff}
Given a matrix $C \in \Rbb^{n \times n}$,  computing its projection onto the Birkhoff polytope can be formulated as
\be\label{prob:Birk:proj}
\min_{X\in \mathcal{B}_n}\,  \frac12 \|X - C\|_{\Fsf}^2,
\ee
where $\mathcal{B}_n:=\{X \in \Rbb^{n \times n} \mid X \1_n = \1_n, X^{\tran}\1_n = \1_n,  X \geq 0\}$ with $\1_n \in \Rbb^n$ being the all-one vector is known as the Birkhoff polytope. Problem \eqref{prob:Birk:proj} has wide applications in solving the optimization problems involving permutations; see \cite{jiang2016lp,li2019block} and the references therein for more details. Let $x = \mathrm{vec}(X)$, problem \eqref{prob:Birk:proj} can be seen as a special instance of \eqref{P1} with
$f(x) = \frac12 \|x - \mathrm{vec}(C)\|^2 + \mathbb{I}_{\mathcal{X}}$ with $\mathcal{X} = \Rbb_+^{n^2}$ and  $\mathbb{I}_{\mathcal{X}}$ being the indicator function of the set $\mathcal{X}$, and
$K = \begin{pmatrix}
\1_n^\tran \otimes I_n \\ I_n \otimes \1_n^\tran\end{pmatrix},
b = \1_{2n}$, where $\otimes$ is the Kronecker product.  For such $K$, we have  $\|K\|^2 = 2n$ (see \cite{he2022generalized} for instance) and
$$
\left(KK^\tran+ \theta I_{2n}\right)^{-1} = \frac{1}{n + \theta} I_{2n} +
\frac{1}{2n \theta + \theta^2}\begin{pmatrix}
 \frac{n}{n + \theta} \1_n\1_n^\tran & -\1_n\1^\tran \\
-\1_n\1_n^\tran &  \frac{n}{n + \theta} \1_n\1_n^\tran
\end{pmatrix}, \quad \theta > 0.
$$

We consider two particular choices of PrePDHG \eqref{alg:ebalm}.
The first one is eBALM \eqref{ebalm}, whose main iterations  are given as:
\begin{equation}\label{pdhg:Birk:ebalm}
\left\{
\begin{aligned}
X^{k+1}={}&\frac{1}{1 + \tau}\proj_{+}\left(X^k + \tau C  - \tau \left( y_1^k \1_n^\tran + \1_n( y_2^k)^\tran\right)\right),\\
a^{k+1} ={}& \1_n^\tran (2 X^{k+1} - X^k)\1_n  + n + \theta, \\
y^{k+1}={}&  y^k +  \frac{1}{\gamma \tau (n + \theta)}
\begin{pmatrix}
   (2X^{k+1} - X^k) \1_n \\
    (2X^{k+1} - X^k)^\tran \1_n
   \end{pmatrix}
 -  \frac{a^{k+1}}{\gamma \tau (n + \theta)(2n  + \theta)} \1_{2n},
\end{aligned}
\right.
\end{equation}
where $\proj_{+}(\cdot)$ is the projection operator over $\Rbb_+^{n \times n}$ and $\theta$ is taken as $10^{-4}$,  $y_1^k \in \Rbb^n$ is the vector formulated by the first $n$ components of $y^k$ and $y_2^k \in \Rbb^n$ is the vector formulated by the last $n$ components of $y^k$.
The second one is PDHG \eqref{ipdhg} with main iterations given as:
\begin{equation}\label{pdhg:Birk:proj}
\left\{
\begin{aligned}
X^{k+1}={}&\frac{1}{1 + \tau}\proj_{+}\left(X^k + \tau C  - \tau \left( y_1^k \1_n^\tran + \1_n( y_2^k)^\tran\right)\right),\\y^{k+1}={}&  y^k +  \sigma \begin{pmatrix}
   (2X^{k+1} - X^k) \1_n \\
    (2X^{k+1} - X^k)^\tran \1_n
   \end{pmatrix}
     - \sigma \1_{2n}. \\
\end{aligned}
\right.
\end{equation}

Note that for  problem \eqref{prob:Birk:proj}, we have $\Sigma_f = I_{n^2}$.   By Lemma \ref{lemma:M1M2},   we know that the parameters $\tau > 0$ and $\gamma > 0$ in \eqref{pdhg:Birk:ebalm} should satisfy  $\gamma \geq  \frac{0.75}{1 + \tau/2}$. In our numerical results, we consider
 $\tau = \tilde \tau /\sqrt{2n}$ with $\tilde \tau>0$ and $\gamma \in \big\{1, \frac{0.75}{1 + \tau/2}\big\}$.
 In addition,  by \eqref{equ:phdg:condition:new}, the parameters  $\tau > 0$  and $\sigma > 0$ in \eqref{pdhg:Birk:proj} satisfies  $2 n \tau \sigma < \frac43 (1 + \frac{\tau}{2})$. In our numerical results, we consider $\tau = \tilde \tau /\sqrt{2n}$ and $\sigma = 1/(\gamma \tilde \tau \sqrt{2n})$ with $\gamma \in \big\{1, \frac{0.751}{1 + \tau/2}\big\}$.  For a given $n$, we follow the way in \cite{li2019block}  to randomly generate 20 matrices $C$ via  \texttt{C = rand(n); C = (C+C')./2} and report the average performance.  The initial points are always chosen as $X^0 = \frac1n \1_n\1_n^\tran$ and $y^0 = 0$.    By \eqref{equ:Rxkyk-1:Kxb}, we stop both algorithms  when the relative KKT residual
$\widetilde \Rcal^{k} := \max\{d^{k},  p^{k}\} \leq 10^{-8}$
 with $p^k = \tau^{-1}\|X^{k} - X^{k-1}\|_{\Fsf}$ and
 $d^k =
  \left\|
  \begin{pmatrix}
  X^{k}\1_n - \1_n \\
  (X^{k})^\tran\1_n - \1_n
 \end{pmatrix}
 \right\|$.

For both algorithms, we test a series of  $\tilde \tau \in 10^a$ with $a = [0.2: 0.01: 0.6]$.  The comparison results are depicted in Figure \ref{Birkhoff:figures:ratio:iter}.   In the figures (e)-(h),  the ``ratio'' is computed according to \eqref{equ:ratio} with  $\textnormal{\underline{iter}}$ taken as the iteration number of PDHG \eqref{pdhg:Birk:proj} with $\gamma = 1$, and in the figures   (i)-(l), the ``ratio'' is computed according to \eqref{equ:ratio} with  $\textnormal{\underline{iter}}$ taken as the iteration number of eBALM \eqref{pdhg:Birk:ebalm} with $\gamma = 1$.  From these figures, we can draw the following observations. (i) Both PDHG and eBALM benefit from choosing a larger stepsize, namely, a smaller $\gamma$. For a large range of $\tilde \tau$, PDHG with $\gamma = \frac{0.751}{1 + \tau/2}$ is more than 30\% faster than PDHG with $\gamma = 1$ and eBALM with $\gamma = \frac{0.75}{1 + \tau/2}$ is more than 15\% faster than eBALM with $\gamma = 1$.  (ii) eBALM with $\gamma = \frac{0.75}{1 + \tau/2}$ performs best among the four algorithms, and it is even about more than 50\% faster than the classical PDHG with $\gamma = 1$ for $\tilde \tau \in 10^{[0.35, 0.6]}$.

\begin{figure}[!t]
\centering
\subfloat[$n = 200$]{\includegraphics[width=2.8cm,height=2.33cm]{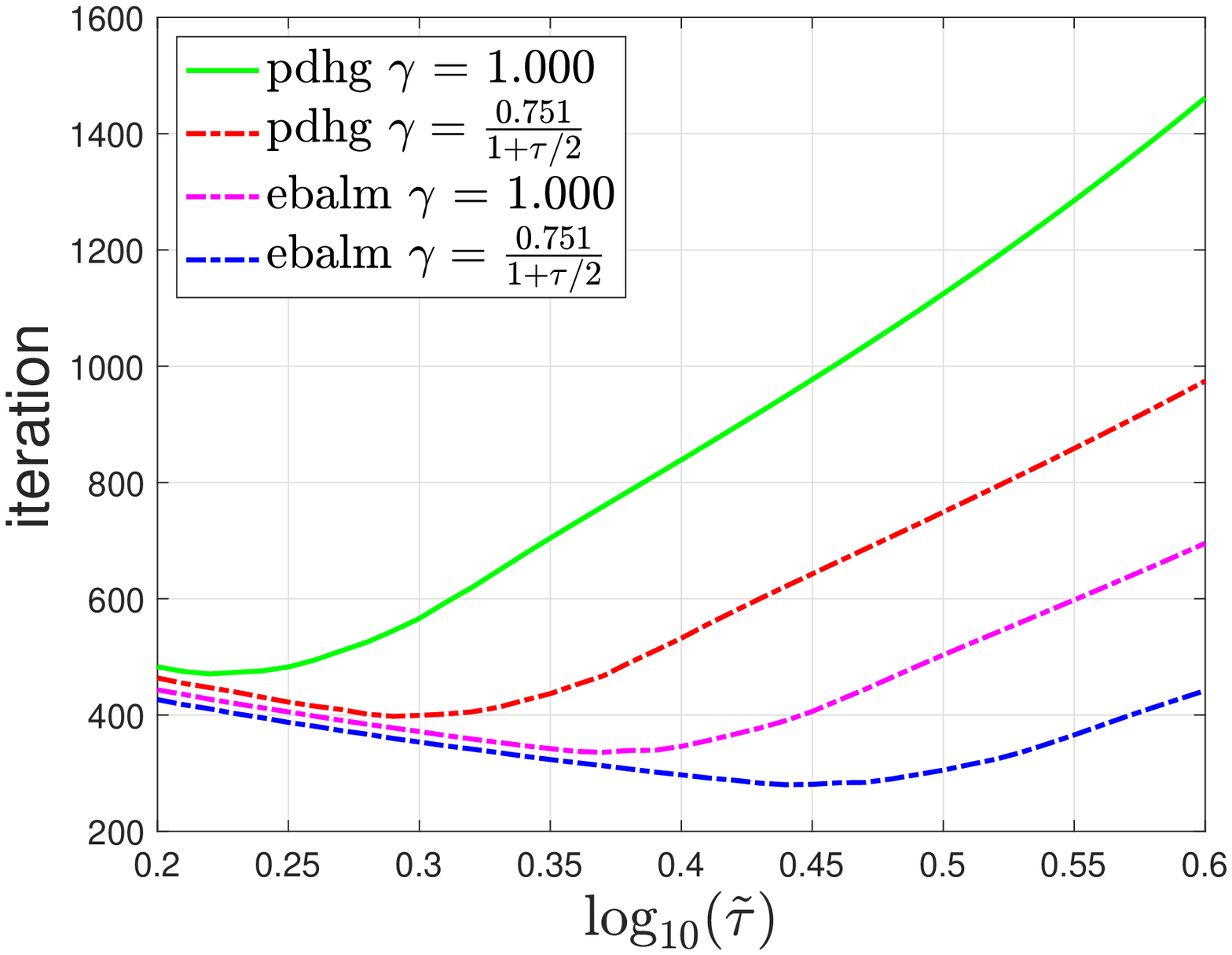}}~
\subfloat[$n = 400$]{\includegraphics[width=2.8cm,height=2.33cm]{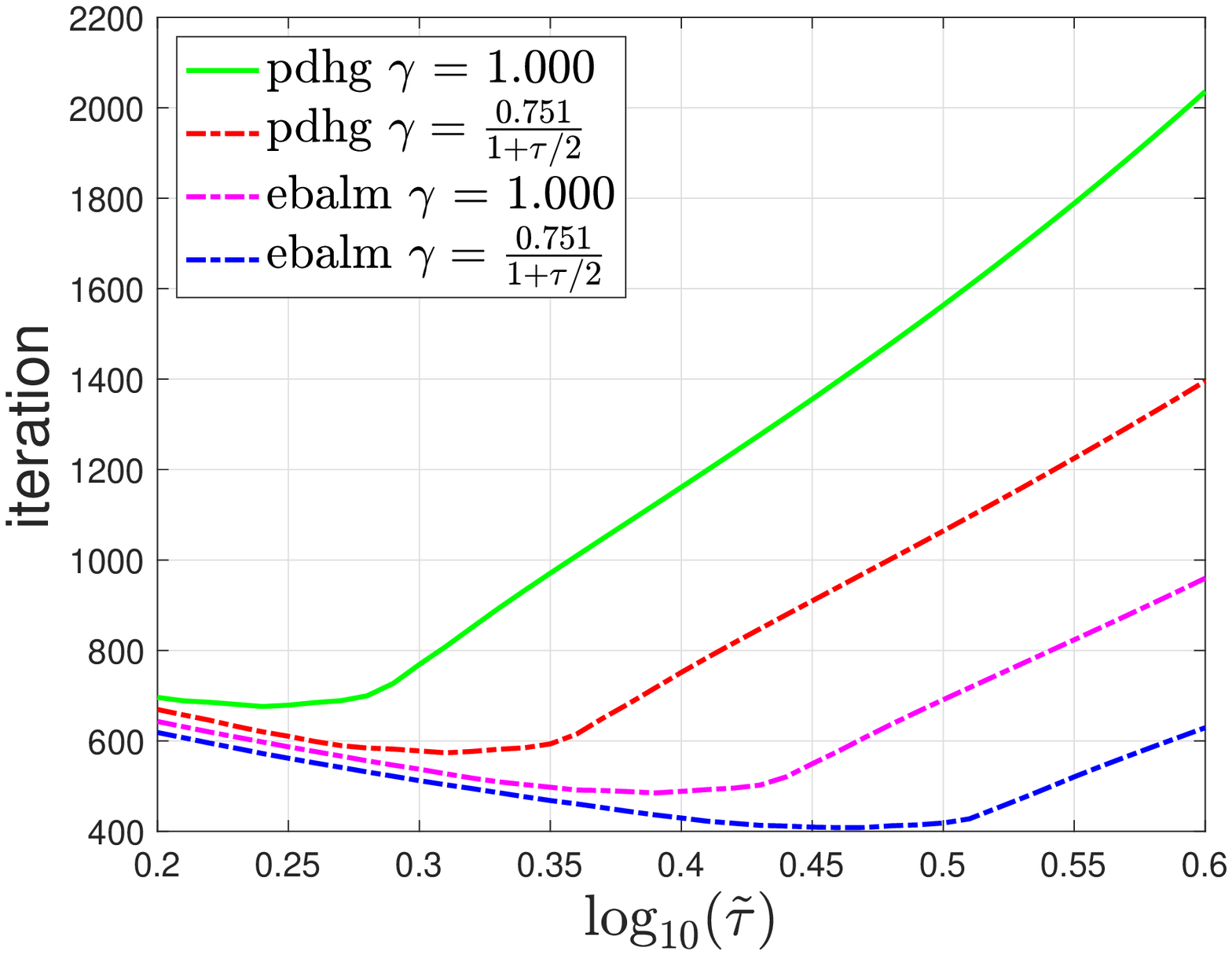}}~
\subfloat[$n = 600$]{\includegraphics[width=2.8cm,height=2.33cm]{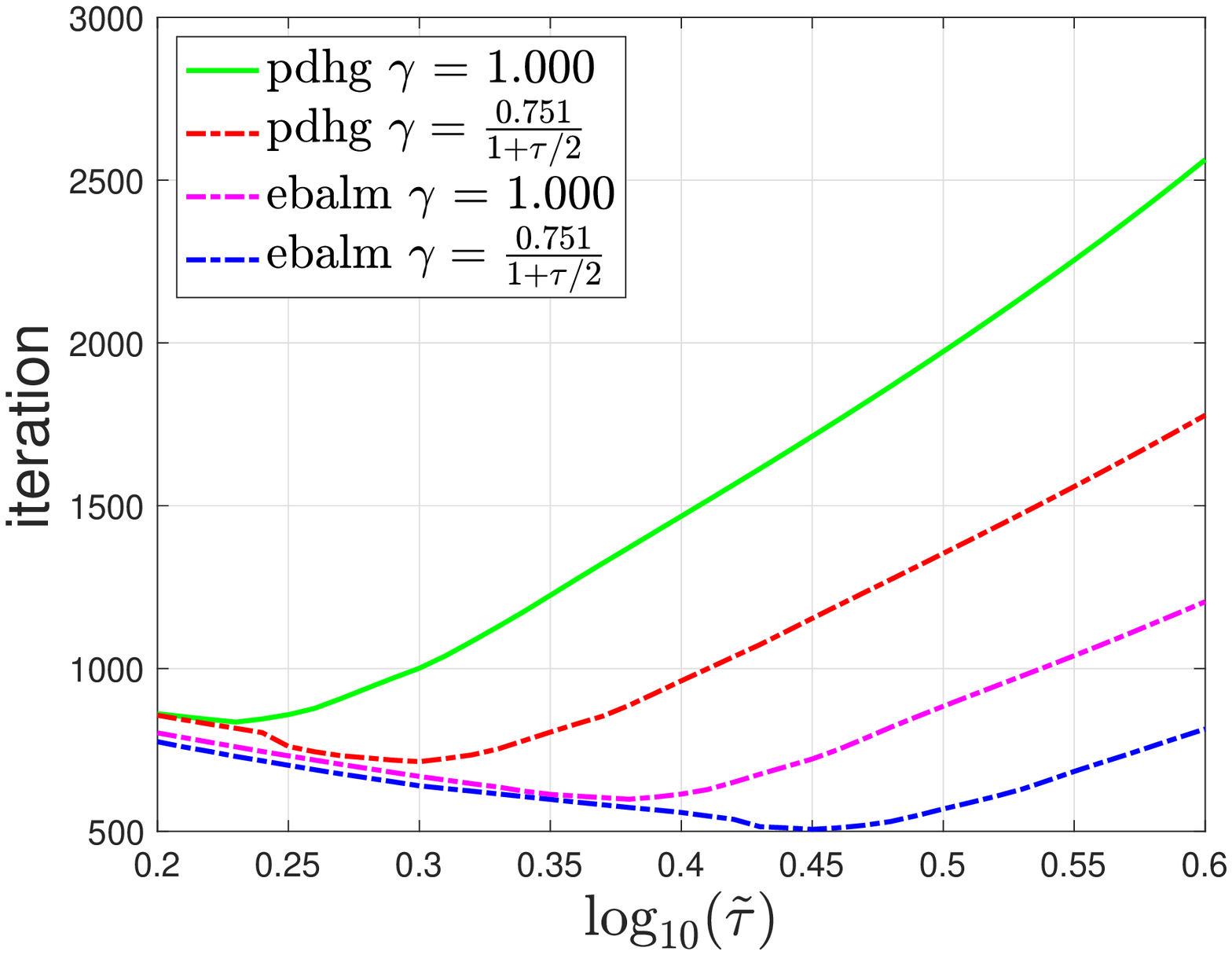}}~
\subfloat[$n = 800$]{\includegraphics[width=2.8cm,height=2.33cm]{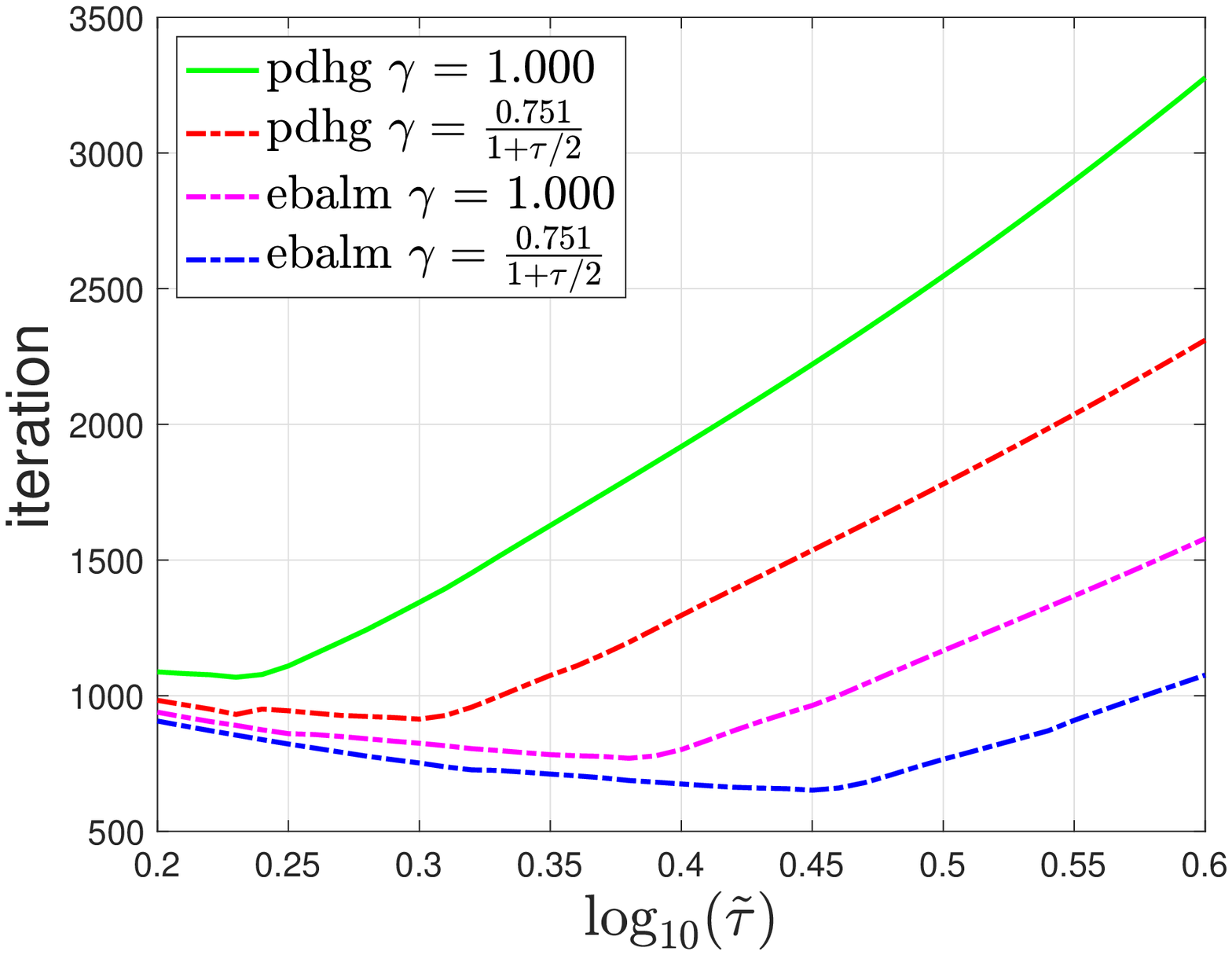}}\\
\subfloat[$n = 200$]{\includegraphics[width=2.8cm,height=2.33cm]{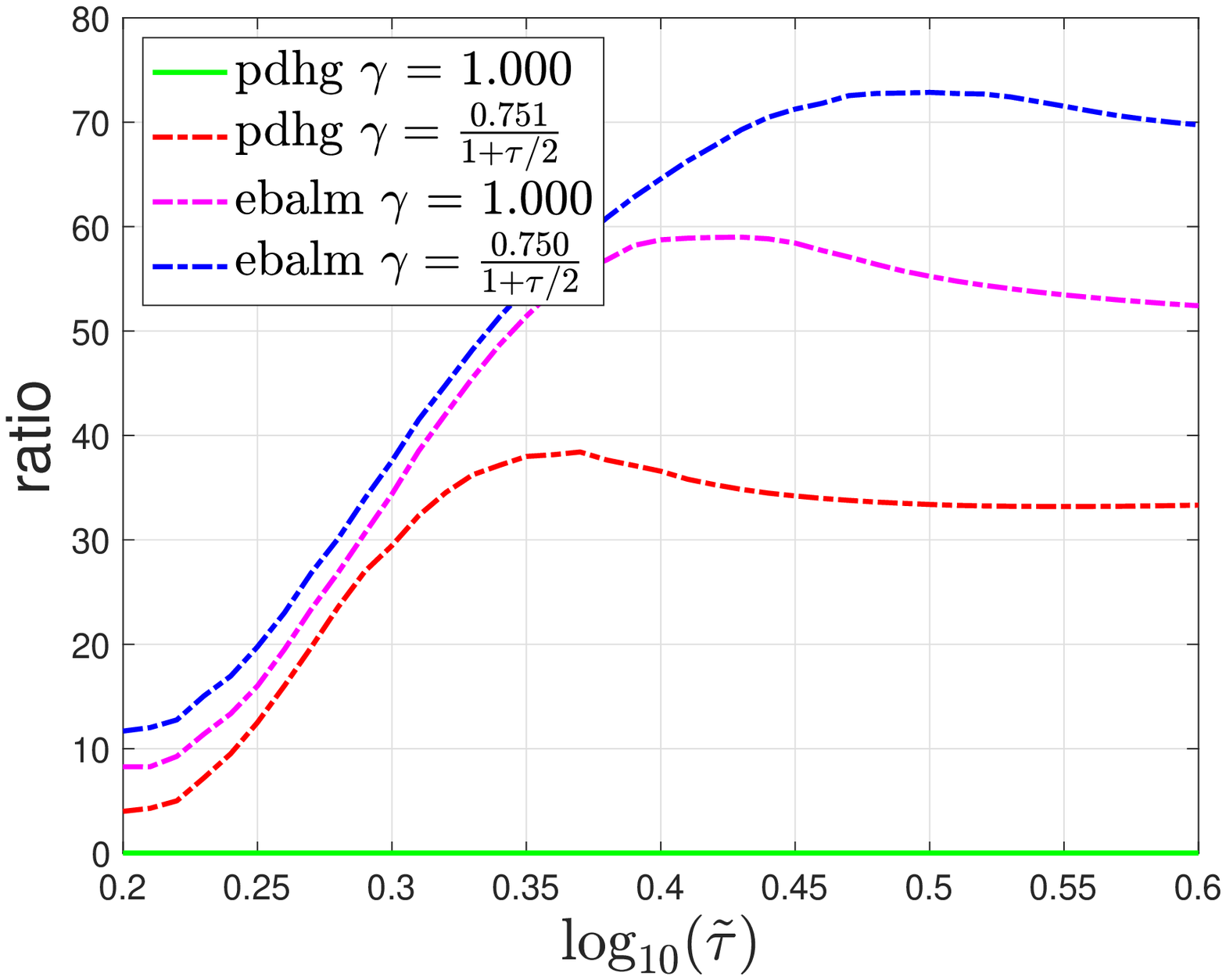}}~
\subfloat[$n = 400$]{\includegraphics[width=2.8cm,height=2.33cm]{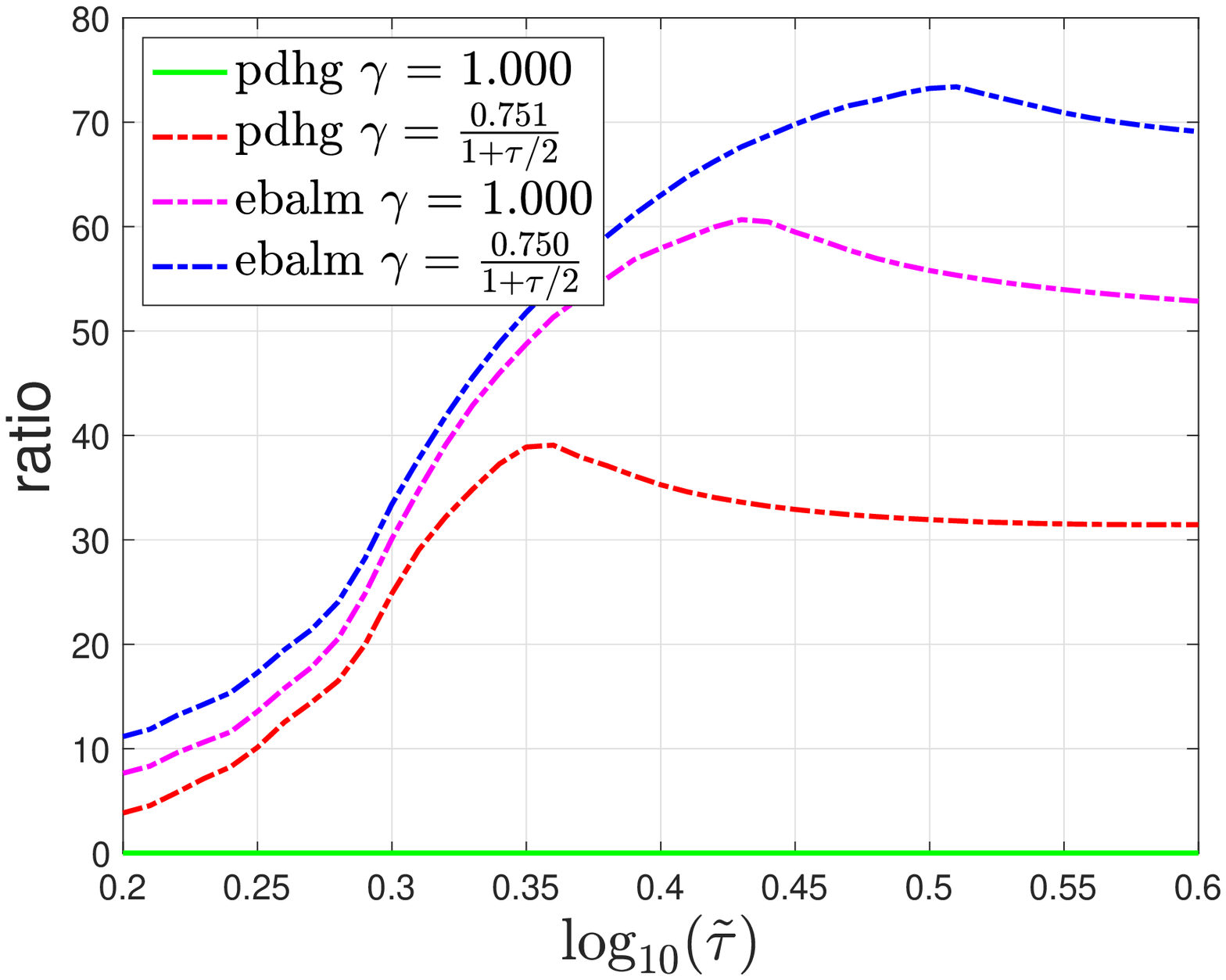}}~
\subfloat[$n = 600$]{\includegraphics[width=2.8cm,height=2.33cm]{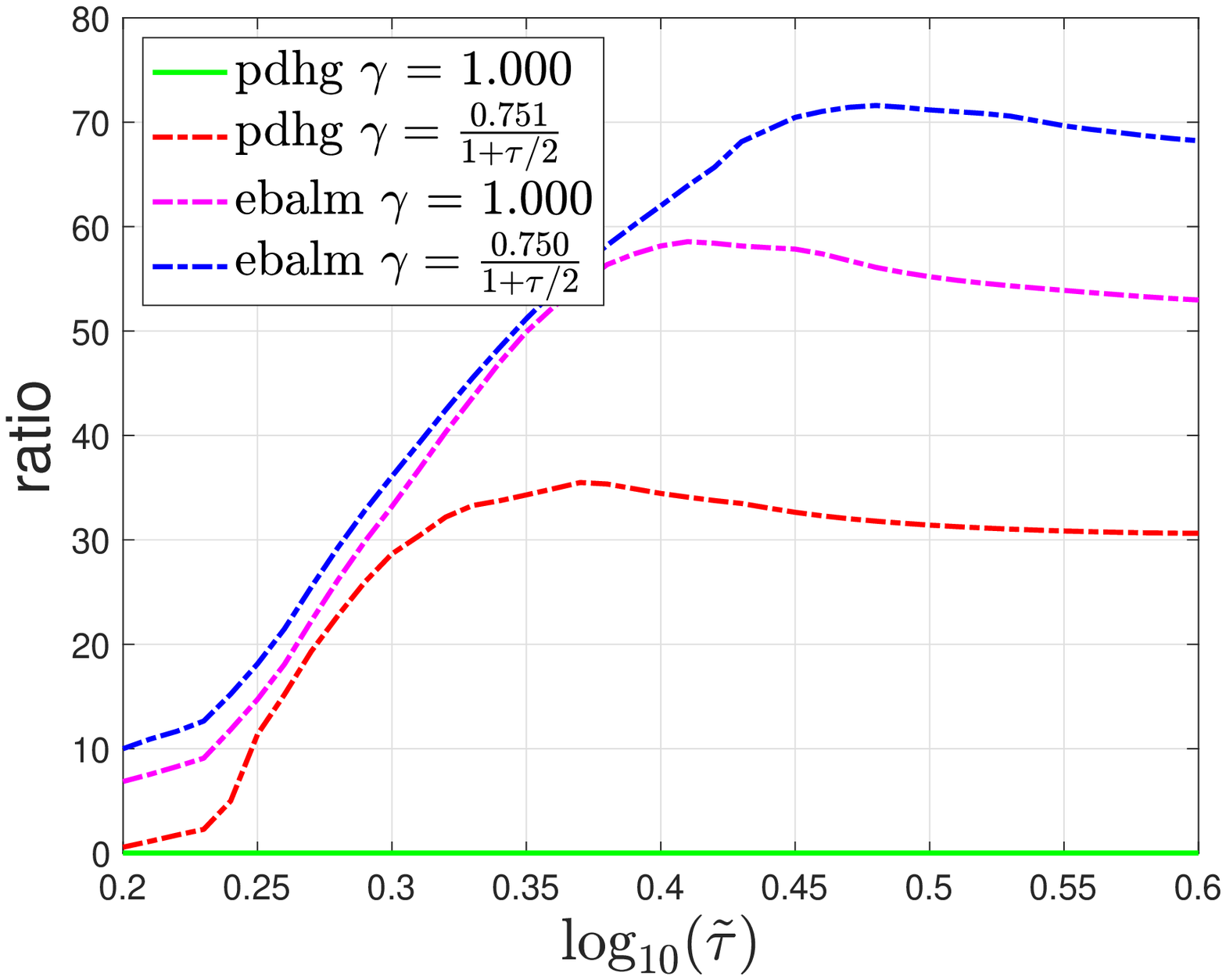}}~
\subfloat[$n = 800$]{\includegraphics[width=2.8cm,height=2.33cm]{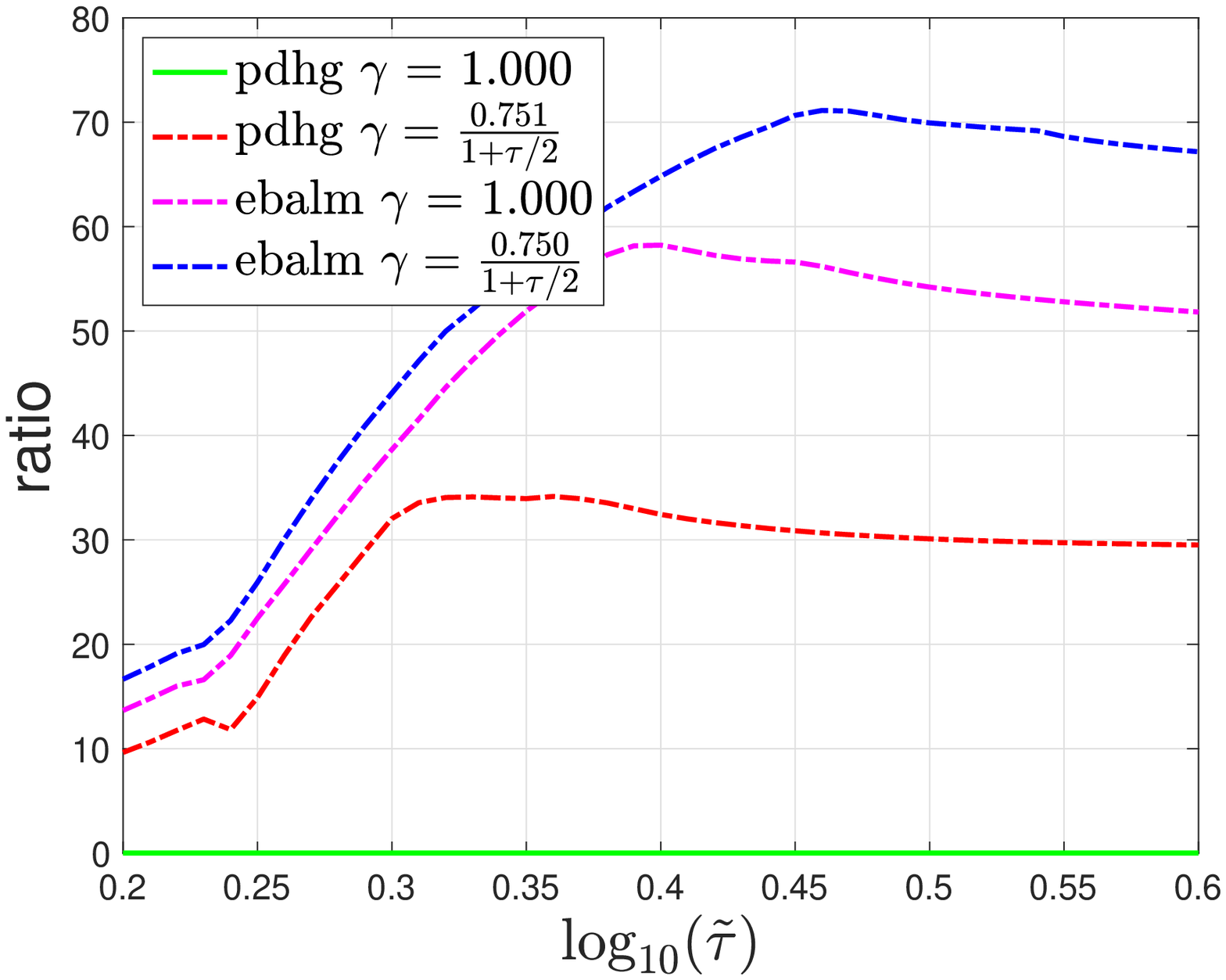}}\\
\subfloat[$n = 200$]{\includegraphics[width=2.8cm,height=2.33cm]{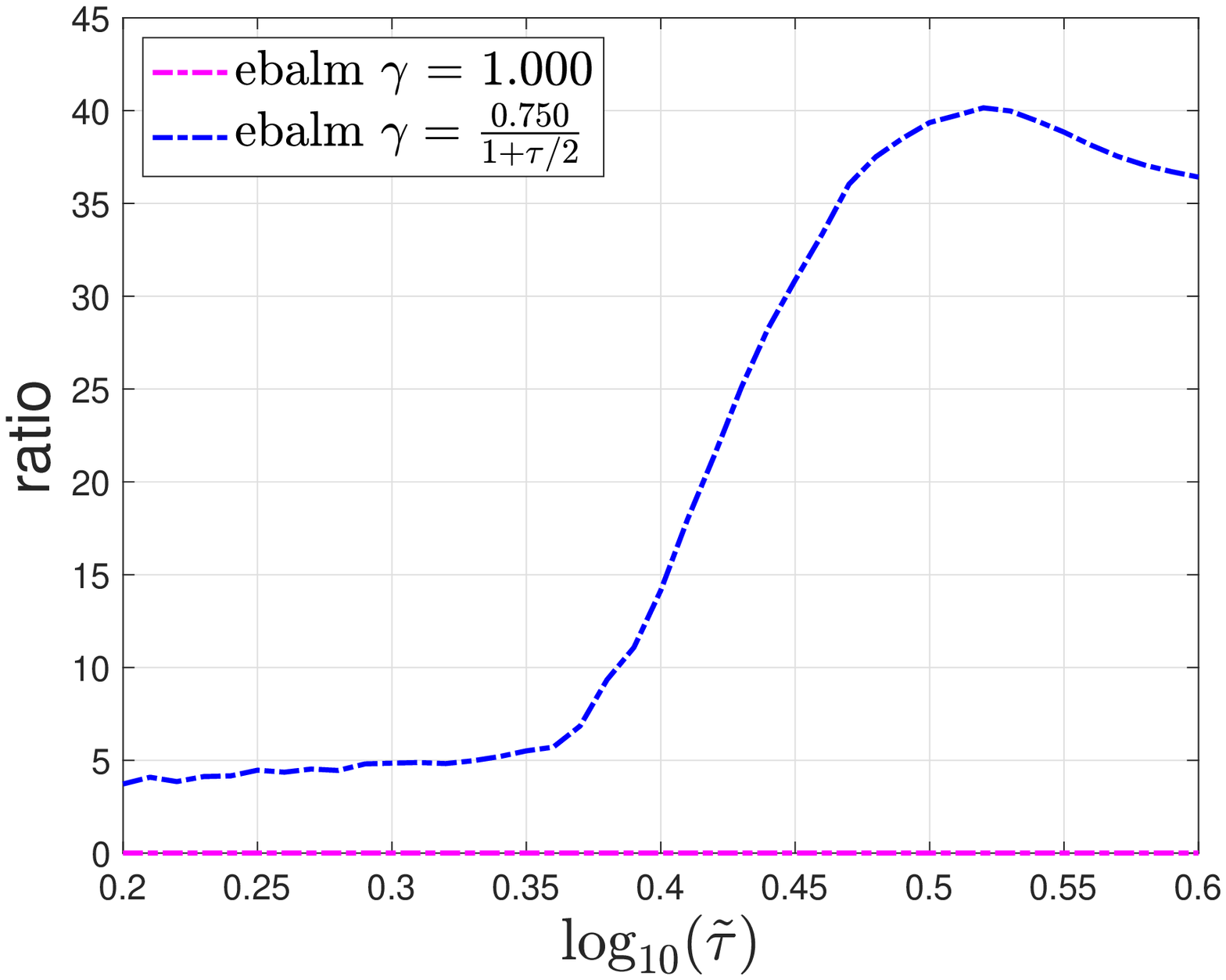}}~
\subfloat[$n = 400$]{\includegraphics[width=2.8cm,height=2.33cm]{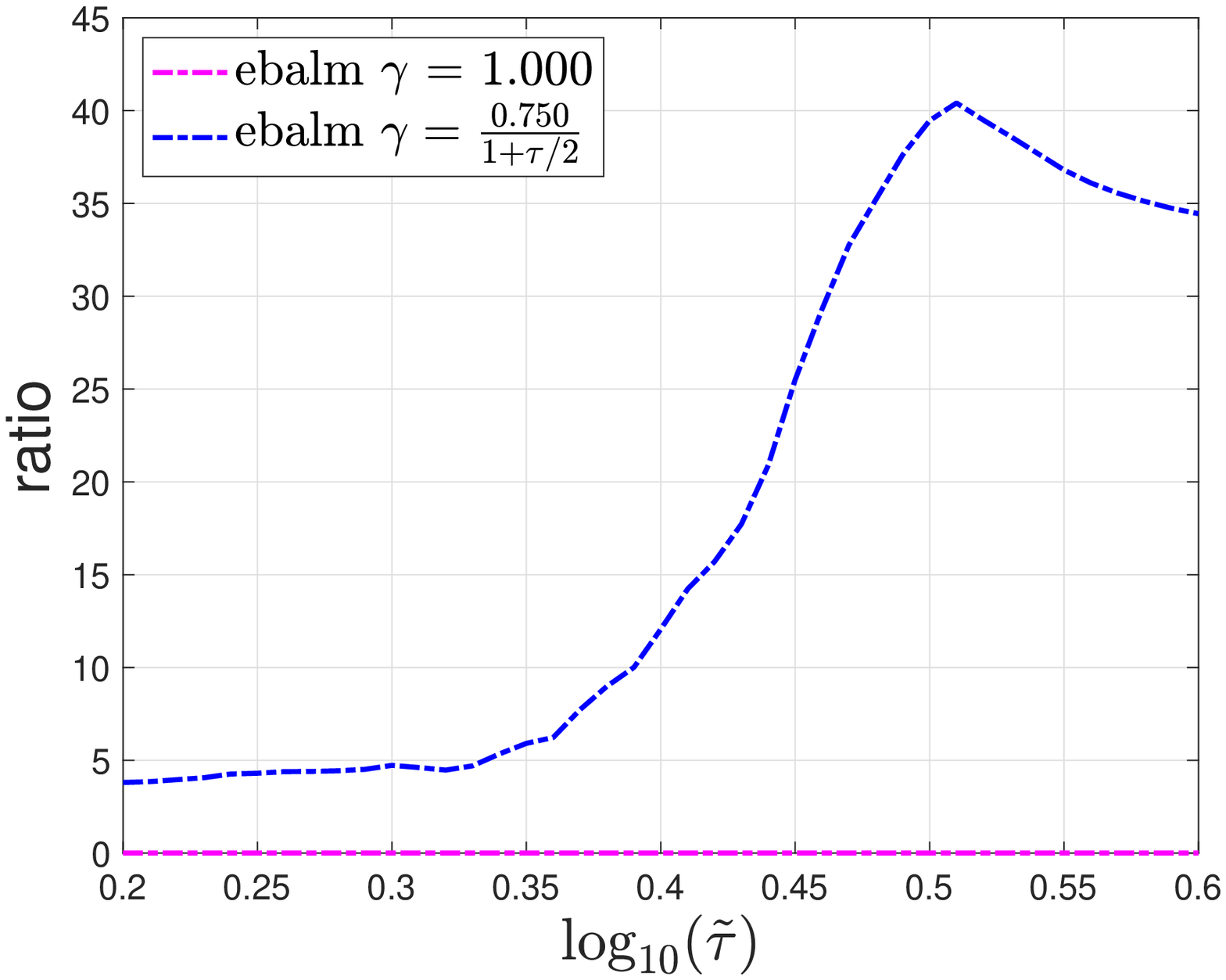}}~
\subfloat[$n = 600$]{\includegraphics[width=2.8cm,height=2.33cm]{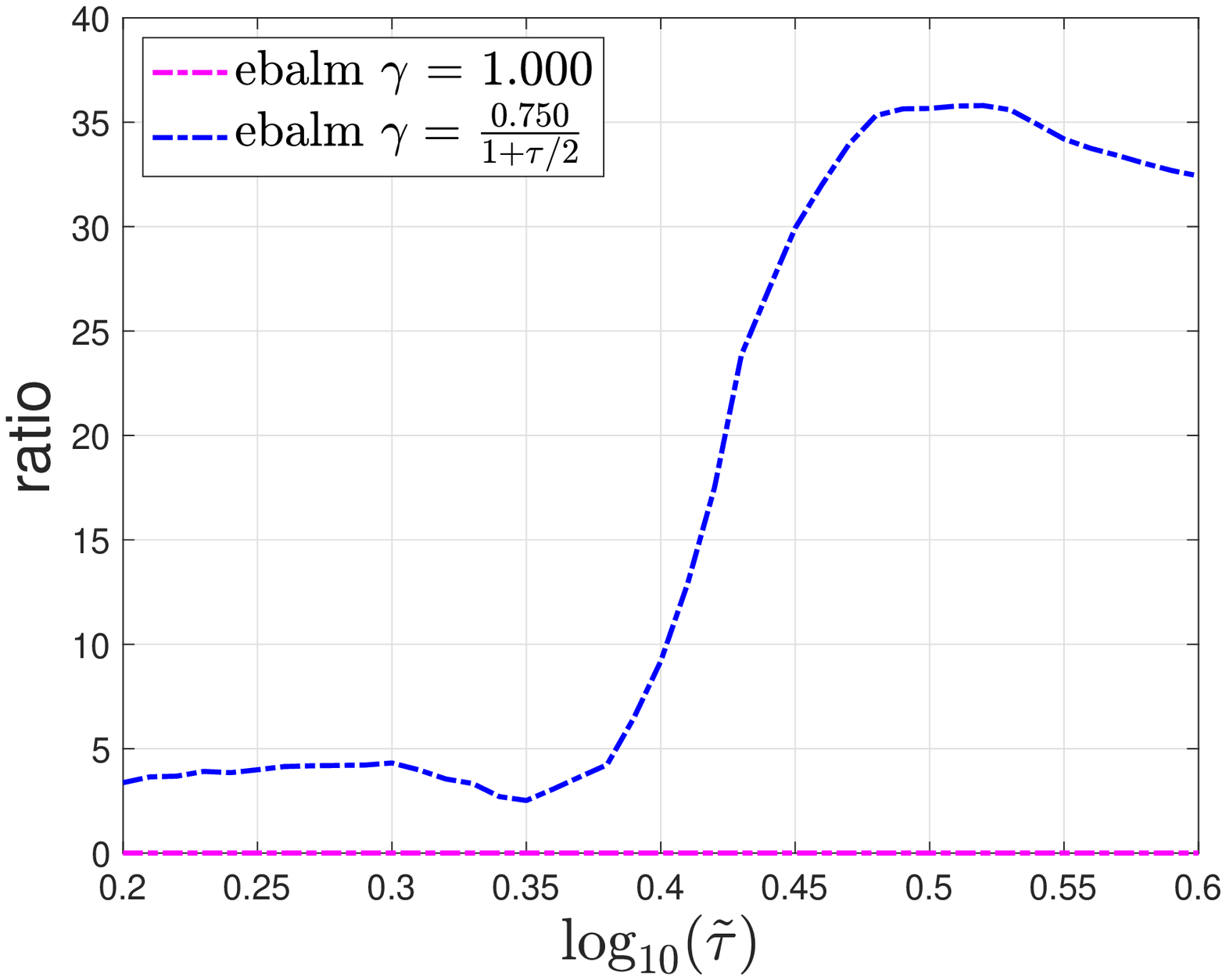}}~
\subfloat[$n = 800$]{\includegraphics[width=2.8cm,height=2.33cm]{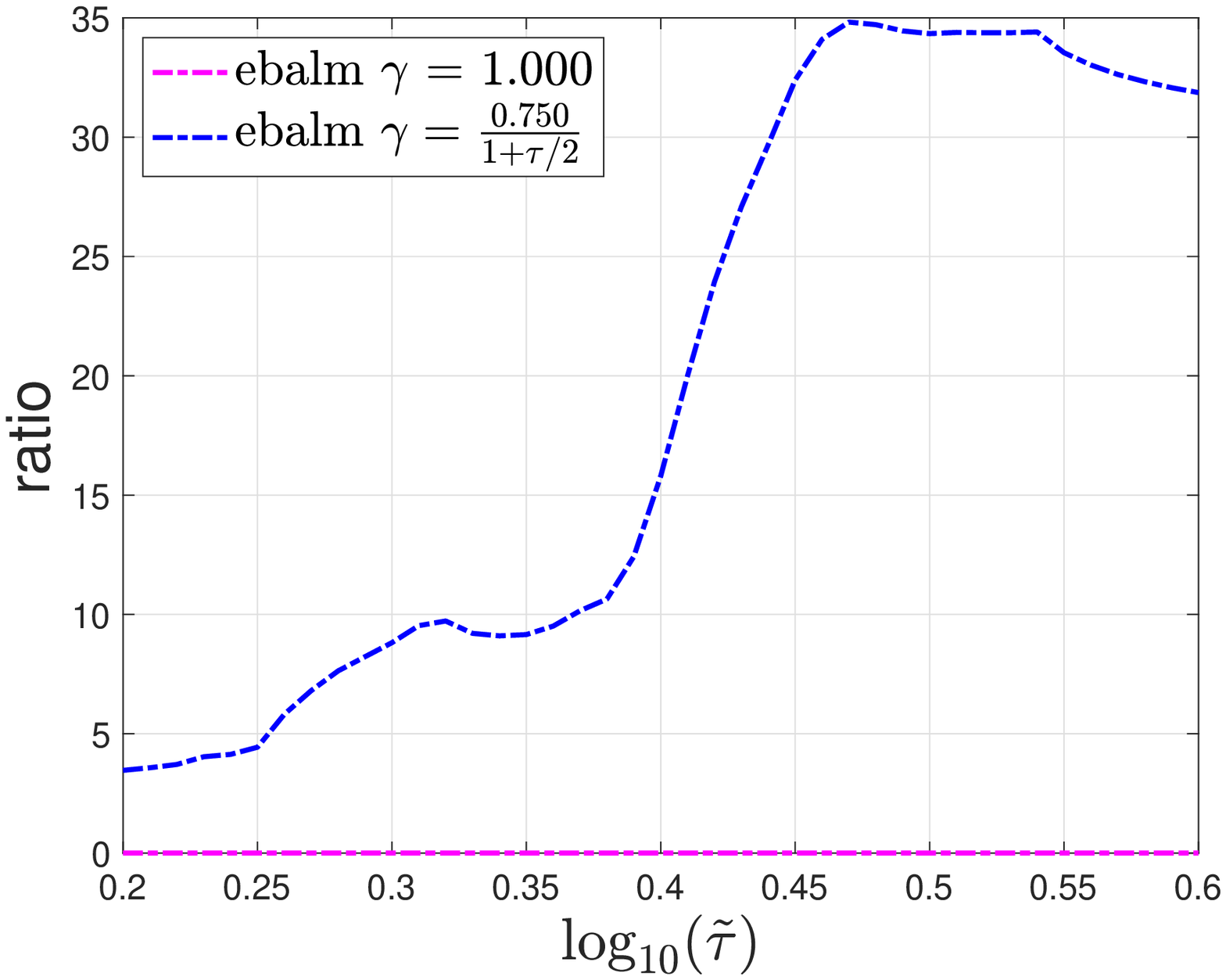}}\\
\caption{Comparison  of  eBALM \eqref{pdhg:Birk:ebalm} with  $\gamma = 1.0$ and $\gamma = \frac{0.75}{1 + \tau/2}$ and PDHG \eqref{pdhg:Birk:proj} with  $\gamma = 1.0$ and $\gamma = \frac{0.751}{1 + \tau/2}$ for problem \eqref{prob:Birk:proj}.}
\label{Birkhoff:figures:ratio:iter}
\end{figure}

To further investigate the effect of $\tilde \tau$ on the performance of different algorithms, as done in Section  \ref{subsection:matrixgame}, we present the performance of each algorithm with  $\tilde \tau_{\mathrm{best}}$ in Table \ref{table:projBirhoff:ebalm}.  From this table, we can see that even with the best possible parameter $\tilde \tau$,  both PDHG and BALM with a smaller $\gamma$ (means the larger stepsize in updating $y$) still have better performance than the corresponding algorithm with larger $\gamma$ for this problem. Besides, eBALM with $\gamma = \frac{0.75}{1 + \tau/2}$ has the best performance,  compared with the classical PDHG with $\gamma = 1$, it saves about $40\%$ of iteration numbers.

\begin{table}[!t]
\small
\setlength{\tabcolsep}{2.5pt}
\centering
\caption{
Performance of eBALM \eqref{pdhg:Birk:ebalm} and PDHG \eqref{pdhg:Birk:proj} with best $\tilde \tau$ for problem \eqref{prob:Birk:proj}.  In the table, ``a'' and ``b'' stands for PDHG \eqref{pdhg:Birk:proj} with $\gamma = 1.0$ and $\gamma = \frac{0.751}{1 + \tau/2}$, respectively; ``c'' and ``d'' stands for eBALM \eqref{pdhg:Birk:ebalm} with $\gamma = 1.0$ and $\gamma = \frac{0.75}{1 + \tau/2}$, respectively.
}~\\
\label{table:projBirhoff:ebalm}
\begin{tabular}{@{}crrrrrrrrrrrrrrr@{}}
\toprule
&  \multicolumn{4}{c}{$\log_{10}(\tilde \tau_{\mathrm{best}})$}  & \multicolumn{4}{c}{time}  & \multicolumn{4}{c}{iter} & \multicolumn{3}{c}{ratio \%}  \\
\cmidrule(lr){2-5} \cmidrule(lr){6-9} \cmidrule(lr){10-13}  \cmidrule(l){14-16}
$n$  & \multicolumn{1}{c}{a}  & \multicolumn{1}{c}{b}  & \multicolumn{1}{c}{c}   & \multicolumn{1}{c}{d}   &\multicolumn{1}{c}{a}   & \multicolumn{1}{c}{b}  & \multicolumn{1}{c}{c}    & \multicolumn{1}{c}{d} & \multicolumn{1}{c}{a}   & \multicolumn{1}{c}{b}   & \multicolumn{1}{c}{c}    & \multicolumn{1}{c}{d}   & \multicolumn{1}{c}{b}  & \multicolumn{1}{c}{c}   & \multicolumn{1}{c}{d} \\
\midrule
200& 0.22& 0.29& 0.37& 0.44& 1.1e-1& 8.6e-2& 7.5e-2& 6.2e-2&  471&  398&  336&  280& 27.0& 28.6& 40.5\\
400& 0.24& 0.31& 0.39& 0.46& 2.6e-1& 2.2e-1& 1.8e-1& 1.9e-1&  676&  574&  485&  408& 29.0& 28.3& 39.6\\
600& 0.23& 0.30& 0.38& 0.45& 5.5e-1& 4.4e-1& 3.7e-1& 3.1e-1&  835&  714&  598&  506& 28.7& 28.4& 39.4\\
800& 0.23& 0.30& 0.38& 0.45& 9.6e-1& 8.2e-1& 7.0e-1& 5.9e-1& 1068&  913&  769&  652& 32.0& 28.0& 39.0\\
\bottomrule
\end{tabular}
\end{table}

\subsection{Earth Mover's Distance}\label{section:EMD}
Given two discrete mass distributions $\rho^{0}$ and $\rho^{1}$ over the $M \times N$ grid,  computing the earth mover's distance between them can be formulated as
 the following optimization problem (see \cite{li2018parallel} for instance):
\be\label{prob:EMD}
\min_{\mbf \in  \Rbb^{2M \times N}}\,\|\mathbf{m}\|_{1,2} \quad \mst \quad \mdiv(\mathbf{m}) + \rho^{1} - \rho^{0} = 0,
\ee
where $\mathbf{m} = \begin{pmatrix} \mathbf{m}^{1} \\ \mathbf{m}^{2}\end{pmatrix}$ is the sought flux vector on the $M \times N$ grid with $\mbf^{1}, \mbf^{2} \in \Rbb^{M \times N}$ and $\mbf_{M,j}^{1}=0$ for $j = 1, \ldots, N$ and $\mbf_{i,N}^{2}=0$ for $i = 1, \ldots, M$.
Here,  $\|\mathbf{m}\|_{1,2}:= \sum_{i=1}^{M}\sum_{j=1}^{N} \sqrt{(\mathbf{m}^{1})_{i,j}^{2} + (\mathbf{m}^{2})_{i,j}^{2}}$. The 2D discrete divergence operator $\mdiv(\mathbf{m}): \Rbb^{2M \times N} \rightarrow \Rbb^{M \times N}$ is defined as
$$
\mdiv(\mbf)_{i,j} = h \left(\mbf_{i,j}^{1}-\mbf_{i-1,j}^{1} + \mbf_{i,j}^{2}-\mbf_{i,j-1}^{2}\right),
$$
where $h$ is the grid stepsize, $\mbf_{0,j}^{1} = 0$ for $j = 1, \ldots, N$ and $\mbf_{i,0}^{2} =0$ for $i = 1, \ldots, M$.
Let $x = \begin{pmatrix} \mvec(\mbf^{1}) \\ \mvec(\mbf^{2}) \end{pmatrix} \in \Rbb^{2MN}$, then problem \eqref{prob:EMD} is a form of \eqref{P1} with  $b = \mvec(\rho^{0} - \rho^{1})$  and the matrix  $\K \in \Rbb^{MN \times 2MN}$ satisfies
$\K x =  \mvec(\mdiv(\mbf))$.

We consider two versions of PrePDHG, namely,  eBALM  \eqref{ebalm} and eBALM-sGS  \eqref{alg:ebalm:sGS} to solve problem \eqref{prob:EMD}.   For eBALM  \eqref{ebalm},  due to the particular structure of $\K\K^{\tran}$ explored in   \cite[Section 4]{liu2021acceleration}, we only performed two epochs of block coordinate descent method as done in \cite{liu2021acceleration}. We name this implementation i-eBALM.  Moreover, we take $\theta = 0$ in \eqref{ebalm} since its performance is very similar to that of very small $\theta$. It should be mentioned that when $\gamma = 1$ in eBALM  \eqref{ebalm}, it becomes the iPrePDHG proposed in \cite{liu2021acceleration}.  Note that \cite{liu2021acceleration} proved the convergence of iPrePDHG under the strong convexity of the objective, which does not hold for problem \eqref{prob:EMD}. Besides, the convergence of i-eBALM \eqref{ebalm} remains unknown, although it performs well.
We consider four choices of $\gamma$. For eBALM  \eqref{ebalm}, we take $\gamma \in \{1.00, 0.90, 0.85, 0.77\}$ and for eBALM-sGS  \eqref{alg:ebalm:sGS}, we take $\gamma \in \{1.00, 0.90, 0.85, 0.75\}$.  Note that the lower bound of $\gamma$ to guarantee the convergence of eBALM-sGS \eqref{alg:ebalm:sGS} is 0.75, see Lemma \ref{lemma:sublinear:LCP}. Actually,  in our numerical tests, eBALM-sGS \eqref{alg:ebalm:sGS} with $\gamma = 0.749$ always diverges.

For this problem, we have $\|b\| \approx 0.009$. Therefore, we replace the term $\|Kx^{k+1} - b\|$    in \eqref{equ:Rxkyk-1:Kxb}  by $\|Kx^{k+1} - b\|/\|b\|$ and stop  each algorithm when the iteration number exceeds 200,000 or the relative KKT residual
 $$\widetilde \Rcal^{k} := \max\{d^{k},  p^{k}\}  \leq \mathrm{tol}:= 5 \times 10^{-5},$$ where
$p^{k}= \tau^{-1}\|x^{k} - x^{k-1}\|$ and $d^{k}= \|\K x^{k} - b\|/\|b\|$.
 The initial $x^{0}$ and $y^{0}$ are both taken as all-zero vectors.  Besides, we adopt the same problem setting  in \cite{li2018parallel,liu2021acceleration}, namely, $M = N = 256$, $h = (N-1)/4$.

The comparison results among different $\gamma$ are reported in Figure  \ref{eBALM-sGS:gamma}. In this figure, for each fixed $\tau\in \{1, 1.1, \ldots, 6.9,7\}\times 10^{-6}$, the saved ratio in terms of iteration number is defined as \eqref{equ:ratio}, where $\underline{\mathrm{iter}}$ is taken as the corresponding method with $\gamma = 1$.
From the figures, we can see that both eBALM and eBALM-sGS benefit from choosing small $\gamma$, which enlarges the stepsize in updating $y^{k+1}$ in some sense. In particular, for eBALM, when $\tau \geq 4 \times 10^{-6}$, the saved ratios of taking $\gamma = 0.77, 0.85, 0.90$ are  about  20\%, 15\% and 10\%, respectively. For eBALM-sGS, when $\tau \geq 3 \times 10^{-6}$, the saved ratios of taking $\gamma = 0.77, 0.85, 0.90$ are about  25\%, 15\% and 10\%, respectively. Besides, we also know that eBALM-sGS always perform worse than i-eBALM, although the former has a convergence guarantee while the latter does not.

\begin{figure}[!t]
\centering
\subfloat[i-eBALM: iteration versus $\tau$\label{fig:y equals x} ]{\includegraphics[width=2.8cm,height=2.33cm]{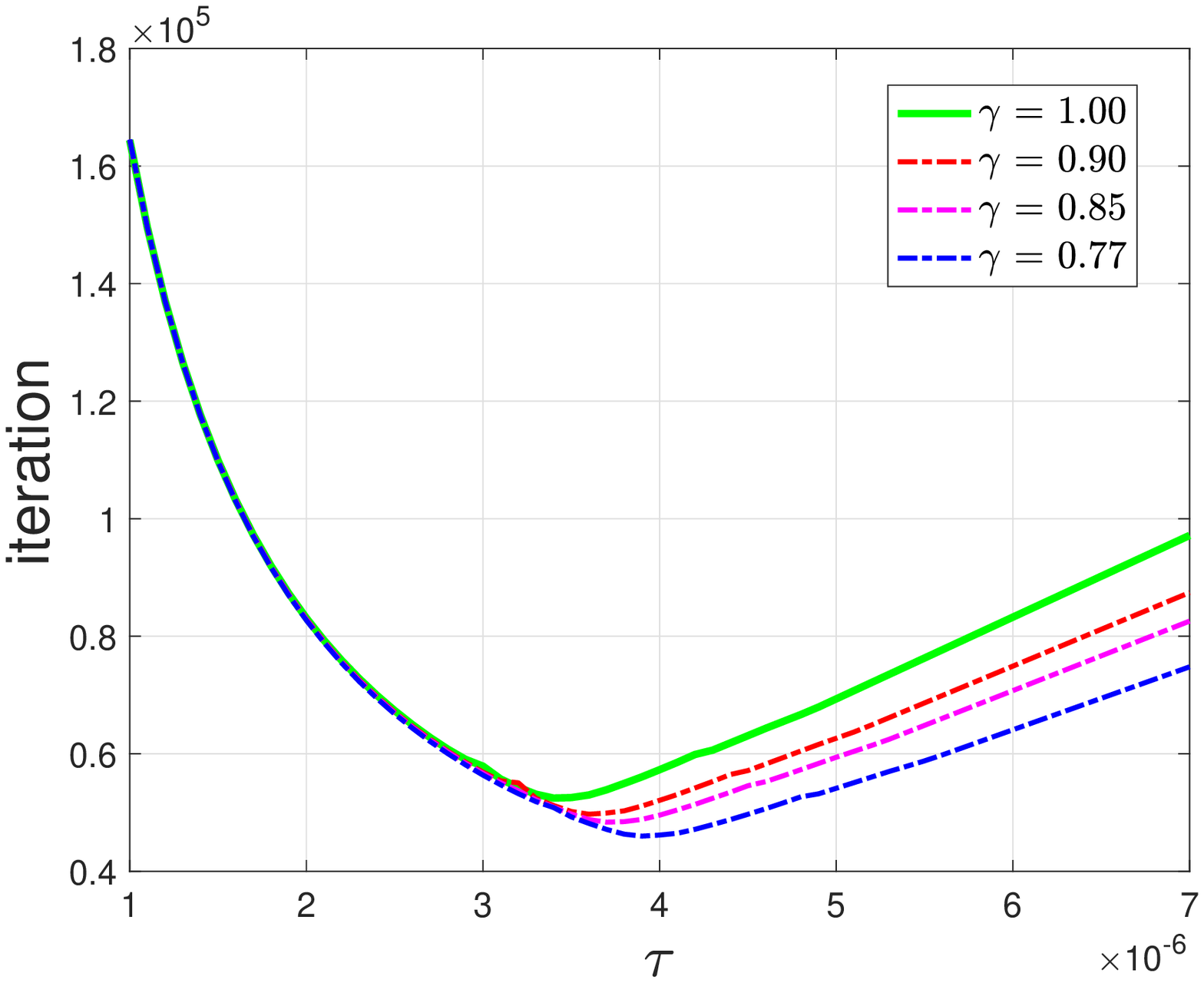}}~
\subfloat[i-eBALM: ratio versus $\tau$]{\includegraphics[width=2.8cm,height=2.33cm]{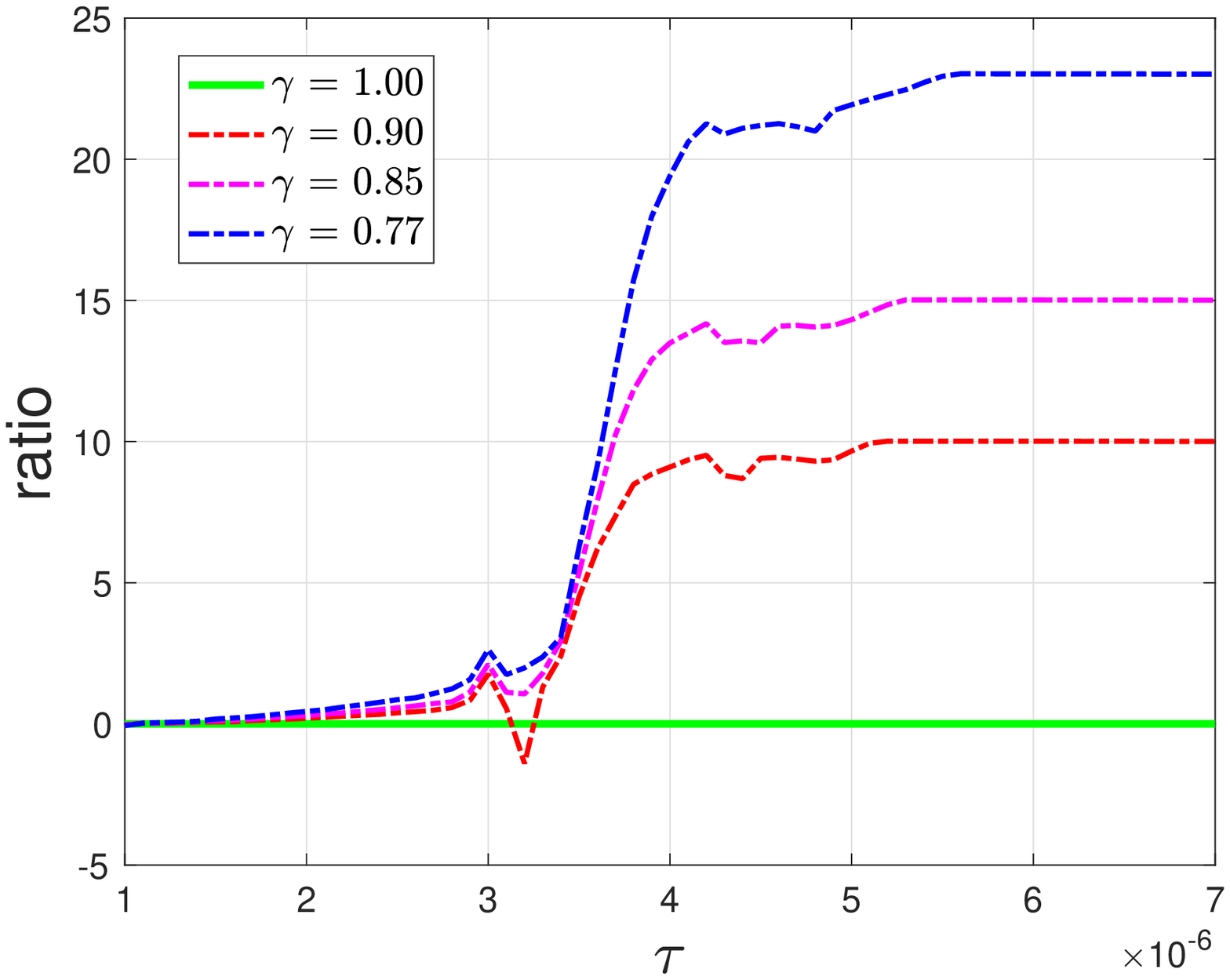}}~
\subfloat[eBALM-sGS: iteration versus $\tau$]{\includegraphics[width=2.8cm,height=2.33cm]{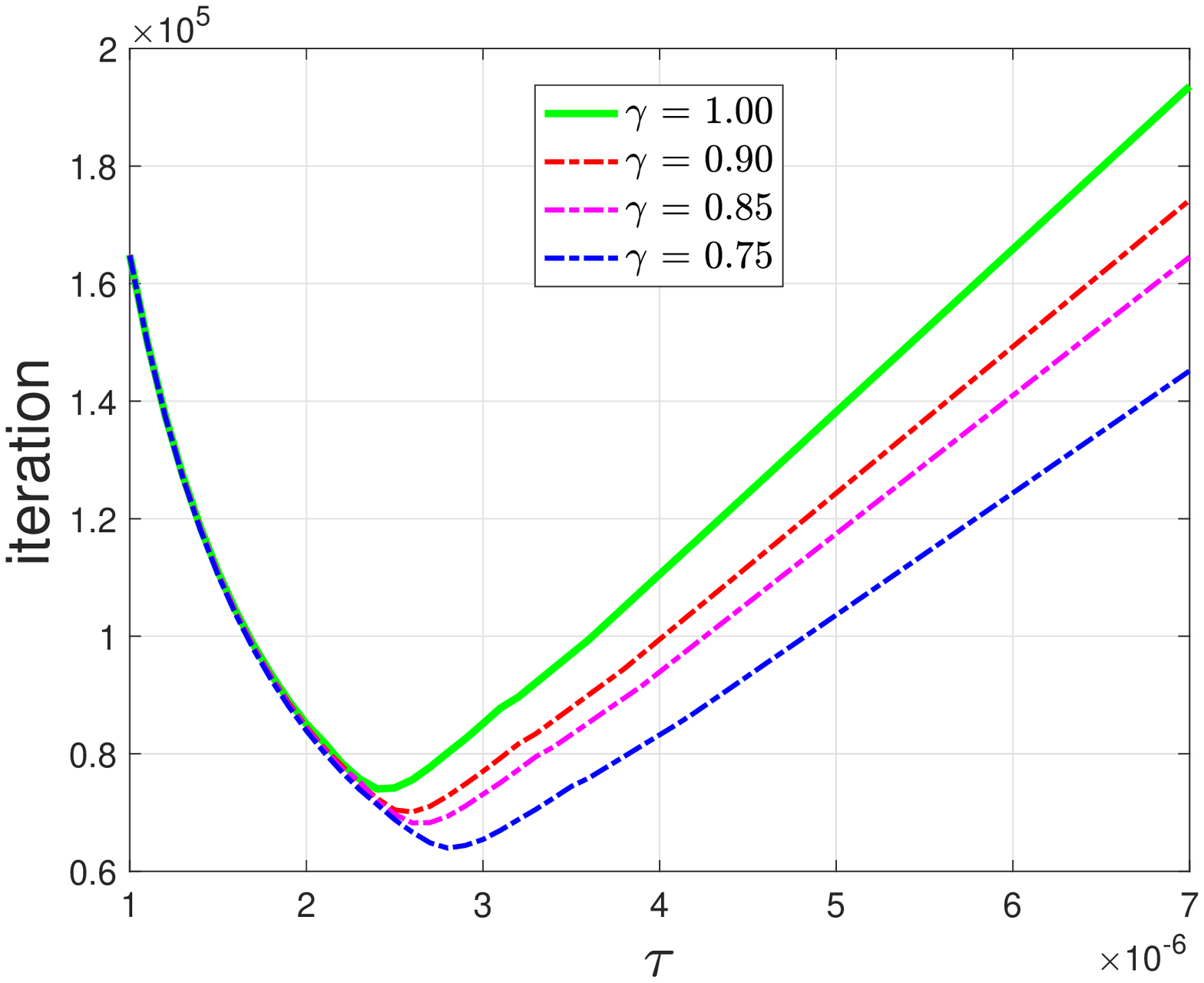}}~
\subfloat[eBALM-sGS: ratio versus $\tau$]{\includegraphics[width=2.8cm,height=2.33cm]{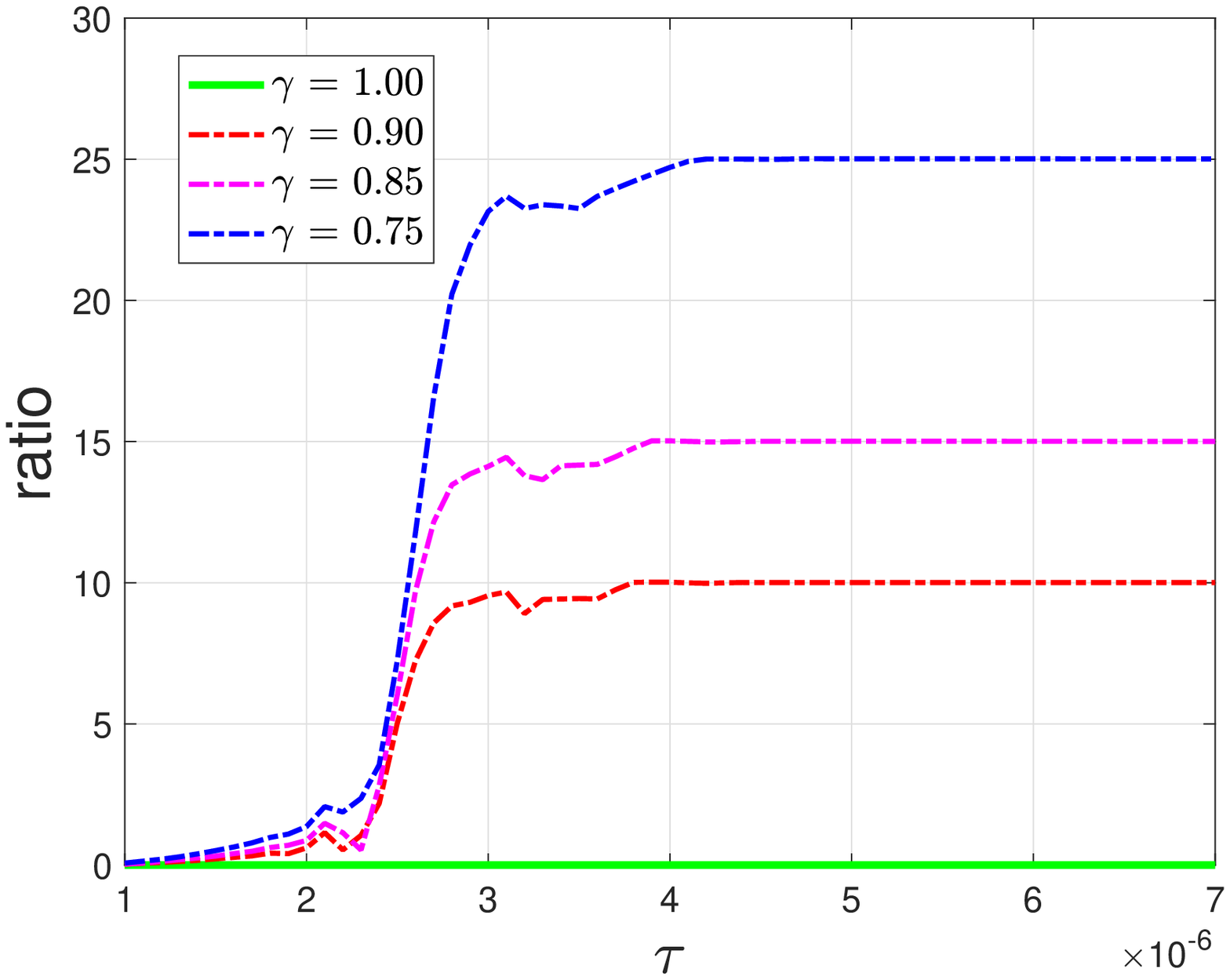}}
\caption{Comparison  on iteration and ratio of i-eBALM with $\gamma = \{1.00, 0.90, 0.85, 0.77\}$  and  eBALM-sGS with $\gamma = \{1.00, 0.90, 0.85, 0.75\}$. Note that i-eBALM with $\gamma = 1$ is exactly iPrePDHG in \cite{liu2021acceleration}.}
\label{eBALM-sGS:gamma}
\end{figure}

\begin{figure}[!t]
\centering
\subfloat[i-eBALM: $d^{k}$ versus iterations]{\includegraphics[width=2.8cm,height=2.33cm]{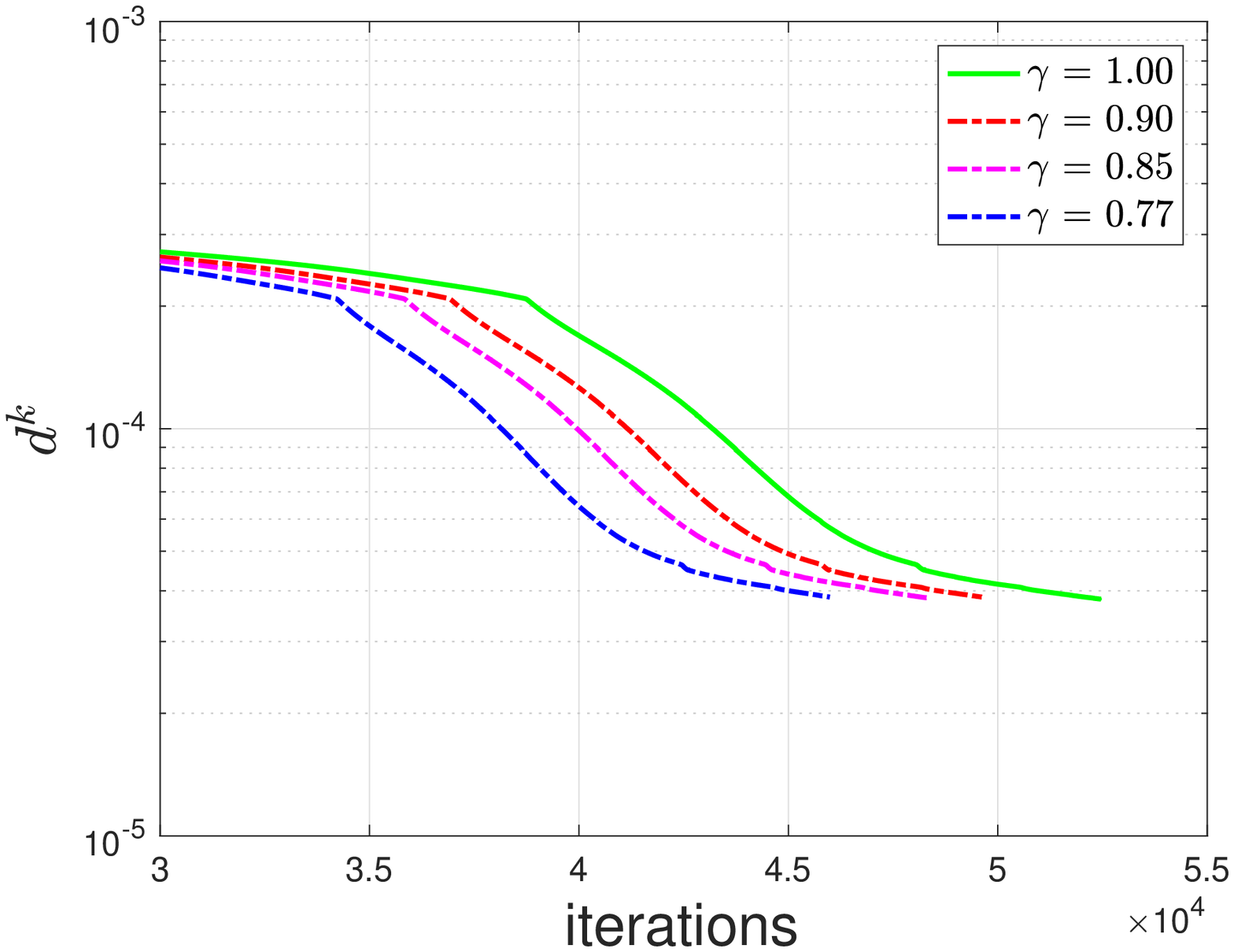}}~
\subfloat[i-eBALM: $p^{k}$ versus iterations]{\includegraphics[width=2.8cm,height=2.33cm]{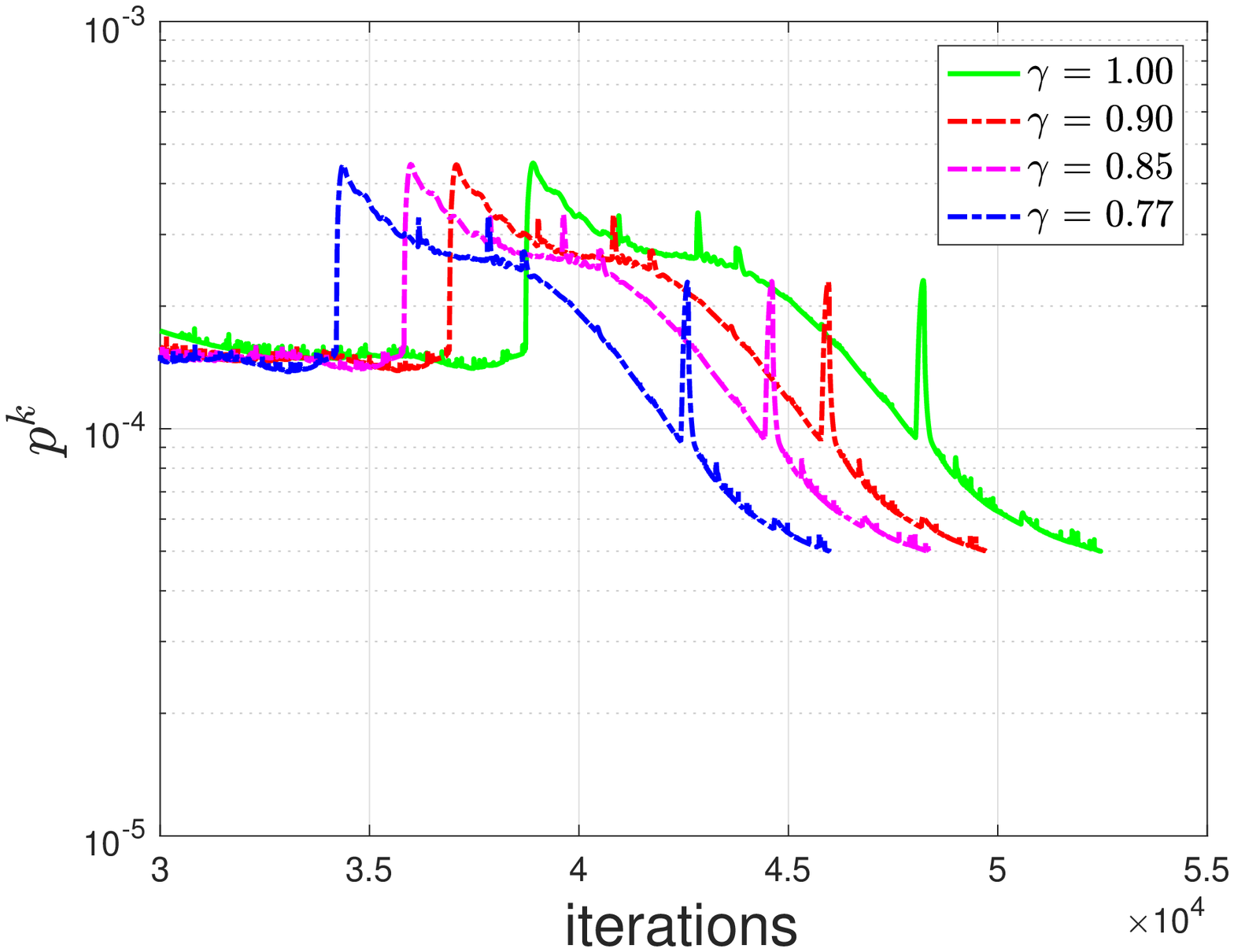}}~
\subfloat[eBALM-sGS: $d^{k}$ versus iterations]{ \includegraphics[width=2.8cm,height=2.33cm]{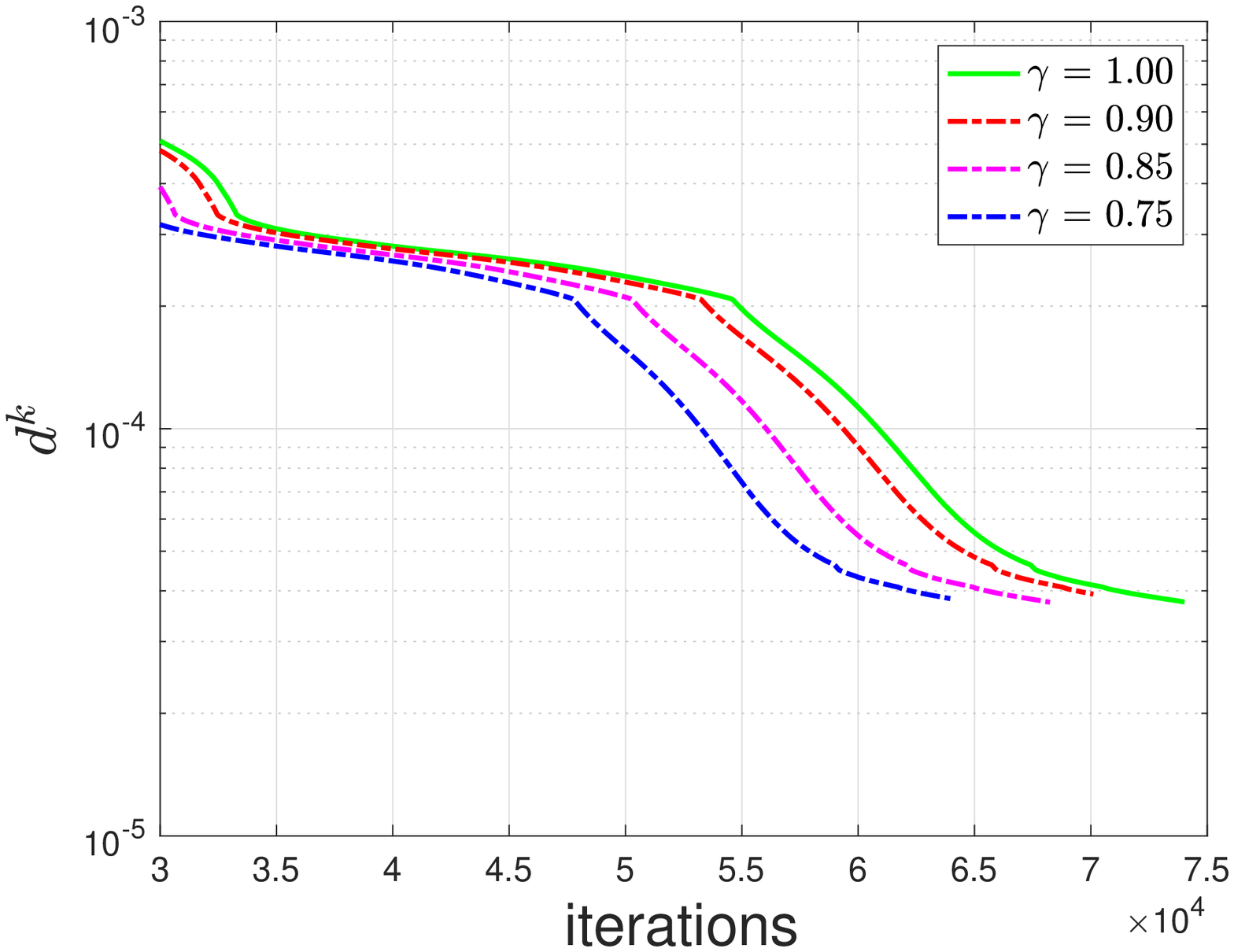}}~
\subfloat[eBALM-sGS: $p^{k}$ versus iterations]{ \includegraphics[width=2.8cm,height=2.33cm]{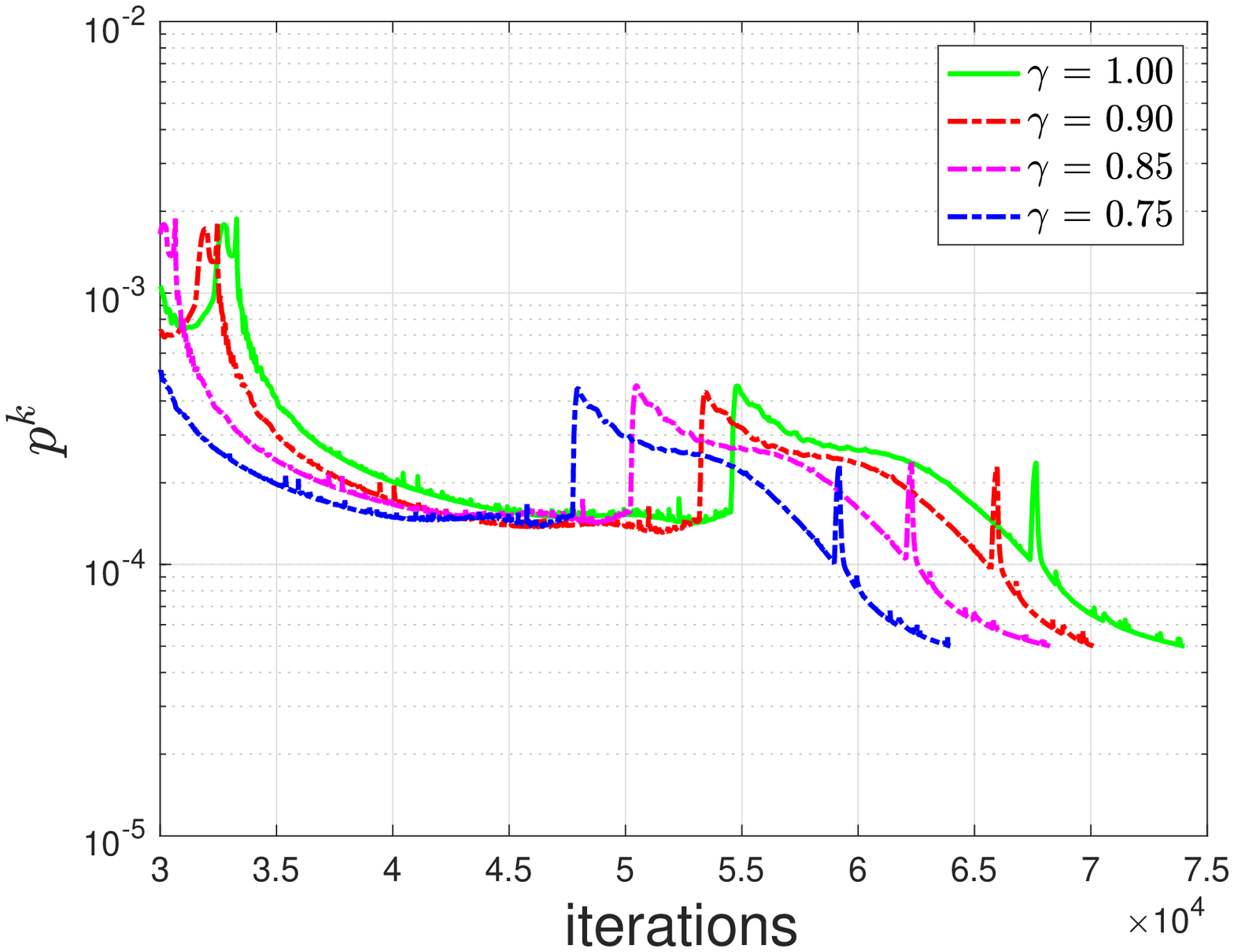}}
\caption{Comparison on $p^k$ and $d^k$ of i-eBALM with $\gamma = \{1.00, 0.90, 0.85, 0.77\}$  and  eBALM-sGS with $\gamma = \{1.00, 0.90, 0.85, 0.75\}$.  The parameter of each method is taken as $\tau_{\mathrm{best}}$ in Table \ref{table:EMD}.  Note that i-eBALM with $\gamma = 1$ is exactly iPrePDHG in \cite{liu2021acceleration}.}
\label{figure:EMD:taubest}
\end{figure}

As done in Section  \ref{subsection:matrixgame}, we present  the results corresponding to the best $\tau$ in Figure \ref{figure:EMD:taubest} and  Table \ref{table:EMD}.  From them, we observe that choosing small $\gamma$ can still accelerate the corresponding method with $\gamma = 1$, and the saved ratio is always more than 12\% when we take $\gamma = 0.77$ in eBALM and $\gamma = 0.75$ in eBALM-sGS. Again note that to achieve such improvement, we only need to change a parameter in the original method without increasing any additional cost.   We also know from Table \ref{table:EMD}  that the saved ratios shown in this table are not as large as those in Figure \ref{eBALM-sGS:gamma}.  However, choosing the best $\tau$ from a portion of candidates is time-consuming and impractical for both i-eBALM and eBALM-sGS.

\begin{table}[!t]
\small
\centering
\caption{Results of i-eBALM and eBALM-sGS with best $\tau$.
}~\\
\label{table:EMD}
\begin{tabular}{cccccr}
\toprule
$\gamma$  & $\tau_{\mathrm{best}}$  & time & iter  & $\|\mathbf{m}\|_{1,2}$  & ratio    \\
\midrule
\multicolumn{6}{c}{i-eBALM} \\
1.00 & 3.4e-6 & 139.6  &  52461  & 0.671770 & 0.00\\
0.90 & 3.6e-6 & 131.2  &  49715  & 0.671770 & 5.23\\
0.85 & 3.7e-6 & 125.5  &  48362  & 0.671770 & 7.81\\
0.77 & 3.9e-6 & 119.2  &  45990  & 0.671770 & 12.33\\
\midrule
\multicolumn{6}{c}{eBALM-sGS} \\
1.00 & 2.4e-6 & 166.1  &  74024  & 0.671770 & 0.00\\
0.90 & 2.6e-6 & 155.7  &  70105  & 0.671769 & 5.29\\
0.85 & 2.6e-6 & 153.4  &  68241  & 0.671770 & 7.81\\
0.75 & 2.8e-6 & 142.7  &  63955  & 0.671770 & 13.60\\
\bottomrule
\end{tabular}
\end{table}

\begin{figure}[!t]
\small
\centering
\subfloat[i-eBALM, tol = $5 \times 10^{-1}$, $m^k_r = 3.6\times 10^{-1}$, iter = 842]{\includegraphics[width=2.8cm]{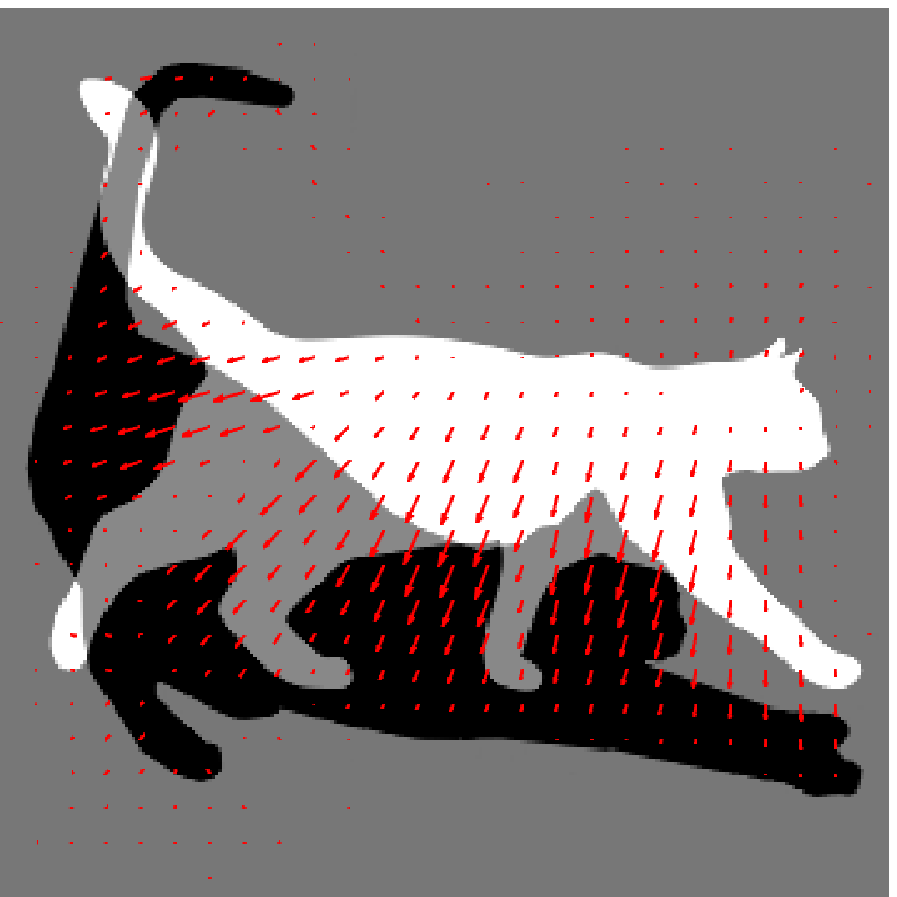}}~
\subfloat[i-eBALM, tol = $5 \times 10^{-3}$, $m_r^k = 1.7\times 10^{-2}$, iter = 8404]{\includegraphics[width=2.8cm]{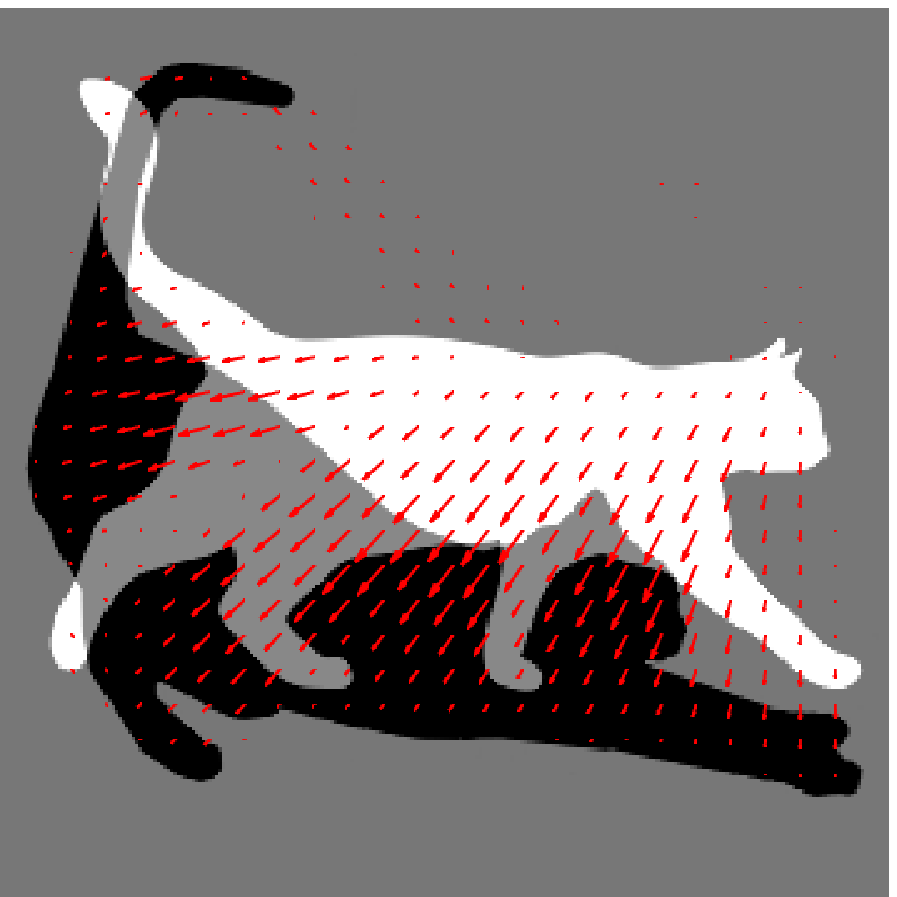}}~
\subfloat[i-eBALM, tol = $5 \times 10^{-5}$, $m_r^k = 8.9\times 10^{-4}$, iter = 45990]{\includegraphics[width=2.8cm]{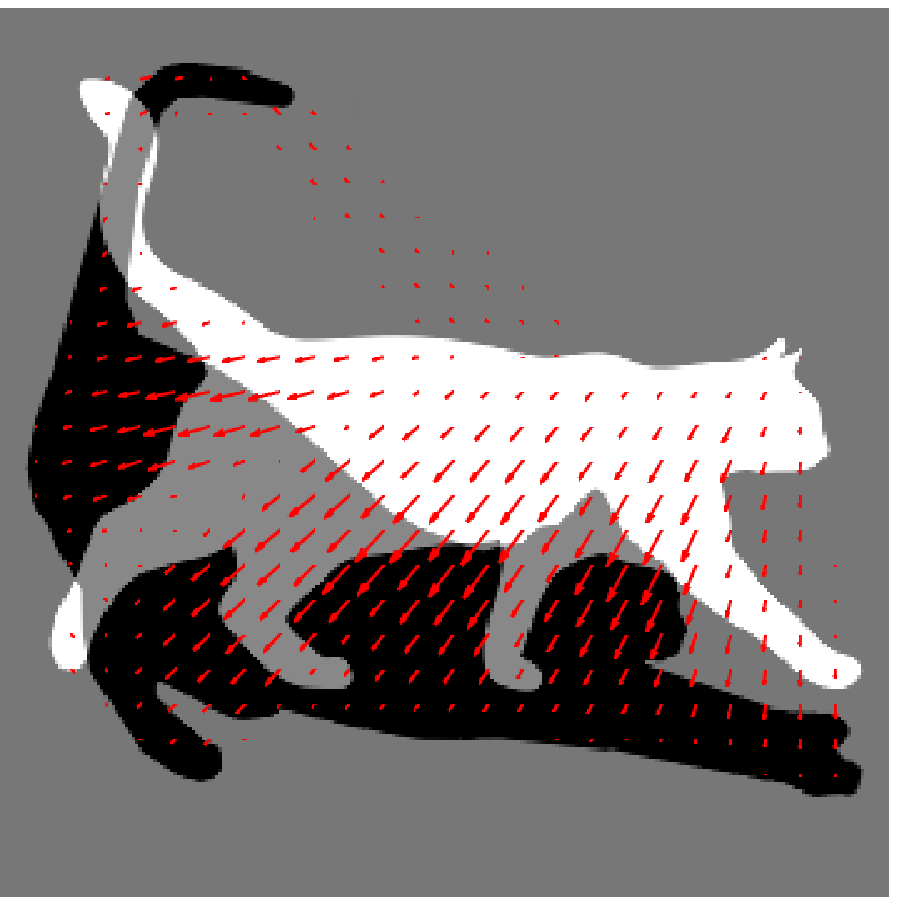}}~
\subfloat[groundtruth from \cite{liu2021acceleration}]{\includegraphics[width=2.8cm]{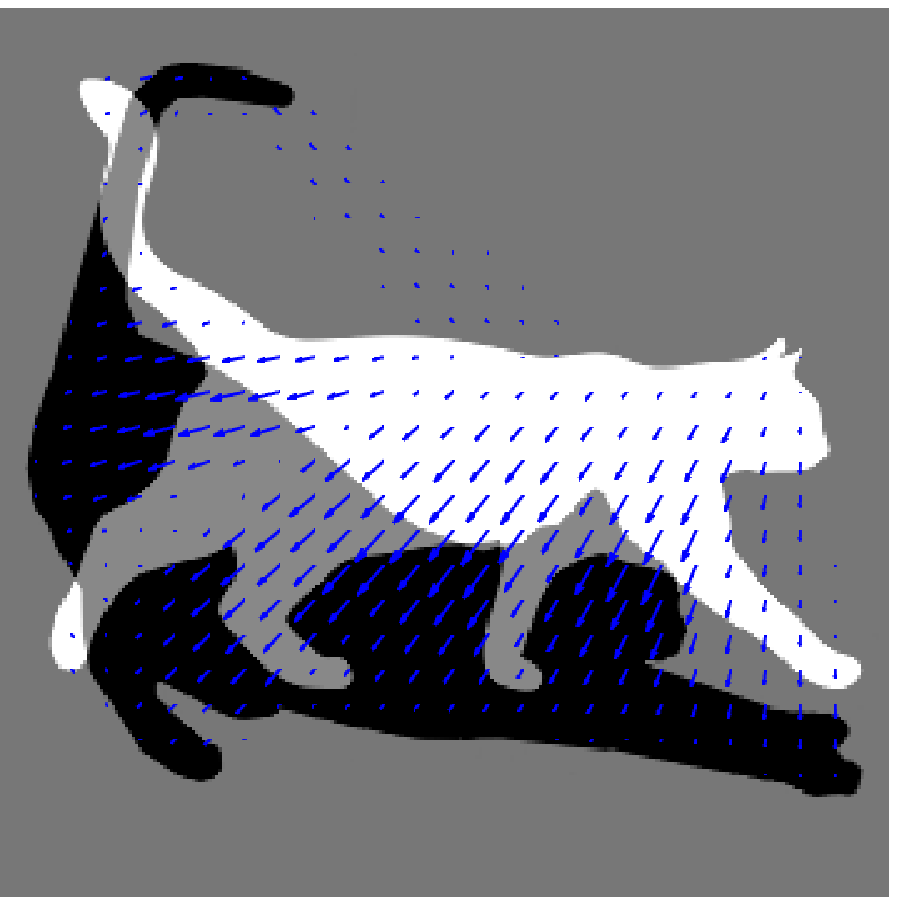}}

\subfloat[eBALM-sGS, tol = $5 \times 10^{-1}$, $m^k_r = 3.6\times 10^{-1}$, iter = 1171]{\includegraphics[width=2.8cm]{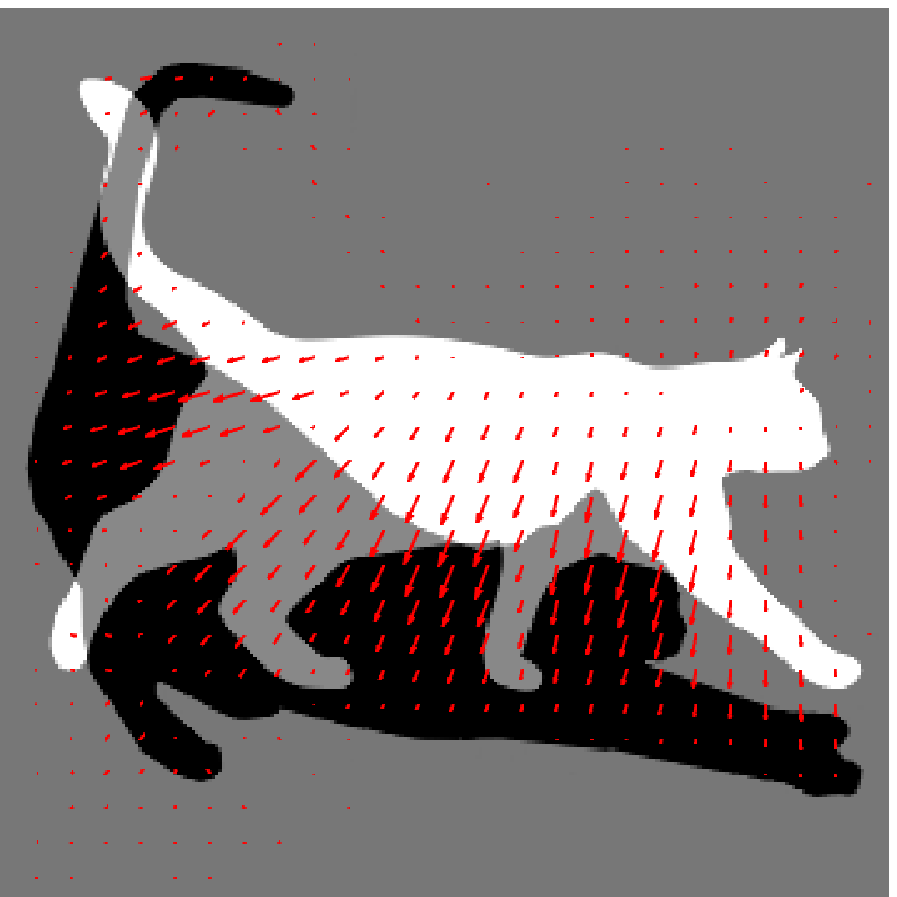}}~
\subfloat[eBALM-sGS, tol = $5 \times 10^{-3}$, $m_r^k = 1.7\times 10^{-2}$, iter = 11729]{\includegraphics[width=2.8cm]{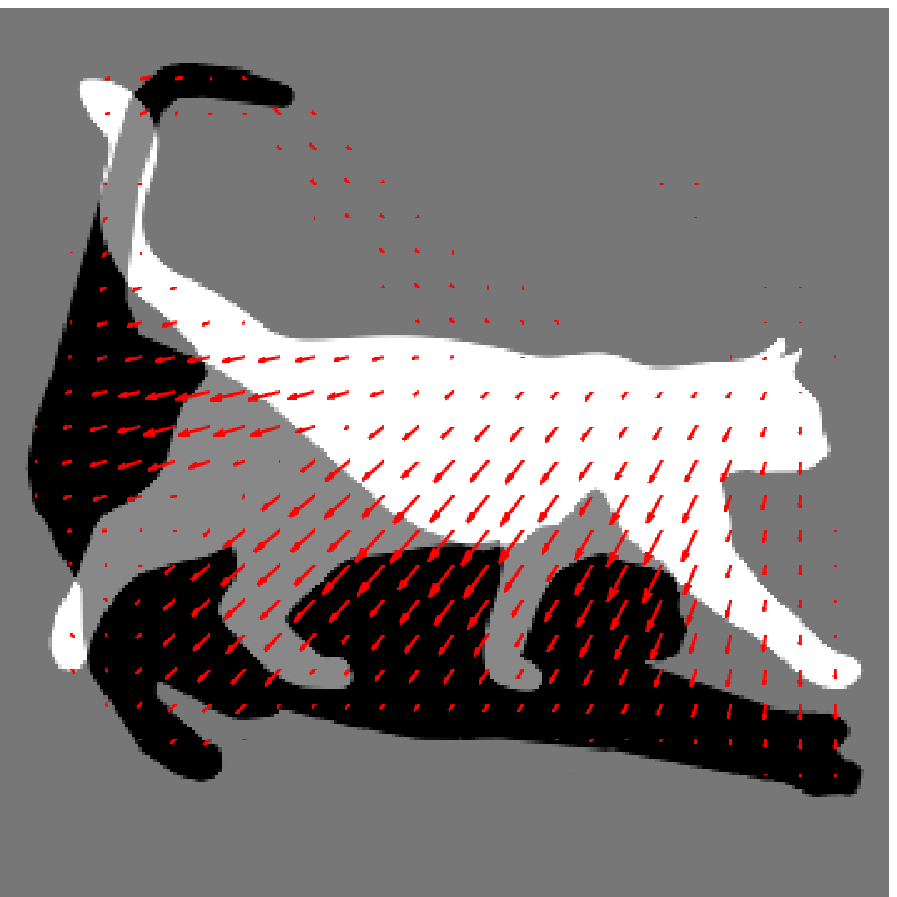}}~
\subfloat[eBALM-sGS, tol = $5 \times 10^{-5}$, $m_r^k = 8.5\times 10^{-4}$, iter = 63955]{\includegraphics[width=2.8cm]{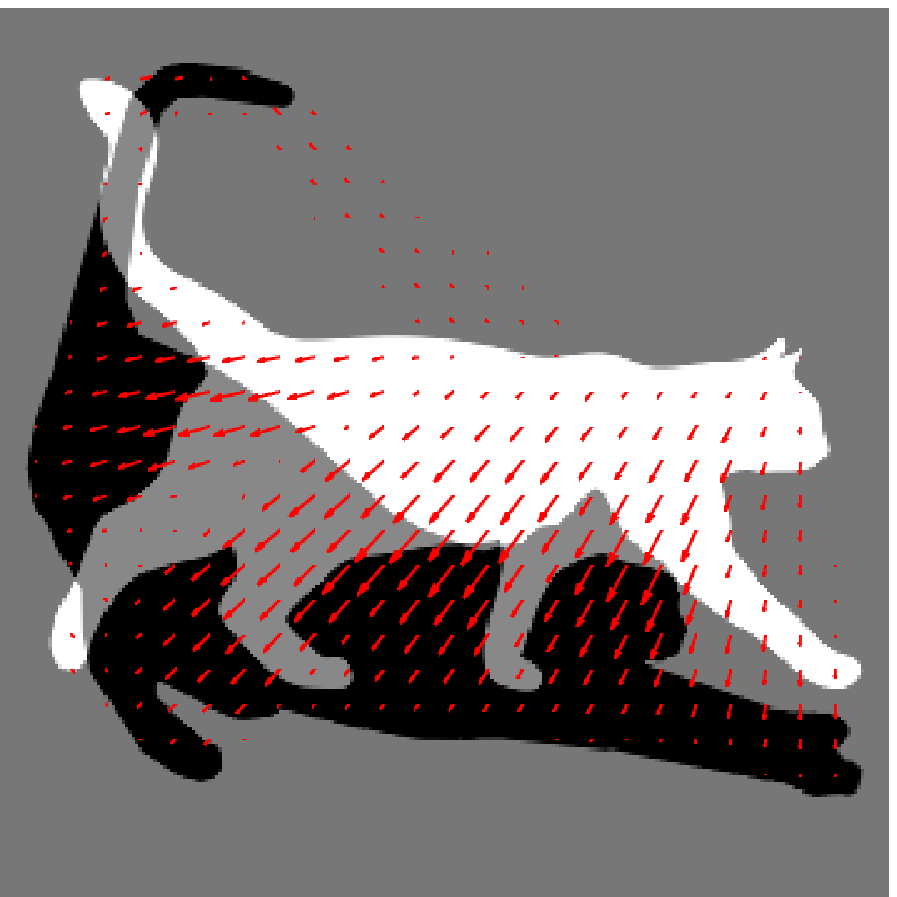}}~
\subfloat[groundtruth from \cite{liu2021acceleration}]{\includegraphics[width=2.8cm]{figs/EMD_cat_opt.eps}}

\caption{Mass distributions $\rho^{0}$ and $\rho^{1}$ are both with size $256 \times 256$. The white standing cat is $\rho^{0}$ and the black crouching cat is $\rho^{1}$. The red or blue curves are the flux that moves the standing cat $\rho^{0}$ into the crouching cat $\rho^{1}$. The ground truth flux, denoted by $m^{\mathrm{cvx}}$, is obtained by CVX after several hours. The earth mover's distance between $\rho^{0}$ and $\rho^{1}$ is 0.671783.
The term $m_r^k= \|m^{k} - m^{\mathrm{cvx}}\|/\|m^{\mathrm{cvx}}\|$, where $m^k$ is the flux obtained by each method.
The data matrices $\rho^0$, $\rho^1$, and  $m^{\mathrm{cvx}}$ are downloaded from   \url{https://github.com/xuyunbei/Inexact-preconditioning}.} \label{figure:emd:rho0rho1}
\end{figure}

Finally,  in Figure \ref{figure:emd:rho0rho1}, we show the solutions obtained by eBALM-sGS with different tolerance and the ground truth obtained by running CVX in several hours, see \cite{liu2021acceleration}. We can see that eBALM-sGS with tolerance $\mathrm{tol} = 5\times 10^{-5}$ can return a solution with satisfactory precision.

\subsection{CT Reconstruction} \label{subsection:CT}
Let $x_{\mathrm{true}} \in \Rbb^{n}$ with $n = M N$  and $M = N = 256$ be a true image. Given a vector of line-integration values  $b = R x_{\mathrm{true}} \in \Rbb^m$, where  $R \in \Rbb^{m \times n}$ is a system matrix for 2D fan-beam CT with a curved detector,  the CT image reconstruction aims to recover $x_{\mathrm{true}}$ via solving  the following optimization problem:
\be\label{prob:CT}
\min_{x \in \Rbb^{n}}\ \Phi(x):= \frac12 \|Rx - b\|^2 + \lambda \|D x \|_1,
\ee
where  $D \in \Rbb^{2n \times n}$ is the 2D discrete gradient operator with $h = 1$ (see \cite[Section 4]{liu2021acceleration} for instance) and $\lambda = 1$ is a regularization parameter.

To avoid solving the linear system related  to the matrices  $R$ and $D$, as done in \cite{liu2021acceleration} and \cite{sidky2012convex},  we understand problem \eqref{prob:CT} as a form of \eqref{P} with
$$
f(x) = 0, \quad g(z) = \frac12 \|p - b\|^2 + \lambda \|q\|_1, \quad z = \begin{pmatrix} p \\q \end{pmatrix} \in \Rbb^{m + 2n}, \quad \K =\begin{pmatrix} R \\ D \end{pmatrix}.
$$
 We  choose the variable metric matrices $M_1$ and $M_2$ via \eqref{M1M2:general:2block}, wherein $\K_{1}$ and $\K_{2}$ are $R$ and $D$, respectively and $\sigma = (\tau \|R\|^2)^{-1}$, $P_{2} = \theta I_{n}$.
The constant $\theta$ is taken as $10^{-3}$ in our experiments.  According to Remark \ref{remark:M1M2:general:2block}, we have  the parameter  $\gamma  \geq 3/4$.   Note that the dual variable $y$ can be decomposed as $y = \begin{pmatrix} y_{1} \\ y_{2} \end{pmatrix}$ with $y_{1} \in \Rbb^{m}$ and $y_{2} \in \Rbb^{2n}$. The main  iteration scheme of PrePDHG \eqref{vmpdhg} for solving problem \eqref{prob:CT}  is given as follows:
\begin{subnumcases}{\label{vm-pdhg4ct}}
x^{k+1}= x^k - \frac{\tau}{2\gamma}\left(R^{\tran} y_1^k + D^{\tran} y_2^k\right)\label{vmpdhg:ct:x},\\
y_1^{k+1} = \frac{\tau \|R\|^2 y_1^k + R(2x^{k+1} - x^k) - b}{1 + \tau \|R\|^2} \label{vmpdhg:ct:y1},\\
y_2^{k+1} = \argmin_{\|y_2\|_{\infty} \leq \lambda}\, \frac12\left\|y_2 - y_2^k\right\|_{\tau D D^{\tran} + \theta I_{2n}}^2 - \iprod{y_2}{D(2x^{k+1} - x^k)}.\label{vmpdhg:ct:y2}
\end{subnumcases}
The $y_{2}$-subproblem in \eqref{vmpdhg:ct:y2} does not take a closed-form solution. However, thanks to the special structure of $D$, \cite{liu2021acceleration} developed an efficient block coordinate descent (BCD) method to solve \eqref{vmpdhg:ct:y2}.
To guarantee the convergence of  PrePDHG \eqref{vm-pdhg4ct}, theoretically, we need to run many BCD epochs to solve \eqref{vmpdhg:ct:y2} almost exactly.  However, this may be time-consuming as observed in \cite{liu2021acceleration}.  As suggested by \cite{liu2021acceleration}, we find that running two BCD steps is enough to make the PrePDHG \eqref{vm-pdhg4ct} perform well. Hence, in our numerical experiments, we only apply two BCD steps in solving the $y_{2}$-subproblem.  Considering that  \cite{liu2021acceleration} has shown the superiority of their proposed inexact preconditioned PDHG (iPrePDHG)  over other variants of  PDHG, here we mainly compare our PrePDHG \eqref{vm-pdhg4ct}  with iPrePDHG therein. It should be mentioned that iPrePDHG corresponds to our PrePDHG  \eqref{vm-pdhg4ct} with $\gamma = 1$.  For PrePDHG \eqref{vm-pdhg4ct}, we consider three versions with $\gamma = 5/6$, $\gamma = 3/4$ and $\gamma = 1/2$, respectively. Although there is no convergence guarantee for the last one, it performs very well in our numerical experiments.

Given a vector $z = [0, \nu, 2\nu, \ldots, 360 - \nu]^{\tran}$ containing the projection angles in degrees,  we generate a test problem by using the \textsf{fancurvedtomo} function from the AIR Tools II package \cite{hansen2018air} with input $N$ and $z$.  In our numerical results, we consider $\nu \in \{6, 9, 12, 15, 18,24, 30,36\}$. The starting points of PrePDHG and iPrePDHG are both taken as $x^0 = 0$ and $y^0 = 0$.  We stop each algorithm at $(x^k, y^k)$  when the KKT residual $\Rcal(x^k,y^k) \leq 5 \times 10^{-6}$, where $\Rcal(x^k,y^k)$ is computed according to
\[
\begin{aligned}
{}&\Rcal(x^k,y^k)   \\
={}& \max\big\{\|R^\tran y_1^k + D^\tran y_2^k\|,  \|Rx^k - y_1^k -b\|, \dist\left(Dx^k, \partial \mathbb{I}_{\|y_2\|_{\infty} \leq \lambda}(y_2^k)\right)\!\big\}.
\end{aligned}
\]

\begin{figure}[!t]
\centering
\subfloat[$\nu = 6$]{\includegraphics[width=2.8cm,height=2.33cm]{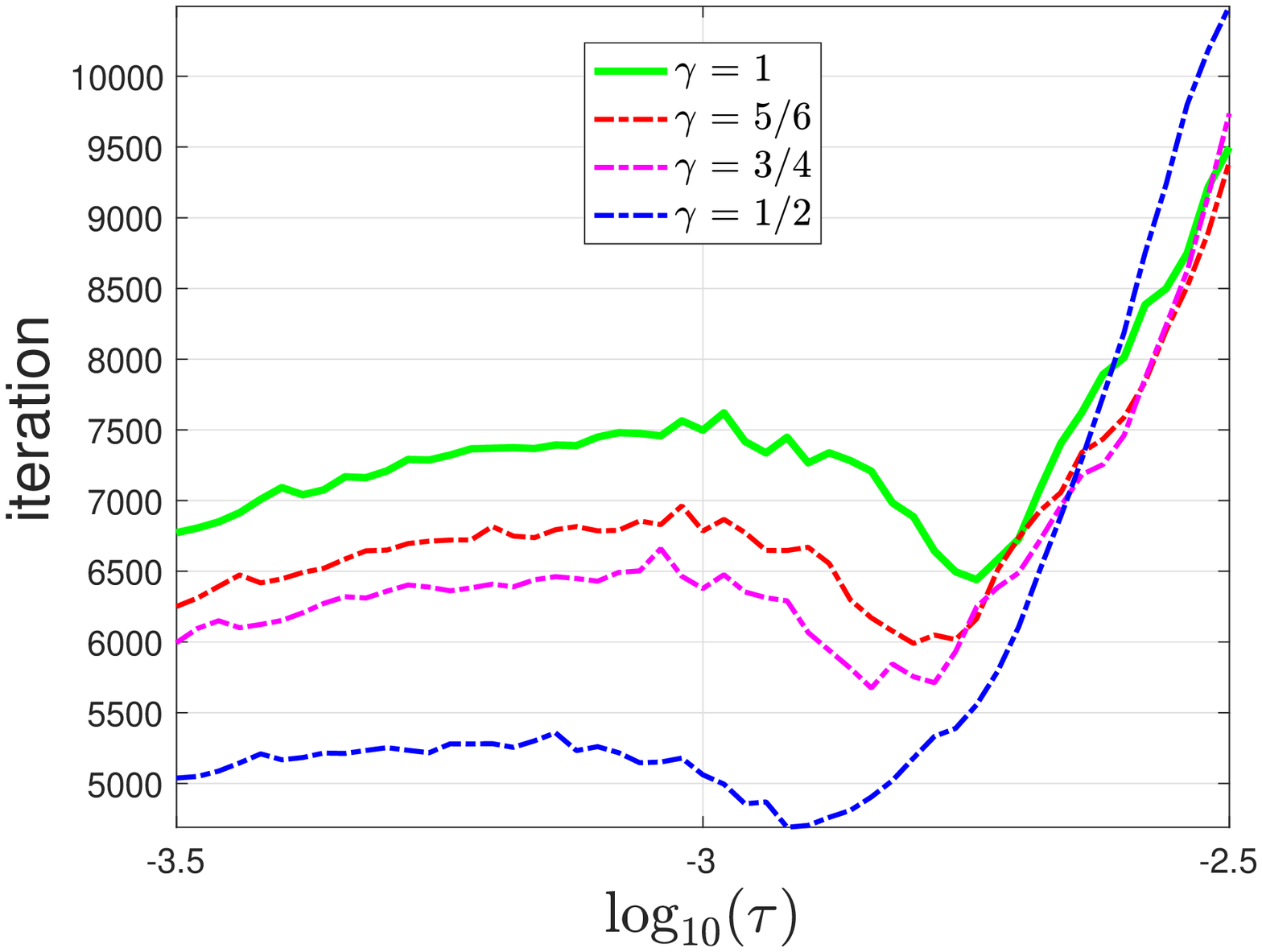}}~
\subfloat[$\nu = 9$]{\includegraphics[width=2.8cm,height=2.33cm]{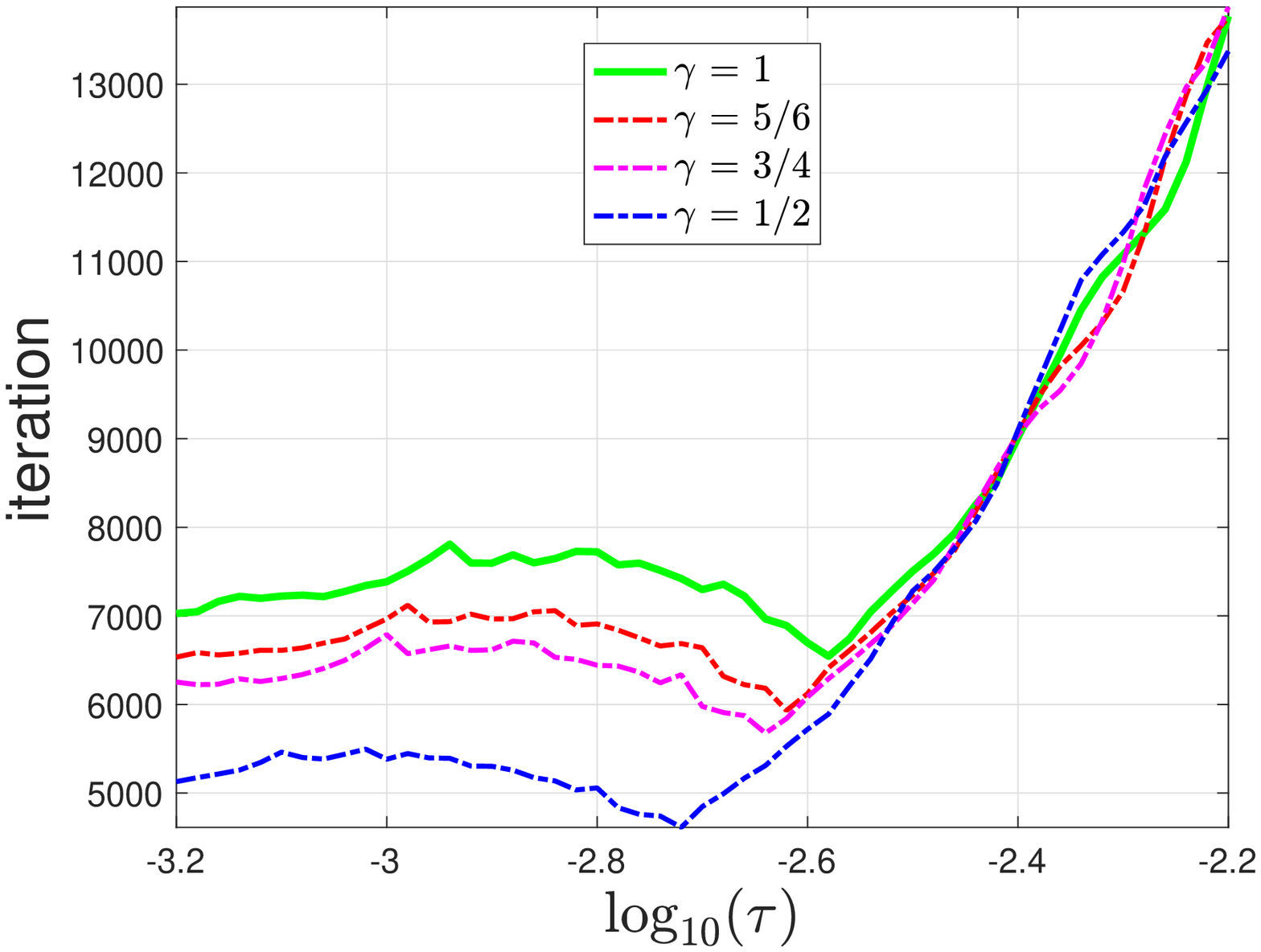}}~
\subfloat[$\nu = 12$]{\includegraphics[width=2.8cm,height=2.33cm]{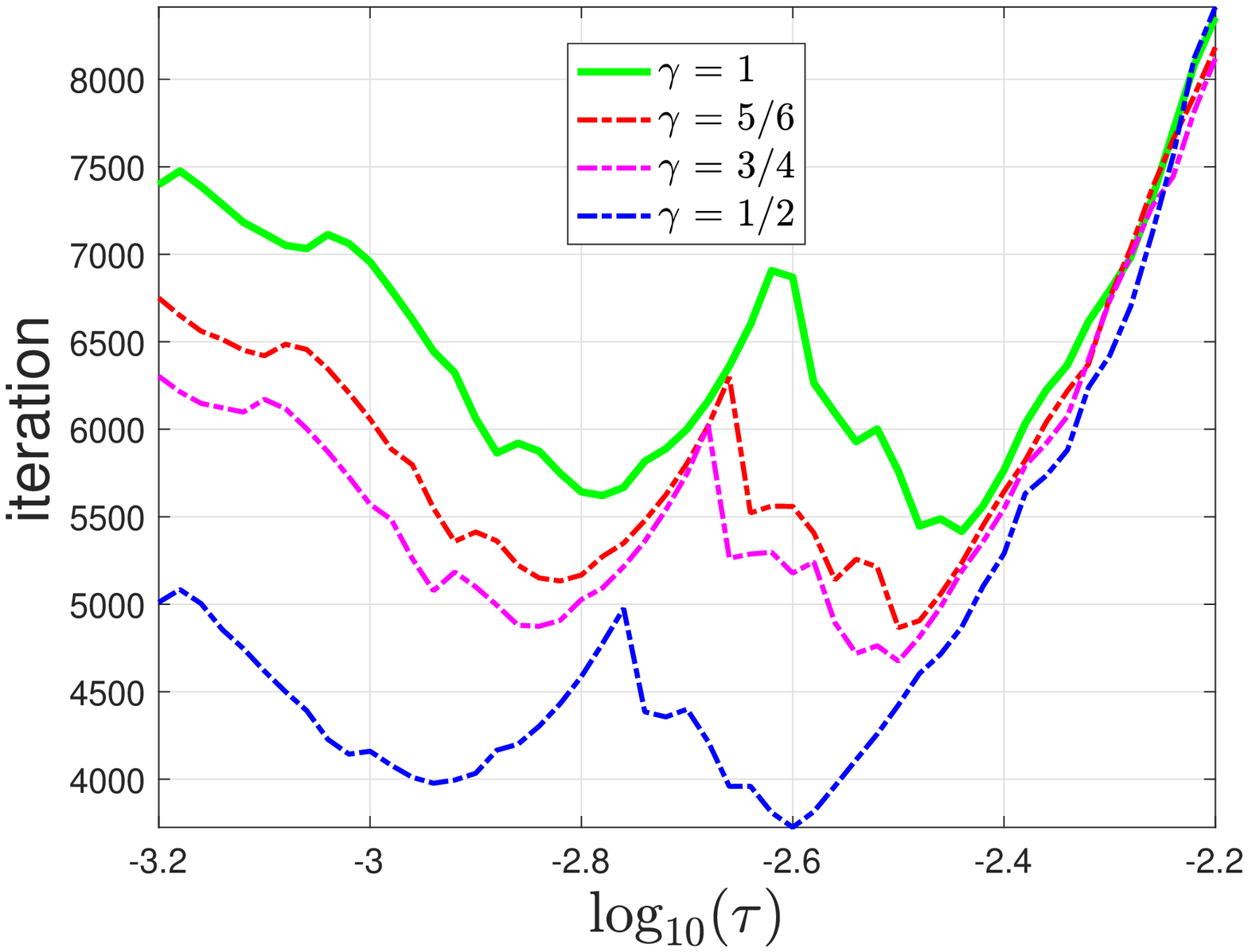}}~
\subfloat[$\nu = 15$]{\includegraphics[width=2.8cm,height=2.33cm]{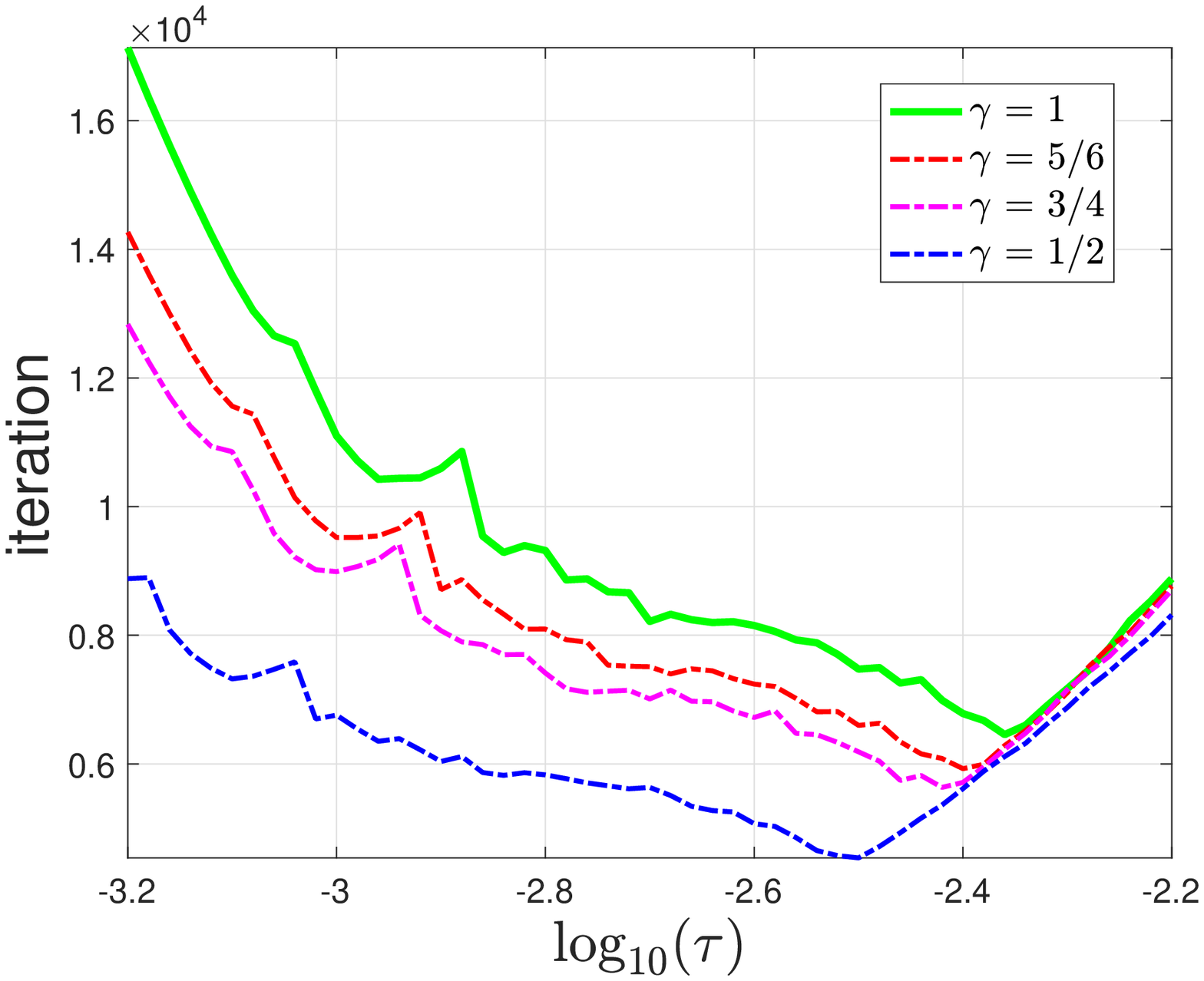}}\\
\subfloat[$\nu = 18$]{\includegraphics[width=2.8cm,height=2.33cm]{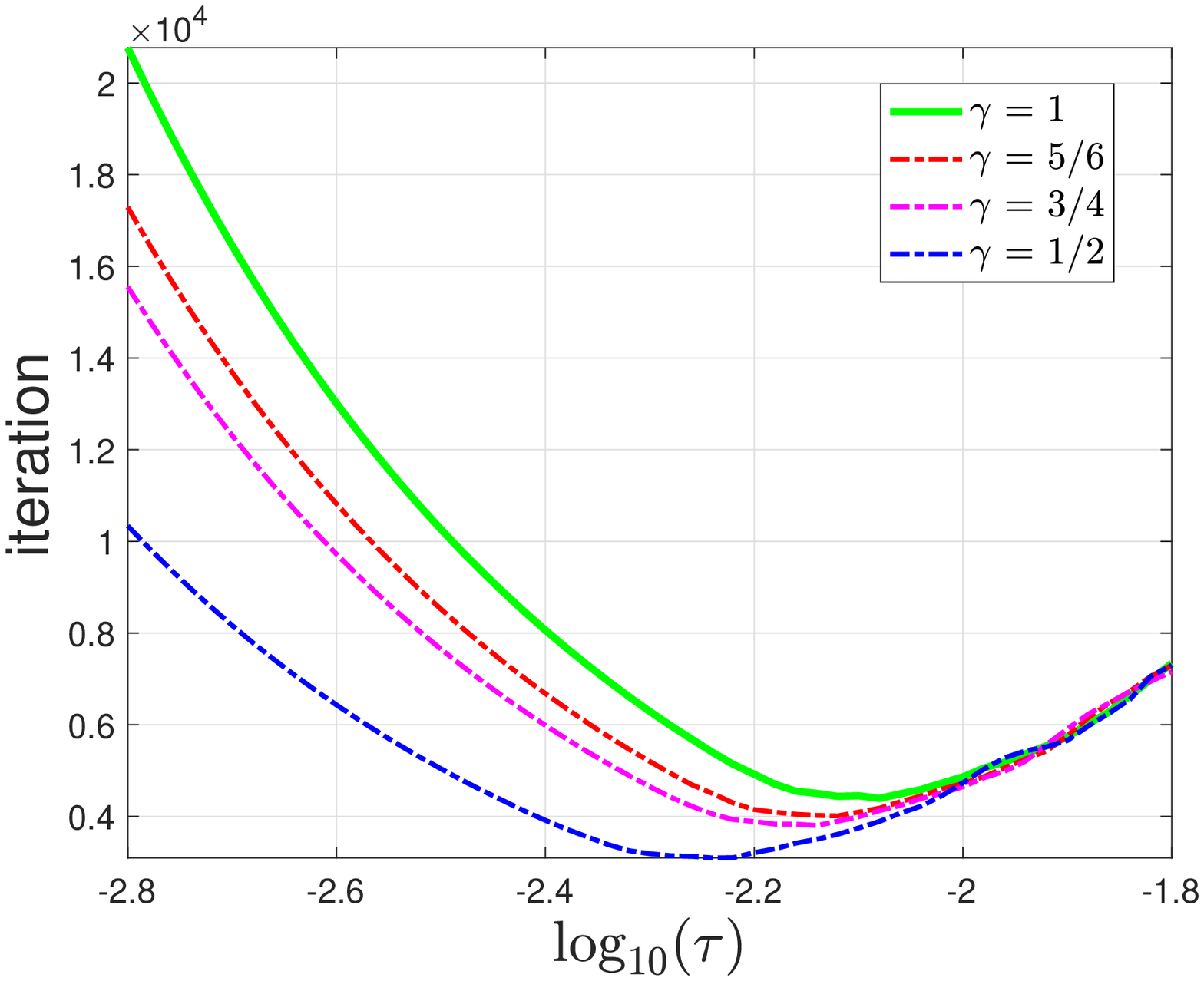}}~
\subfloat[$\nu = 24$]{\includegraphics[width=2.8cm,height=2.33cm]{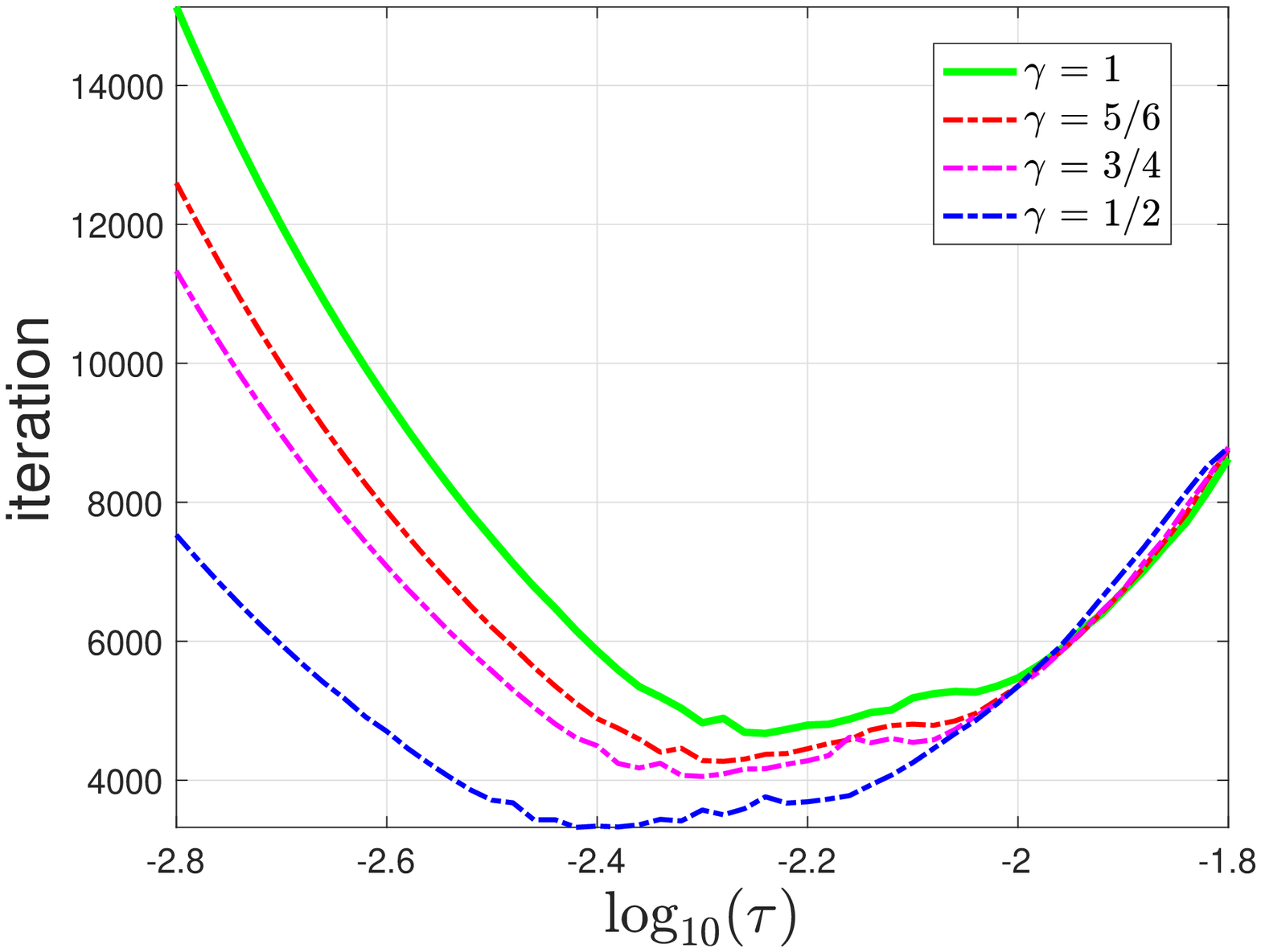}}~
\subfloat[$\nu = 30$]{\includegraphics[width=2.8cm,height=2.33cm]{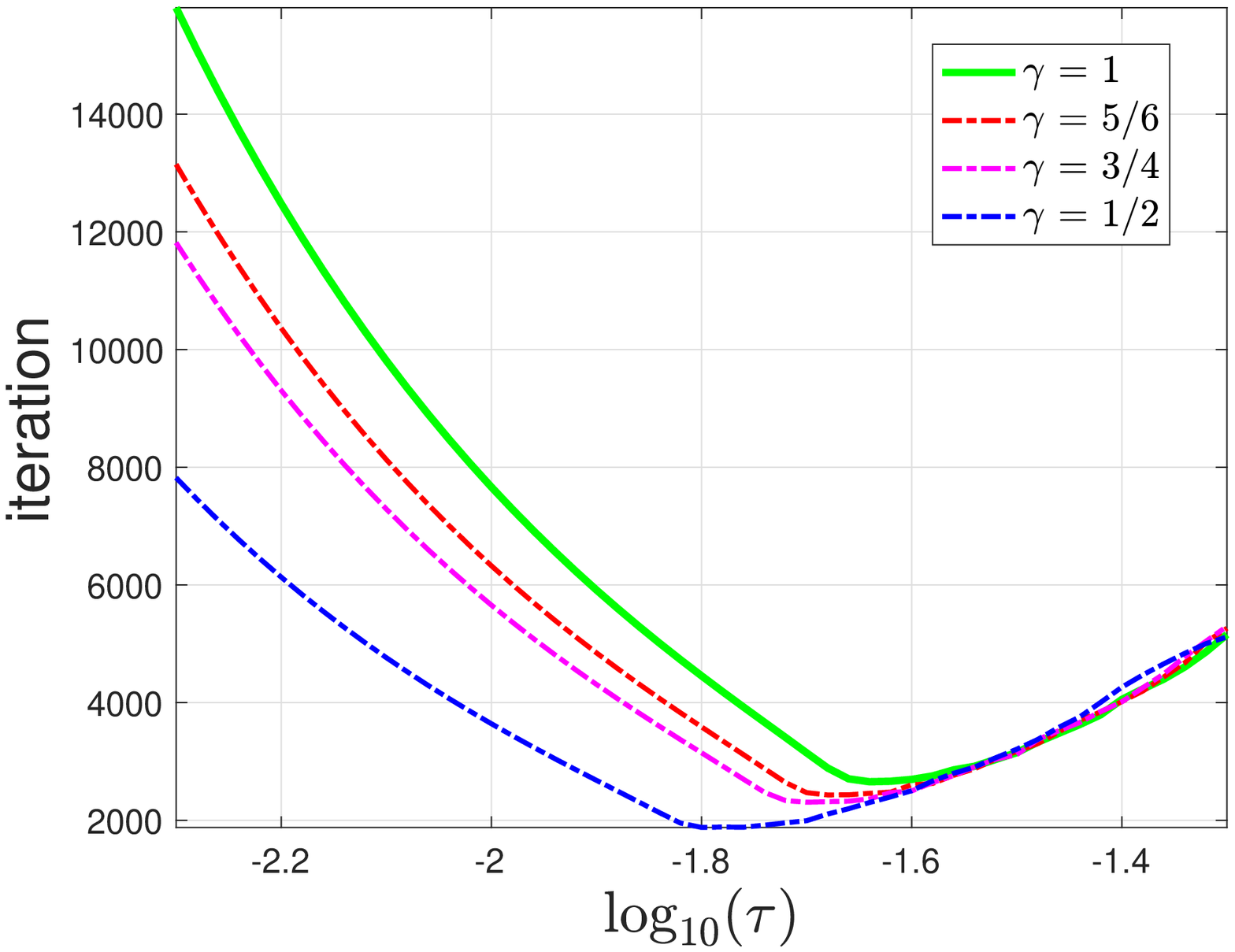}}~
\subfloat[$\nu = 36$]{\includegraphics[width=2.8cm,height=2.33cm]{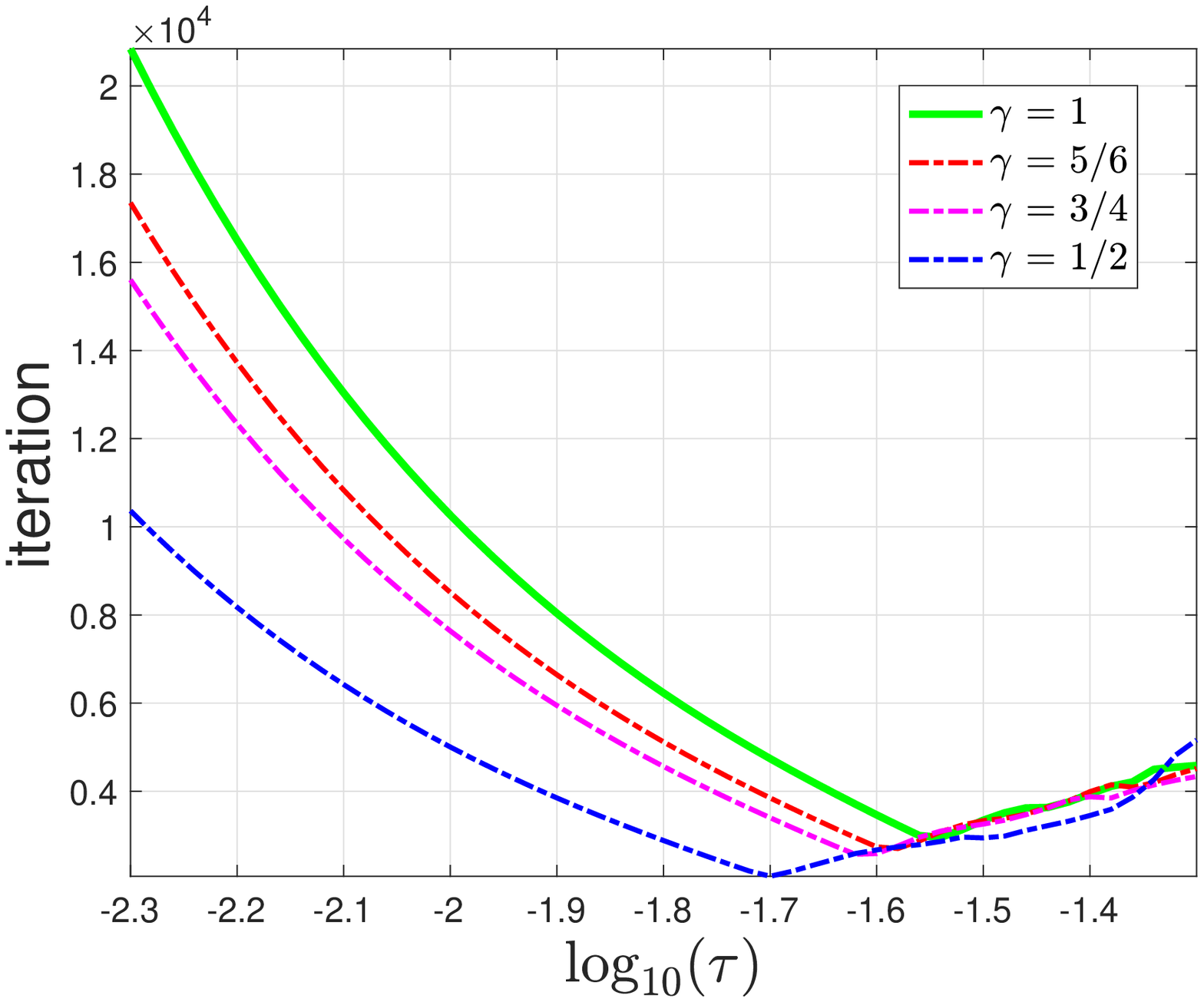}}\\
\subfloat[$\nu = 6$]{\includegraphics[width=2.8cm,height=2.33cm]{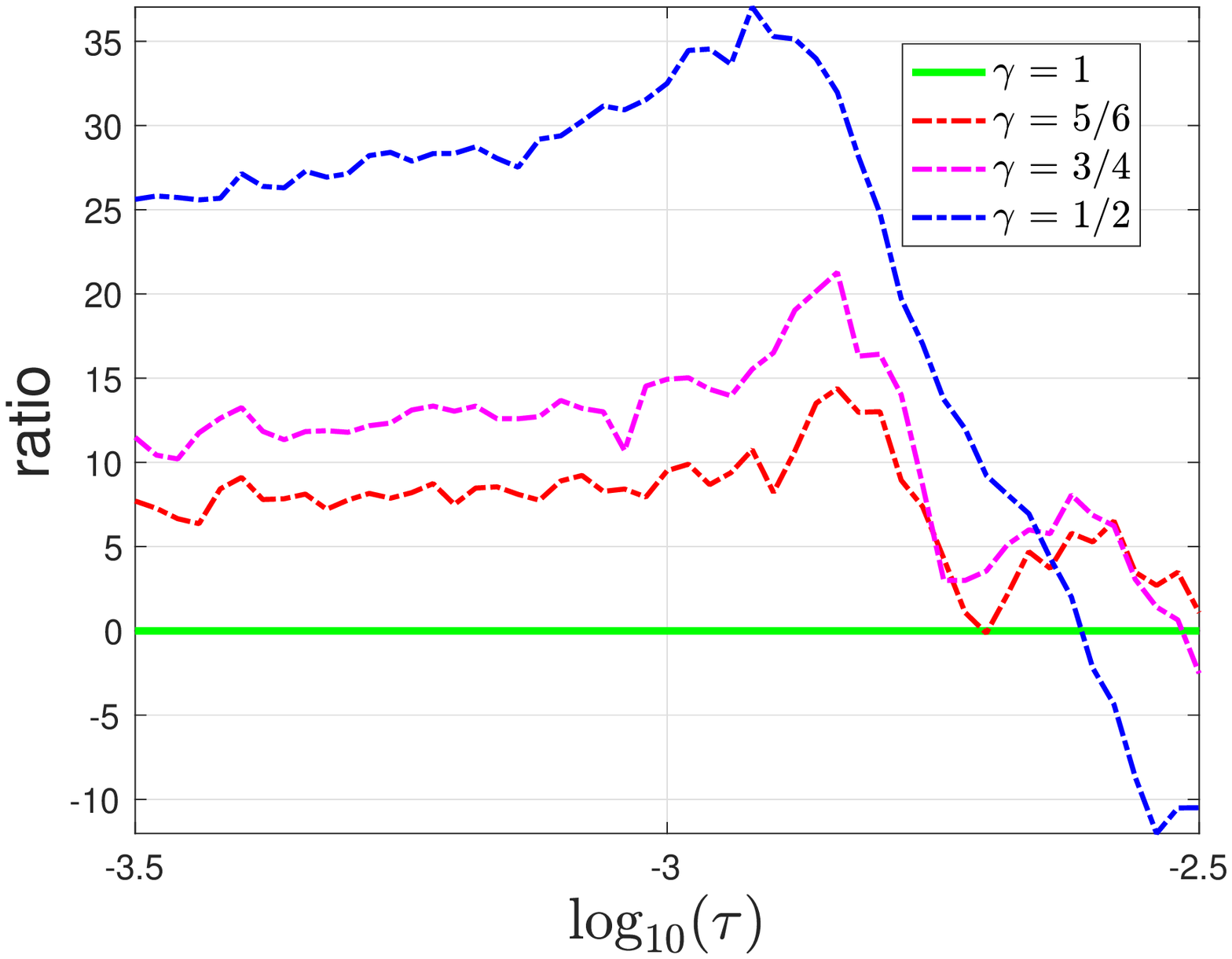}}~
\subfloat[$\nu = 9$]{\includegraphics[width=2.8cm,height=2.33cm]{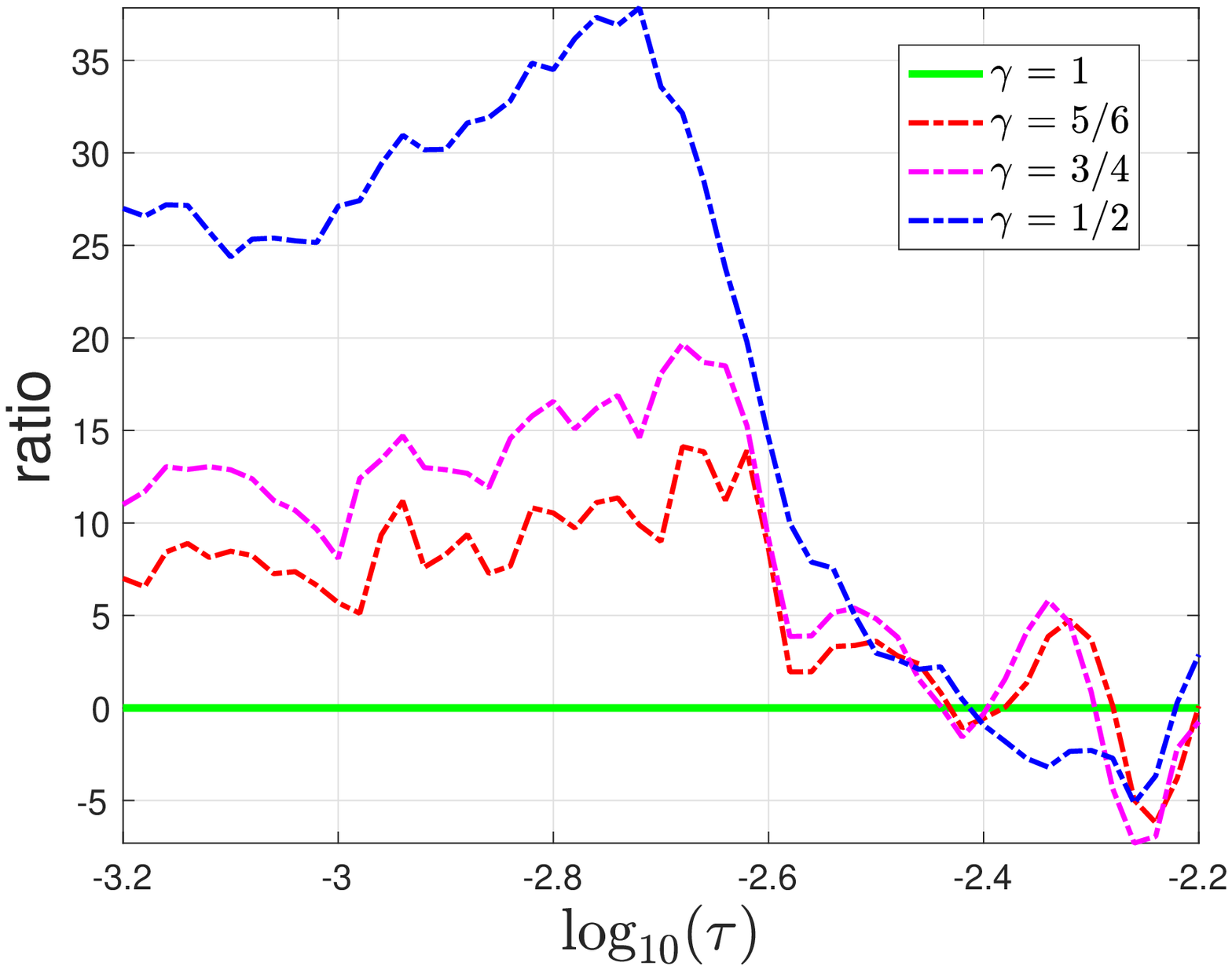}}~
\subfloat[$\nu = 12$]{\includegraphics[width=2.8cm,height=2.33cm]{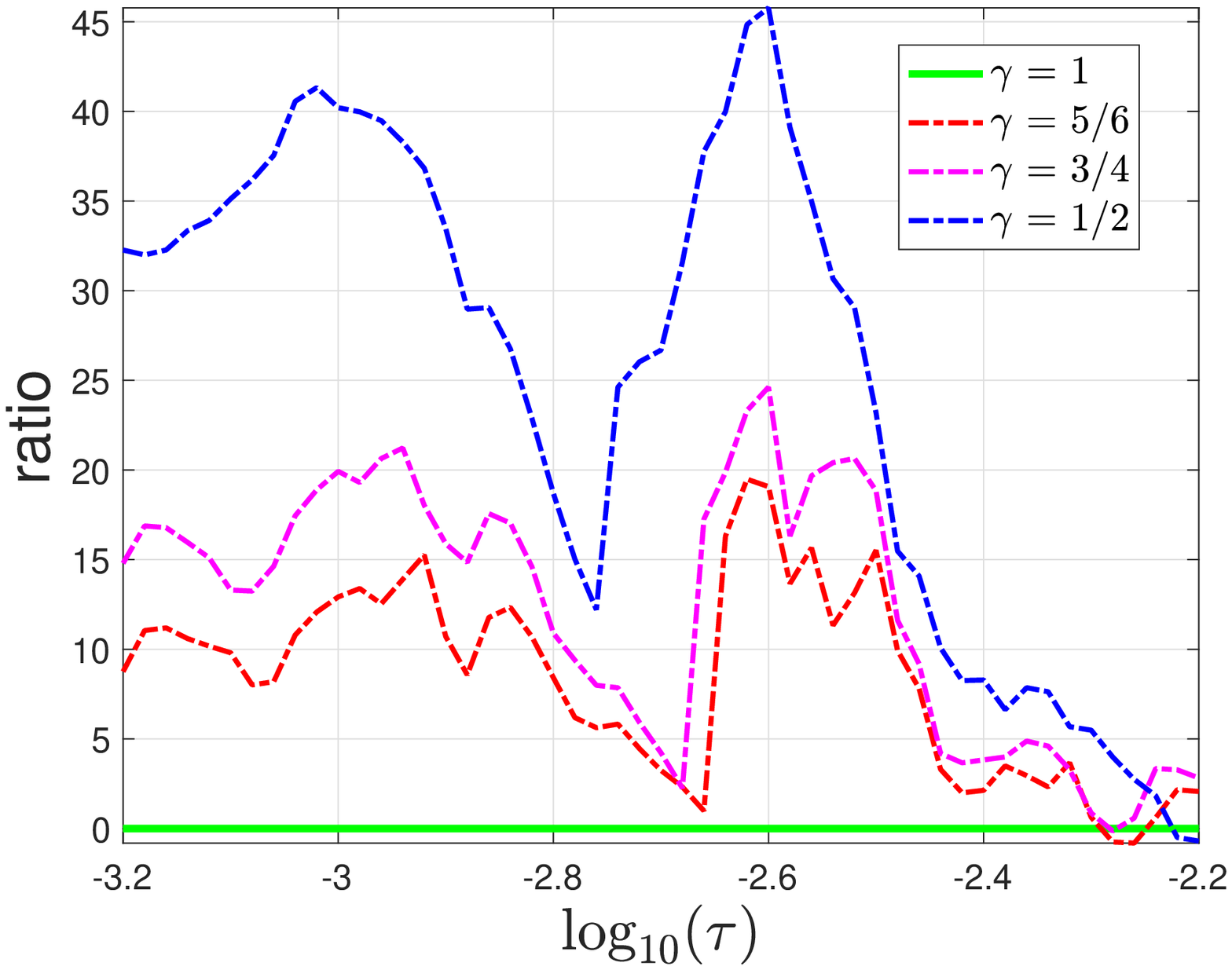}}~
\subfloat[$\nu = 15$]{\includegraphics[width=2.8cm,height=2.33cm]{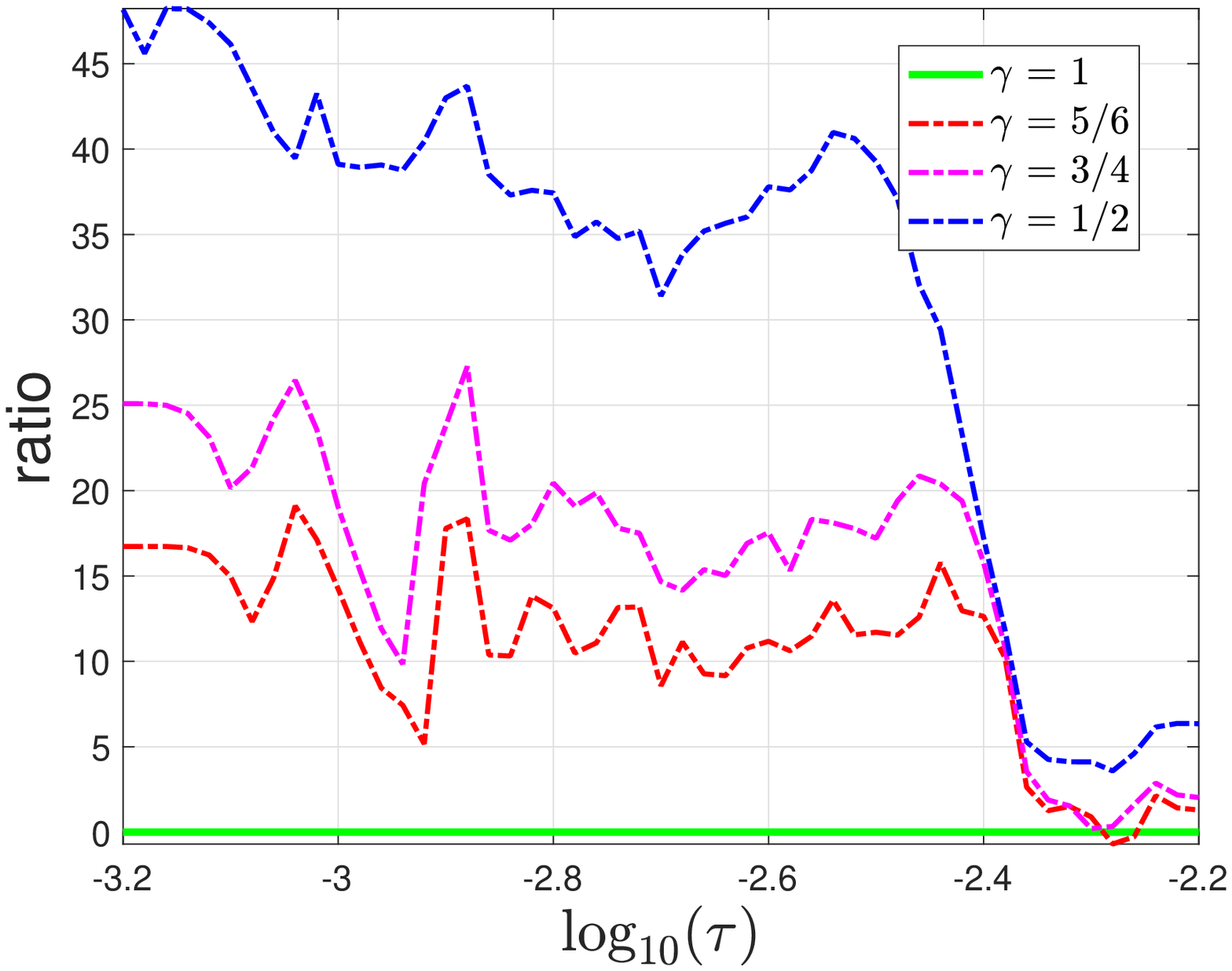}}\\
\subfloat[$\nu = 18$]{\includegraphics[width=2.8cm,height=2.33cm]{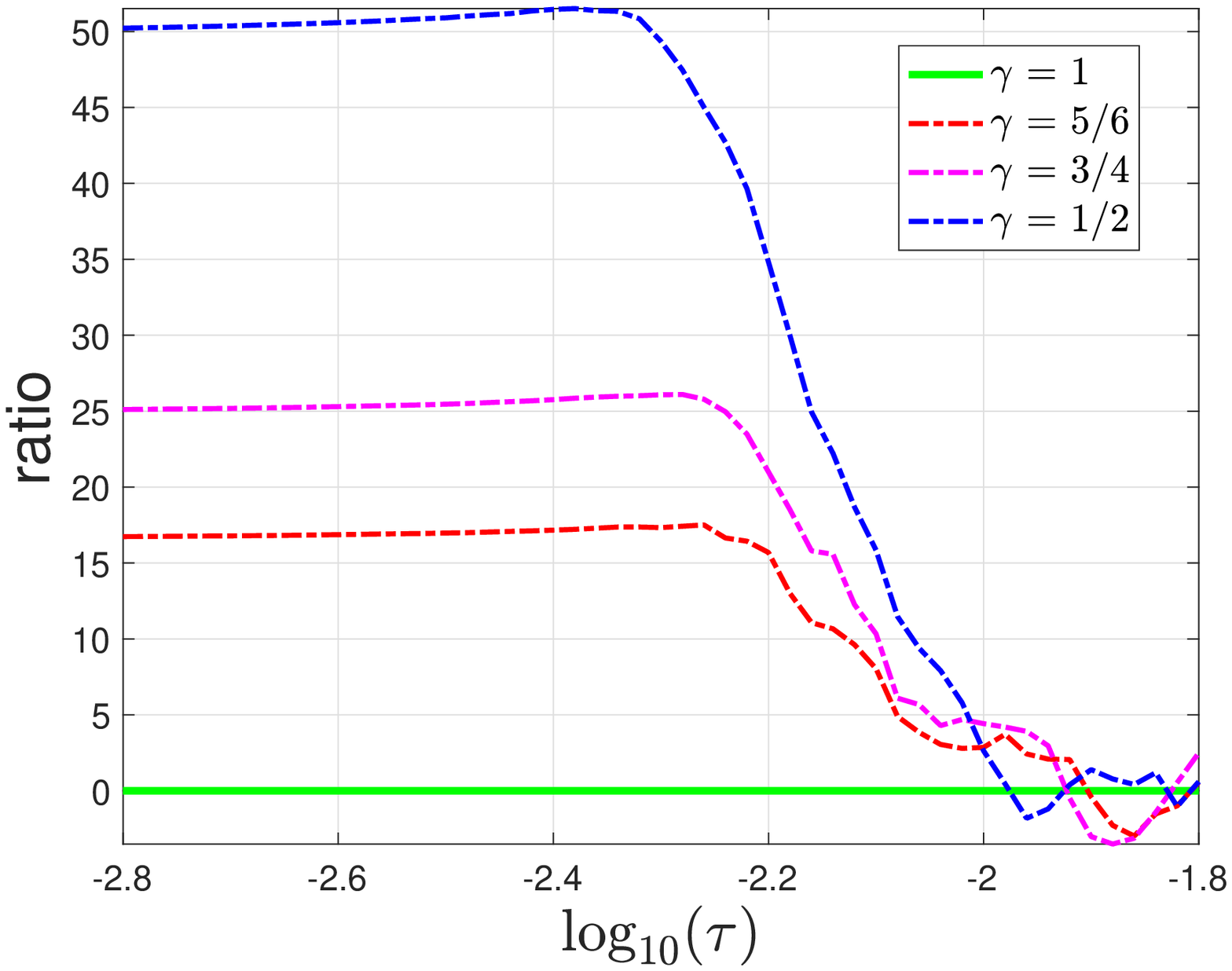}}~
\subfloat[$\nu = 24$]{\includegraphics[width=2.8cm,height=2.33cm]{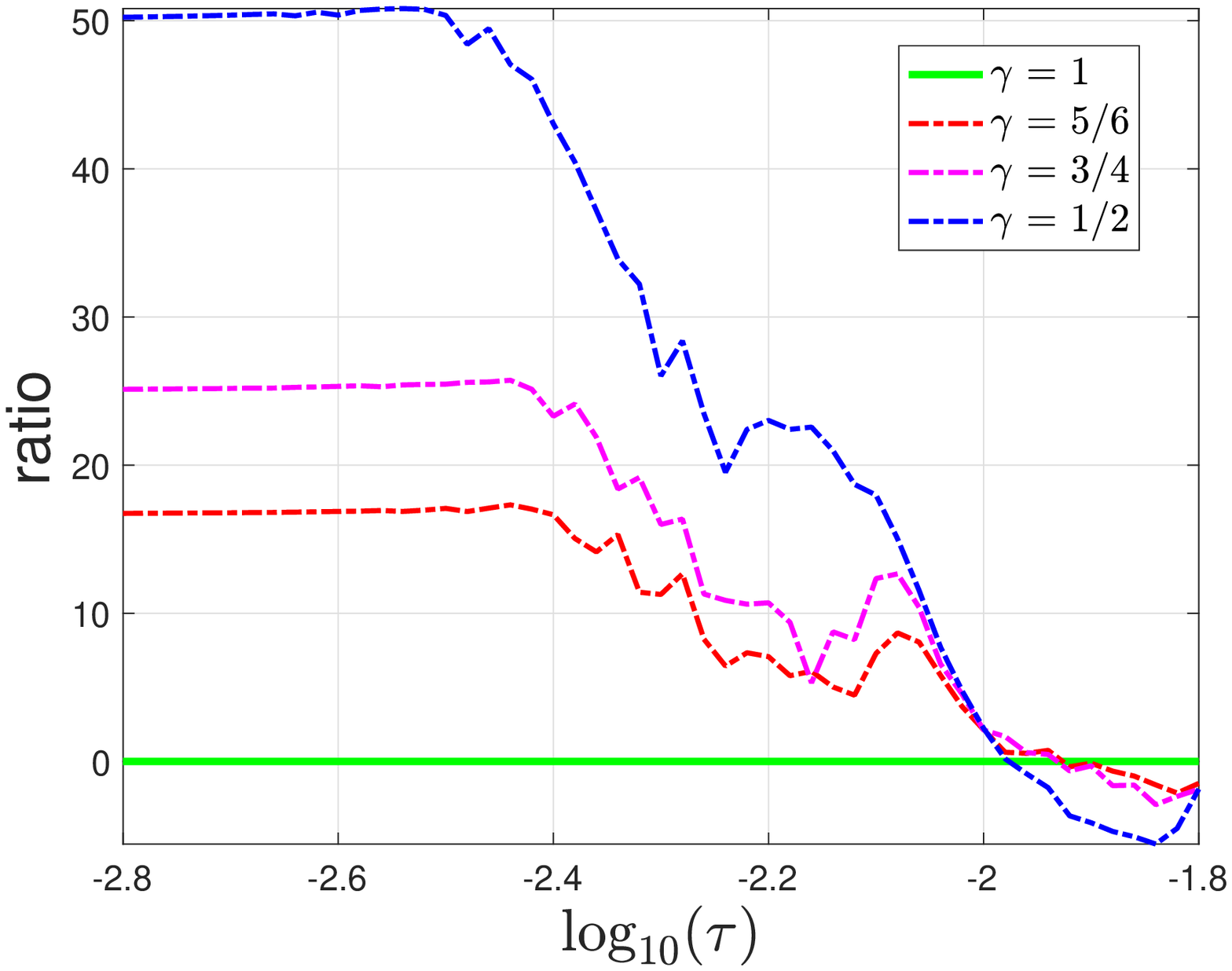}}~
\subfloat[$\nu = 30$]{\includegraphics[width=2.8cm,height=2.33cm]{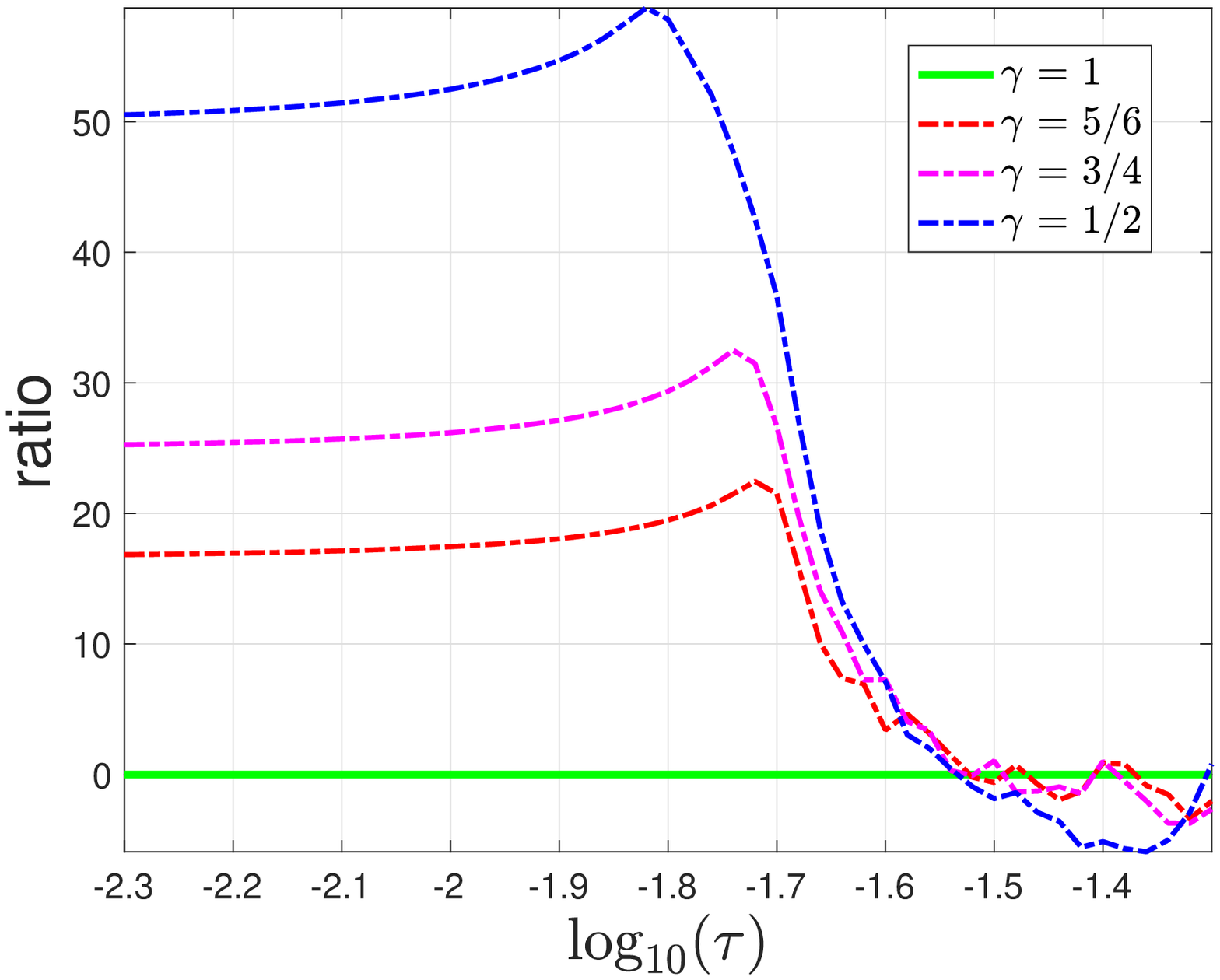}}~
\subfloat[$\nu = 36$]{\includegraphics[width=2.8cm,height=2.33cm]{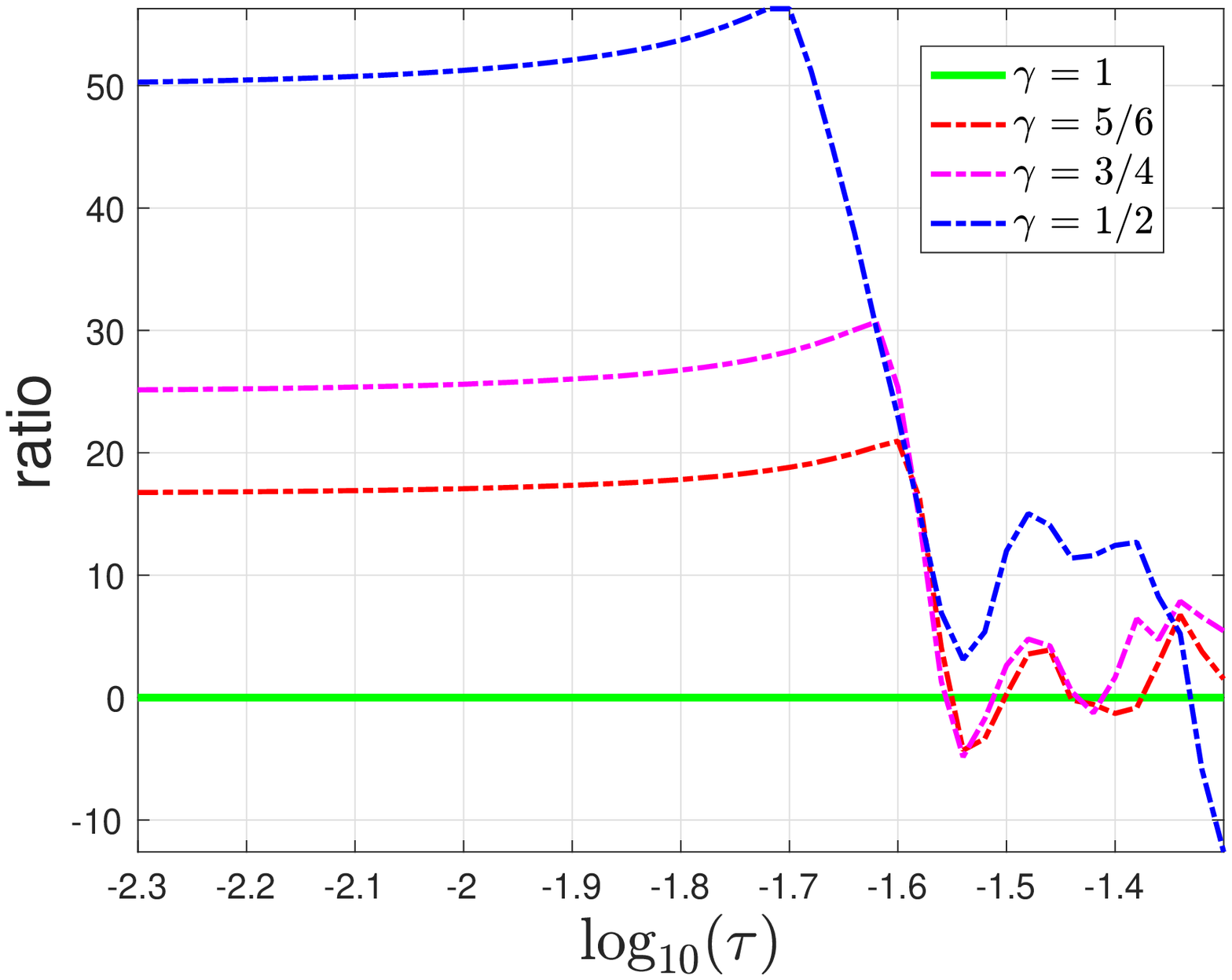}}\\
\caption{Comparison  of PrePDHG with $\gamma = \{1, 5/6, 3/4, 1/2\}$ for CT reconstruction problem \eqref{prob:CT}. Note that PrePDHG with $\gamma = 1$ is exactly iPrePDHG in \cite{liu2021acceleration}.}
\label{CT:figures:ratio:iter}
\end{figure}

For each fixed $\nu$,  we test a series of  $\tau \in 10^a$ with $a = [-a_1:0.02:-a_1 + 1]$. The parameter
 $a_1$   is 3.5 for $\nu = 6$,  $3.2$ for $\nu = 9$, 12 or 15, 2.8 for $\nu  = 18$ or 24, and 2.3 for $\nu = 30$ or 36. The results are presented in  Figure \ref{CT:figures:ratio:iter}.   In these figures, the term ratio is computed according to \eqref{equ:ratio}, wherein ``$\underline{\mathrm{iter}}$'' is the iteration number corresponding to $\gamma = 1$, namely, iPrePDHG.  From these figures, we can see that taking smaller  $\gamma$ (meaning the larger stepsize in updating the primal variable $x$, see \eqref{vmpdhg:ct:x}) can always speed up the performance of PrePDHG. More specifically, the saved ratios of taking  $\gamma = 3/4$, the theoretical lower bound, are about 13\% for $\nu =6,9,12$ and about 25\% for $\nu \geq 15$ for a large portion of $\tau$.  On the other hand,  although taking $\gamma = 1/2$ has no convergence guarantee since 1/2 is smaller than the theoretical lower bound of $\gamma$, it always has the best performance for a large portion of $\tau$. The corresponding saved ratios are more than 25\% for  $\nu =6,9,12$, more than 40\% for $\nu = 15$, and even  50\% for $\nu \geq 18$ for a large portion of $\tau$.

The numerical results corresponding to the best $\tau$, denoted by $\tau_{\mathrm{best}}$,   for each instance are reported in Table \ref{table:CT}.   From this table, we can see that, compared with iPrePDHG,  PrePDHG with smaller $\gamma$ is always faster. For $\gamma = 5/6$, it can save about $8\%$ of iteration number; for $\gamma = 3/4$, it can save about 13\% of iteration number.   More interesting,  PrePDHG with $\gamma = 1/2$ can save about 30\% of iteration number.  This tells that reducing the parameter $\gamma$  (in a reasonable range) in PrePDHG for the CT reconstruction problem can still bring some benefits even though the so-called best stepsize $\tau$ is chosen. However,  it should be emphasized again that selecting the best stepsize $\tau$ is very hard in practice.

\begin{table}[!htbp]
\small
\setlength{\tabcolsep}{2.5pt}
\centering
\caption{Performance of PrePDHG and iPre-PDHG with best $\tau$ for CT reconstruction problem \eqref{prob:CT}.  In the table, ``a'' stands for  iPrePDHG, ``b'', ``c''  and ``d'' stands for PrePDHG with $\gamma = 5/6$, $\gamma = 3/4$ and $\gamma = 1/2$, respectively.}~\\~\\
\label{table:CT}
\begin{tabular}{@{}crrrrrrrrrrrrrrr@{}}
\toprule
&\multicolumn{4}{c}{$\log_{10}(\tau_{\mathrm{best}})$}  & \multicolumn{4}{c}{time}  & \multicolumn{4}{c}{iter} & \multicolumn{3}{c}{ratio \%}  \\
\cmidrule(lr){2-5} \cmidrule(lr){6-9} \cmidrule(lr){10-13}  \cmidrule(l){14-16}
$\theta$ & \multicolumn{1}{c}{a}  & \multicolumn{1}{c}{b}  & \multicolumn{1}{c}{c}   & \multicolumn{1}{c}{d}   &\multicolumn{1}{c}{a}   & \multicolumn{1}{c}{b}  & \multicolumn{1}{c}{c}    & \multicolumn{1}{c}{d} & \multicolumn{1}{c}{a}   & \multicolumn{1}{c}{b}   & \multicolumn{1}{c}{c}    & \multicolumn{1}{c}{d}   & \multicolumn{1}{c}{b}  & \multicolumn{1}{c}{c}   & \multicolumn{1}{c}{d} \\
\midrule

6& -2.74& -2.80& -2.84& -2.92& 96.8& 89.5& 84.9& 70.3& 6441& 5990& 5674& 4690&  7.0& 11.9& 27.2\\
9& -2.58& -2.62& -2.64& -2.72& 76.0& 69.4& 66.1& 53.6& 6544& 5932& 5677& 4613&  9.4& 13.2& 29.5\\
12& -2.44& -2.50& -2.50& -2.60& 54.5& 49.2& 47.4& 37.6& 5416& 4866& 4675& 3725& 10.2& 13.7& 31.2\\
15& -2.36& -2.40& -2.42& -2.50& 59.3& 54.5& 51.7& 41.6& 6456& 5926& 5635& 4539&  8.2& 12.7& 29.7\\
18& -2.08& -2.12& -2.14& -2.24& 37.8& 34.5& 32.7& 26.4& 4393& 4010& 3800& 3094&  8.7& 13.5& 29.6\\
24& -2.24& -2.28& -2.30& -2.42& 36.5& 33.7& 32.0& 26.2& 4673& 4271& 4054& 3321&  8.6& 13.2& 28.9\\
30& -1.64& -1.68& -1.70& -1.80& 19.5& 17.8& 17.0& 13.8& 2655& 2431& 2307& 1879&  8.4& 13.1& 29.2\\
36& -1.54& -1.58& -1.62& -1.70& 21.0& 19.3& 18.3& 15.3& 2954& 2703& 2571& 2073&  8.5& 13.0& 29.8\\

\bottomrule
\end{tabular}

\end{table}

\section{Conclusions}\label{S6}
In this paper, we investigate the PrePDHG algorithm from the iPADMM point of view. We establish the equivalence between PrePDHG and iPADMM, based on which we can obtain a tight convergence condition for PrePDHG.  Some counter-examples are given to show the tightness of the convergence condition we established for PrePDHG.  This result subsumes the latest convergence condition for the original PDHG and derives an interesting by-product, namely, the dual stepsize of the BALM can be extended to $4/3$ other than $1$. Besides, based on the equivalence between PrePDHG and iPADMM, we also establish the global convergence and  the ergodic and non-ergodic sublinear convergence rate of PrePDHG.  In order to make PrePDHG practical, we also discuss the various choices of the proximal terms. A variety of numerical results on the matrix game, projection onto the Birkhoff polytope,   earth mover's distance, and  CT reconstruction show the efficiency of PrePDHG with improved convergence conditions.  Considering that the subproblems in PrePDHG are still hard to solve in some cases, it would be interesting to investigate the inexact version of PrePDHG in future work.

\section*{Data availability statements}
The authors confirm that all data generated or analyzed during this study are included in the paper.
The data matrices $\rho^0$, $\rho^1$, and  $m^{\mathrm{cvx}}$ in Section \ref{section:EMD} are from \cite{liu2021acceleration}
and  downloaded at \url{https://github.com/xuyunbei/Inexact-preconditioning}.

\bibliography{pdhg}

\begin{thebibliography}{10}
\providecommand{\url}[1]{{#1}}
\providecommand{\urlprefix}{URL }
\expandafter\ifx\csname urlstyle\endcsname\relax
  \providecommand{\doi}[1]{DOI~\discretionary{}{}{}#1}\else
  \providecommand{\doi}{DOI~\discretionary{}{}{}\begingroup
  \urlstyle{rm}\Url}\fi

\bibitem{bai2021new}
Bai, J.: A new insight on augmented lagrangian method and its extensions.
\newblock arXiv preprint arXiv:2108.11125  (2021)

\bibitem{becker2019quasi}
Becker, S., Fadili, J., Ochs, P.: On quasi-{N}ewton forward-backward splitting:
  Proximal calculus and convergence.
\newblock SIAM Journal on Optimization \textbf{29}(4), 2445--2481 (2019)

\bibitem{briceno2021split}
Briceno-Arias, L.M., Rold{\'a}n, F.: Split-{D}ouglas--{R}achford algorithm for
  composite monotone inclusions and split-{ADMM}.
\newblock SIAM Journal on Optimization \textbf{31}(4), 2987--3013 (2021)

\bibitem{cai2010singular}
Cai, J.F., Cand{\`e}s, E.J., Shen, Z.: A singular value thresholding algorithm
  for matrix completion.
\newblock SIAM Journal on Optimization \textbf{20}(4), 1956--1982 (2010)

\bibitem{cai2022developments}
Cai, X., Guo, K., Jiang, F., Wang, K., Wu, Z., Han, D.: The developments of
  proximal point algorithms.
\newblock Journal of the Operations Research Society of China \textbf{10},
  197--239 (2022)

\bibitem{cai2013improved}
Cai, X., Han, D., Xu, L.: An improved first-order primal-dual algorithm with a
  new correction step.
\newblock Journal of Global Optimization \textbf{57}(4), 1419--1428 (2013)

\bibitem{chambolle2011first}
Chambolle, A., Pock, T.: A first-order primal-dual algorithm for convex
  problems with applications to imaging.
\newblock Journal of Mathematical Imaging and Vision \textbf{40}(1), 120--145
  (2011)

\bibitem{chambolle2016ergodic}
Chambolle, A., Pock, T.: On the ergodic convergence rates of a first-order
  primal--dual algorithm.
\newblock Mathematical Programming \textbf{159}(1), 253--287 (2016)

\bibitem{chang2022golden}
Chang, X.K., Yang, J., Zhang, H.: Golden ratio primal-dual algorithm with
  linesearch.
\newblock SIAM Journal on Optimization \textbf{32}(3), 1584--1613 (2022)

\bibitem{chen2021equivalence}
Chen, L., Li, X., Sun, D., Toh, K.C.: On the equivalence of inexact proximal
  {ALM} and {ADMM} for a class of convex composite programming.
\newblock Mathematical Programming \textbf{185}(1), 111--161 (2021)

\bibitem{chen2019unified}
Chen, L., Sun, D., Toh, K.C., Zhang, N.: A unified algorithmic framework of
  symmetric {G}auss-{S}eidel decomposition based proximal {ADMMs} for convex
  composite programming.
\newblock Journal of Computational Mathematics \textbf{37}(6), 739--757 (2019)

\bibitem{combettes2013moreau}
Combettes, P.L., Reyes, N.N.: Moreau’s decomposition in {B}anach spaces.
\newblock Mathematical Programming \textbf{139}(1), 103--114 (2013)

\bibitem{condat2013primal}
Condat, L.: A primal--dual splitting method for convex optimization involving
  {L}ipschitzian, proximable and linear composite terms.
\newblock Journal of Optimization Theory and Applications \textbf{158}(2),
  460--479 (2013)

\bibitem{esser2010general}
Esser, E., Zhang, X., Chan, T.F.: A general framework for a class of first
  order primal-dual algorithms for convex optimization in imaging science.
\newblock SIAM Journal on Imaging Sciences \textbf{3}(4), 1015--1046 (2010)

\bibitem{gabay1976dual}
Gabay, D., Mercier, B.: A dual algorithm for the solution of nonlinear
  variational problems via finite element approximation.
\newblock Computers \& Mathematics with Applications \textbf{2}(1), 17--40
  (1976)

\bibitem{glowinski1975approximation}
Glowinski, R., Marroco, A.: Sur l'approximation, par {\'e}l{\'e}ments finis
  d'ordre un, et la r{\'e}solution, par p{\'e}nalisation-dualit{\'e} d'une
  classe de probl{\`e}mes de dirichlet non lin{\'e}aires.
\newblock ESAIM: Mathematical Modelling and Numerical Analysis-Mod{\'e}lisation
  Math{\'e}matique et Analyse Num{\'e}rique \textbf{9}(R2), 41--76 (1975)

\bibitem{gu2015indefinite}
Gu, Y., Jiang, B., Han, D.: An indefinite-proximal-based strictly contractive
  {Peaceman-Rachford} splitting method.
\newblock arXiv preprint arXiv:1506.02221  (2022)

\bibitem{han2022survey}
Han, D.: A survey on some recent developments of alternating direction method
  of multipliers.
\newblock Journal of the Operations Research Society of China \textbf{10}(1),
  1--52 (2022)

\bibitem{hansen2018air}
Hansen, P.C., J{\o}rgensen, J.S.: {AIR} tools {II}: algebraic iterative
  reconstruction methods, improved implementation.
\newblock Numerical Algorithms \textbf{79}(1), 107--137 (2018)

\bibitem{haupt2008compressed}
Haupt, J., Bajwa, W.U., Rabbat, M., Nowak, R.: Compressed sensing for networked
  data.
\newblock IEEE Signal Processing Magazine \textbf{25}(2), 92--101 (2008)

\bibitem{he2022generalized}
He, B., Ma, F., Xu, S., Yuan, X.: A generalized primal-dual algorithm with
  improved convergence condition for saddle point problems.
\newblock SIAM Journal on Imaging Sciences \textbf{15}(3), 1157--1183 (2022)

\bibitem{he2020optimally}
He, B., Ma, F., Yuan, X.: Optimally linearizing the alternating direction
  method of multipliers for convex programming.
\newblock Computational Optimization and Applications \textbf{75}(2), 361--388
  (2020)

\bibitem{he2014convergence}
He, B., You, Y., Yuan, X.: On the convergence of primal-dual hybrid gradient
  algorithm.
\newblock SIAM Journal on Imaging Sciences \textbf{7}(4), 2526--2537 (2014)

\bibitem{he2012convergence}
He, B., Yuan, X.: Convergence analysis of primal-dual algorithms for a
  saddle-point problem: from contraction perspective.
\newblock SIAM Journal on Imaging Sciences \textbf{5}(1), 119--149 (2012)

\bibitem{he2021balanced}
He, B., Yuan, X.: Balanced augmented {L}agrangian method for convex
  programming.
\newblock arXiv preprint arXiv:2108.08554  (2021)

\bibitem{jiang2016lp}
Jiang, B., Liu, Y.F., Wen, Z.: ${L}_p$-norm regularization algorithms for
  optimization over permutation matrices.
\newblock SIAM Journal on Optimization \textbf{26}(4), 2284--2313 (2016)

\bibitem{jiang2021indefinite}
Jiang, F., Cai, X., Han, D.: The indefinite proximal point algorithms for
  maximal monotone operators.
\newblock Optimization \textbf{70}(8), 1759--1790 (2021)

\bibitem{jiang2021approximate}
Jiang, F., Cai, X., Wu, Z., Han, D.: Approximate first-order primal-dual
  algorithms for saddle point problems.
\newblock Mathematics of Computation \textbf{90}(329), 1227--1262 (2021)

\bibitem{jiang2021first}
Jiang, F., Wu, Z., Cai, X., Zhang, H.: A first-order inexact primal-dual
  algorithm for a class of convex-concave saddle point problems.
\newblock Numerical Algorithms \textbf{88}(3), 1109--1136 (2021)

\bibitem{jiang2022solving}
Jiang, F., Zhang, Z., He, H.: Solving saddle point problems: a landscape of
  primal-dual algorithm with larger stepsizes.
\newblock Journal of Global Optimization,
  https://doi.org/10.1007/s10898-022-01233-0 pp. 1--26 (2022)

\bibitem{jiang2022bregman}
Jiang, X., Vandenberghe, L.: Bregman three-operator splitting methods.
\newblock arXiv preprint arXiv:2203.00252  (2022)

\bibitem{li2016majorized}
Li, M., Sun, D., Toh, K.C.: A majorized {ADMM} with indefinite proximal terms
  for linearly constrained convex composite optimization.
\newblock SIAM Journal on Optimization \textbf{26}(2), 922--950 (2016)

\bibitem{li2018parallel}
Li, W., Ryu, E.K., Osher, S., Yin, W., Gangbo, W.: A parallel method for earth
  mover’s distance.
\newblock Journal of Scientific Computing \textbf{75}(1), 182--197 (2018)

\bibitem{li2019block}
Li, X., Sun, D., Toh, K.C.: A block symmetric {G}auss--{S}eidel decomposition
  theorem for convex composite quadratic programming and its applications.
\newblock Mathematical Programming \textbf{175}(1), 395--418 (2019)

\bibitem{li2022improved}
Li, Y., Yan, M.: On the improved conditions for some primal-dual algorithms.
\newblock arXiv preprint arXiv:2201.00139  (2022)

\bibitem{liu2021acceleration}
Liu, Y., Xu, Y., Yin, W.: Acceleration of primal--dual methods by
  preconditioning and simple subproblem procedures.
\newblock Journal of Scientific Computing \textbf{86}(2), 1--34 (2021)

\bibitem{ma2021majorized}
Ma, Y., Li, T., Song, Y., Cai, X.: Majorized {iPADMM} for nonseparable convex
  minimization models with quadratic coupling terms.
\newblock Asia-Pacific Journal of Operational Research,
  https://doi.org/10.1142/S0217595922400024  (2021)

\bibitem{malitsky2018first}
Malitsky, Y., Pock, T.: A first-order primal-dual algorithm with linesearch.
\newblock SIAM Journal on Optimization \textbf{28}(1), 411--432 (2018)

\bibitem{o2020equivalence}
O’Connor, D., Vandenberghe, L.: On the equivalence of the primal-dual hybrid
  gradient method and {D}ouglas--{R}achford splitting.
\newblock Mathematical Programming \textbf{179}(1), 85--108 (2020)

\bibitem{pock2011diagonal}
Pock, T., Chambolle, A.: Diagonal preconditioning for first order primal-dual
  algorithms in convex optimization.
\newblock In: 2011 International Conference on Computer Vision, pp. 1762--1769.
  IEEE (2011)

\bibitem{pock2009algorithm}
Pock, T., Cremers, D., Bischof, H., Chambolle, A.: An algorithm for minimizing
  the {M}umford-{S}hah functional.
\newblock In: 2009 IEEE 12th International Conference on Computer Vision, pp.
  1133--1140. IEEE (2009)

\bibitem{rasch2020inexact}
Rasch, J., Chambolle, A.: Inexact first-order primal--dual algorithms.
\newblock Computational Optimization and Applications \textbf{76}(2), 381--430
  (2020)

\bibitem{rockafellar2015convex}
Rockafellar, R.T.: Convex analysis.
\newblock Princeton University Press (2015)

\bibitem{rudin1992nonlinear}
Rudin, L.I., Osher, S., Fatemi, E.: Nonlinear total variation based noise
  removal algorithms.
\newblock Physica D: Nonlinear Phenomena \textbf{60}(1-4), 259--268 (1992)

\bibitem{sidky2012convex}
Sidky, E.Y., J{\o}rgensen, J.H., Pan, X.: Convex optimization problem
  prototyping for image reconstruction in computed tomography with the
  {C}hambolle--{P}ock algorithm.
\newblock Physics in Medicine \& Biology \textbf{57}(10), 3065--3091 (2012)

\bibitem{valkonen2014primal}
Valkonen, T.: A primal-dual hybrid gradient method for nonlinear operators with
  applications to {MRI}.
\newblock Inverse Problems 055012 \textbf{30}(5) (2014)

\bibitem{yang2010review}
Yang, A.Y., Sastry, S.S., Ganesh, A., Ma, Y.: Fast $l_1$-minimization
  algorithms and an application in robust face recognition: {A} review.
\newblock proceedings of 2010 IEEE International Conference on Image Processing
  pp. 1849--1852

\bibitem{zhang2006schur}
Zhang, F.: The Schur complement and its applications, vol.~4.
\newblock Springer Science \& Business Media (2006)

\bibitem{zhang2020linearly}
Zhang, N., Wu, J., Zhang, L.: A linearly convergent majorized {ADMM} with
  indefinite proximal terms for convex composite programming and its
  applications.
\newblock Mathematics of Computation \textbf{89}(324), 1867--1894 (2020)

\bibitem{zhu2008efficient}
Zhu, M., Chan, T.: An efficient primal-dual hybrid gradient algorithm for total
  variation image restoration.
\newblock UCLA CAM Report \textbf{34}, 8--34 (2008)

\end{thebibliography}
\bibliographystyle{spmpsci} 

\end{document}